\documentclass{article}

\usepackage{paralist}
\title{Unstructured space-time finite element methods for optimal control 
of parabolic equations}
\author{Ulrich~Langer\footnote{Johann Radon Institute for Computational and
    Applied Mathematics, Austrian Academy of Sciences, Altenberger Stra{\ss}e 69, 4040 Linz, Austria, Email:
    ulrich.langer@ricam.oeaw.ac.at}, 
    \; Olaf~Steinbach\footnote{Institut f\"{u}r Angewandte Mathematik,
    Technische Universit\"{a}t Graz, Steyrergasse 30, 8010 Graz, Austria,
    Email: o.steinbach@tugraz.at},  
    \; Fredi~Tr\"{o}ltzsch\footnote{Institut f\"{u}r Mathematik, Technische
    Universit\"{a}t Berlin, Stra{\ss}e des 17. Juni 136, 10623 Berlin,
    Germany, Email: troeltzsch@math.tu-berlin.de}, 
    \; Huidong~Yang\footnote{Johann Radon Institute for Computational and Applied
    Mathematics, Austrian Academy of Sciences, Altenberger Stra{\ss}e 69, 4040 Linz, Austria, Email:
    huidong.yang@ricam.oeaw.ac.at} 
}  

\date{\today}

\usepackage{amsmath}
\usepackage{amsthm}
\usepackage{amsfonts}
\usepackage{booktabs}
\usepackage{graphicx}

\usepackage{multirow}
\usepackage{hyperref}

\usepackage[ulem=normalem,draft]{changes}

\newcommand{\be}{\begin{equation}}
\newcommand{\ee}{\end{equation}}

\newcommand{\ignore}[1]{}

\newtheorem{theorem}{Theorem}
\newtheorem{lemma}{Lemma}

\begin{document}

\maketitle

\begin{abstract}
This work presents and analyzes space-time finite element methods on fully
unstructured simplicial space-time meshes for the numerical solution of 
parabolic optimal control problems. Using Babu\v{s}ka's theorem, we show 
well-posedness of the first-order optimality systems for a typical model
problem with linear state equations, but without control constraints. This is
done for both continuous and discrete levels. Based on these results, we
derive discretization error estimates. Then we consider a semilinear parabolic
optimal control problem arising from the Schl\"ogl model. The associated
nonlinear optimality system is solved by Newton's method, where a linear
system, that is similar to the first-order optimality systems considered for
the linear model problems, has to be solved at each Newton step. We present
various numerical experiments including results for adaptive space-time finite
element discretizations based on residual-type error indicators. In the last
two examples, we also consider semilinear parabolic optimal control problems
with box constraints imposed on the control. 
\end{abstract} 


\begin{msc}
49J20,  35K20, 65M60, 65M50, 65M15, 65Y05
\end{msc}

\section{Introduction}
\label{sec:intro}
In this paper, we apply continuous space-time finite element methods 
on fully unstructured simplicial space-time meshes to the numerical solution
of optimal control problems for linear and semilinear parabolic equations. 
More precisely, we treat the corresponding parabolic forward-backward
optimality systems at once. In this way, we are able to apply  the 
semismooth Newton method for the optimal control of semilinear state 
equations and for problems with pointwise control constraints.
In particular, this is a challenge for our examples with 
reaction-diffusion equations that develop wave-type solutions. 
We present an error analysis for problems with a linear state 
equation without control constraints. However, our numerical examples 
confirm that the continuous space-time finite element approach works 
also well for problems with semilinear equations and 
under additional pointwise control constraints.

Continuous space-time finite element methods (FEM)
for solving parabolic ini\-ti\-al-boundary value problems (IBVP)
on fully unstructured simplicial space-time meshes have recently been 
studied from a mathematical point of view, e.g., 
in \cite{REBPSVLTZ17,ULMRAS19,OS15}, and have been used in engineering 
applications; see, e.g., 
\cite{LSTY:Behr:2008a,LSTY:KaryofylliWendlingMakeHostersBehr:2019a}.
We also refer the reader to the recent review article \cite{OSHY19} 
on this topic and the related references therein. This space-time 
approach considers the time variable as just another variable
in contrast to the classical time-stepping methods 
or to the more recent, but closely related time discontinuous Galerkin 
(dG) or discontinuous Petrov-Galerkin (dPG) methods 
operating on time slices or slabs.
There is a huge amount of papers on these methods.
Here we only refer to the classical monograph  \cite{LSTY:Thomee2006a}
and to the survey articles \cite{LSTY:Gander:2015a,OSHY19}.
The fully unstructured space-time FEM is obviously more flexible 
with respect to approximation, adaptivity, and parallelization
than time-stepping methods. Moving interfaces or spatial domains 
are fixed geometric objects in the space-time domain. However, we 
have to solve one large-scale system of linear or non-linear
algebraic equations at once instead of many smaller systems sequentially 
arising at each time step. This may be seen as a disadvantage when 
using a sequential computation on a standard computer (desktop or laptop) 
with one or only a few cores. However, this is definitely  
a huge advantage on massively parallel computers. 
Even on a standard computer, simultaneous space-time
adaptivity can dramatically reduce the complexity as our numerical 
experiments presented in this paper show. Indeed,  we are able to 
solve optimal control problems for linear and semilinear 
parabolic reaction-diffusion equations in two-dimensional spatial domains
fast and with high accuracy on standard desktops or even laptops. 
We should underline that all of our numerical experiments were performed 
on a laptop or desktop computer.

An optimal control problem for  a parabolic partial differential 
equation leads to necessary optimality conditions that include a 
coupled system consisting of the state equation being 
forward in time, and the adjoint equation (co-state equation) that 
is directed backward in time, see, e.g., \cite{LSTY:Lions:1968a} 
or \cite{FT10}. The numerical solution of this forward-backward 
system is very demanding, because this cannot easily be done by 
standard time-stepping methods or time-slice dG or dPG methods
in an efficient way. We here only mention the works 
by Meidner and Vexler \cite{DMBV2008a,DMBV2008b}
for optimal error estimates of advanced time-stepping 
dG methods in the optimal control of parabolic PDEs.
Therefore, various other methods were applied for spatial dimensions 
larger than one, for instance, gradient type methods that proceed by 
sequentially solving forward and backward equations 
\cite{MHRPMUSU2009,FT10}. Moreover, several techniques were used 
that lower the dimension of the discretized equations to be solved. 
We mention multigrid methods, Hackbusch  \cite{LSTY:Hackbusch:1981a}, 
and cf. also the monograph by Borzi and Schulz \cite{ABVS2012};
model order reduction by proper orthogonal decomposistion (POD), 
Alla and Volkwein \cite{AASV2015}, Kunisch, Volkwein and Xie 
\cite{KKSVLX2004}; wavelet decomposition, Kunoth \cite{AK2001}; 
tensor product approximations  
\cite{LSTY:BuengerDolgovStoll:2020a, 
LSTY:GunzburgerKunoth:2011a,
LSTY:KollmannKolmbauerLangerWolfmayrZulehner:2013a,
LSTY:LangerRepinWolfmayr:2016a};
or adaptive methods based on goal-oriented error estimators, 
Becker et al. \cite{RBHKRR2000}, Hinterm\"uller et al \cite{MHMHCKTK2018}.
Gong, Hinze and Zhou \cite{WGMHZZ2012} propose to solve a 
higher-order optimality system (second-order in time and 
fourth-order in space) by means of a time-slice discretization.
A detailed survey on all related literature would exceed the scope 
of this paper. Hence, we  confine ourselves to the papers mentioned above, 
and the references therein. 
  
In contrast to these approaches, we apply fully unstructured space-time 
FEM to the numerical solution of optimal control problems for linear 
and semilinear parabolic partial differential equations.
The fully unstructured space-time methods are especially suited for 
a forward-backward system since it is  one system of two coupled PDEs 
where the time is just another variable.

We start our investigation with a the standard space-time tracking optimal
control problem subject to a linear parabolic IBVP. The well-posedness of the
optimality system and of the discretized optimality system is studied by
Babu\v{s}ka's theorem. In particular, the discrete inf-sup (stability)
condition leads to asymptotically optimal discretization error estimates. 

For problems with a semilinear state equation or pointwise control 
constraints, we apply the semismooth Newton method, the currently most 
popular and powerful numerical technique for solving optimal control 
problems for PDEs; we refer to Ito and Kunisch \cite{KIKK2008}. In 
this Newton-type method, a sequence of forward-backward equations must 
be solved that easily exceed the storage capacity of standard computers, 
if the space dimension is larger than one. Parallel-in-time numerical 
techniques such as parareal methods, cf. Lions et. al. \cite{JLYMGT2015}, 
Ulbrich \cite{SU2007}, see also Gander \cite{LSTY:Gander:2015a}, can 
also overcome this difficulty.  However, they are not easy to implement. 
In this context, we also mention G\"otschel and Minion \cite{SGMM2019}, 
who apply parallel-in-time methods to the optimal control of the 3D heat 
equation and the 1D Nagumo equation.

In our examples, the level of difficulty is even higher, because our 
reaction-diffusion equations exhibit wave type solutions such as traveling 
wave fronts, spiral waves, or scroll rings. We refer, e.g., to  
Casas et. al. \cite{ECCRFT13,ECCRFT2015,ECCRFT2018},
where optimal control problems for FitzHugh-Nagumo or Schl\"ogl (Nagumo)
equations were solved in space dimensions one and two. Only in the spatially
one-dimensional case, the semismooth Newton method was applied,
cf. \cite{ECCRFT2015}, while two-dimensional problems were tackled by a
nonlinear conjugate gradient optimization method, partially invoking a 
model-predictive approximation. In contrast to the papers mentioned 
above, we directly solve the nonlinear optimality systems 
by the Newton or semismooth Newton method. Finally,  in each iteration, 
we solve one system of linear or non-linear algebraic equations.

The rest of the  paper is structured as follows. 
In Section~\ref{sec:prelim}, we introduce some notation and state some 
preliminary results on the solvability and numerical analysis of the 
parabolic initial-boundary value problem that serves as state equation 
in the optimal control problem studied in Section~\ref{sec:stl2}. 
In Section~\ref{sec:semilinearparabolicequations}, we consider the 
optimal control of semilinear parabolic initial-boundary value problem 
without and with pointwise box constraints on the control. We provide 
and discuss typical numerical examples for all optimal control 
problems investigated in the paper. Finally, some conclusions are 
drawn in Section \ref{sec:con}. 

\section{Preliminaries}
\label{sec:prelim}
The state problem, that appears as constraint in the optimal control 
problems, is given by the linear parabolic IBVP
\begin{equation}
\label{eqn:linearstateequation}
\partial_t u-\Delta_x u  =  z  \textup{ in } Q,\quad
u  =  0 \textup{ on } \Sigma,\quad
u = 0  \textup{ on } \Sigma_0,
\end{equation}
where $Q := \Omega \times (0,T)$, $\Sigma:= \partial \Omega \times (0,T)$,
$\Sigma_0:=\Omega\times\{0\}$. The spatial computational domain 
$\Omega \subset \mathbb{R}^d$, $d=1,2,3$, is supposed to be bounded 
and Lipschitz, $T > 0$ is the final time, $\partial_t$ denotes the 
partial time derivative, $\Delta_x = \sum_{i=1}^d \partial^2_{x_i}$ is 
the spatial Laplacian, and the source term $z$ on the right-hand 
side of the parabolic PDE serves as control.
For simplicity, we consider homogeneous initial and boundary conditions only.
It is clear that this simple model problem can be replaced by more advanced 
parabolic IBVPs as they appear in many practical applications such as 
instationary heat conduction, instationary diffusion, 2D eddy current 
simulation, tumor grow, or after Newton linearization of nonlinear 
parabolic IBVPs as considered in
Section~\ref{sec:semilinearparabolicequations}.
The weak solvability of such kind of parabolic IBVPs was studied in 
space-time Sobolev spaces by Ladyzhenskaya and co-workers 
\cite{LSTY:Ladyzhenskaya:1985a,
LSTY:LadyzhenskayaetSolonnikovUraltseva:1967a}, 
and in Bochner spaces of abstract functions, mapping the time 
interval $(0,T)$ to some Hilbert or Banach space, by Lions 
\cite{LSTY:Lions:1968a}, see also \cite{LSTY:Zeidler:1990a}.
Following the latter approach, the standard weak formulation
of the IBVP~\eqref{eqn:linearstateequation} in Bochner spaces 
of abstract functions reads as follows: Given $z \in Y^*$, 
find $u \in X_0$ such that
\begin{equation}
\label{eqn:stateequation} 
b(u,v) = \langle z, v \rangle_Q, \quad \forall v \in Y, 
\end{equation}
with the bilinear form $b(\cdot,\cdot): X_0 \times Y \to \mathbb{R}$,
\begin{equation}
\label{eqn:bilinearform_b(.,.)} 
b(u,v) := \int_Q \Big[
\partial_t u \, v + \nabla_x u \cdot \nabla_x v \Big] \, dx \, dt, \quad 
\forall (u,v) \in X_0 \times Y, 
\end{equation}
and the linear form $\langle z, \cdot \rangle_Q : Y \to \mathbb{R}$ 
with the duality pairing
\begin{equation}
\label{eqn:rhs(z,.)} 
\langle z, q \rangle_Q := \int_Q z \, v \, dx \, dt, \quad 
\forall  v \in Y, 
\end{equation}
as extension of the inner product in $L^2(Q)$. Similarly, the first 
integral in \eqref{eqn:bilinearform_b(.,.)} has to be understood as 
duality pairing as well. The Bochner spaces $X_0$ and $Y$ are 
specified as follows:
\begin{eqnarray*}
X_0 & := & L^2(0,T; H_0^1(\Omega)) \cap H_{0,}^1(0,T;H^{-1}(\Omega)) \\
  & = & \Big \{ v \in L^2(0,T; H_0^1(\Omega)) : \partial_t v \in 
        L^2(0,T;H^{-1}(\Omega)), \, v = 0 \mbox{ on }  \Sigma_0 \Big \},\\
Y & := & L^2(0,T;H_0^1(\Omega)), \qquad
Y^*:= L^2(0,T;H^{-1}(\Omega)),
\end{eqnarray*}
where $H_0^1(\Omega) :=  \{v \in H^1(\Omega):\, v = 0 \mbox{ on }  
\partial \Omega\}$, and $H^{-1}(\Omega) := H_0^1(\Omega)^*$. Note that 
we have $X_0 = \{ v \in W(0,T):\, v = 0 \mbox{ on }  \Sigma_0\}$ 
as used in \cite{LSTY:Lions:1968a}. The related norms are given by
\begin{eqnarray*}
\| u \|_{X_0} & := &
\Big(\| \partial_t u \|^2_{L^2(0,T;H^{-1}(\Omega))} + 
\| \nabla_x u \|^2_{L^2(Q)} \Big)^{1/2}, \\
\| v \|_Y & := & \| \nabla_x v \|_{L^2(Q)}, \\
\| \partial_t u \|_{Y^*} & = &
\| \partial_t u \|_{L^2(0,T;H^{-1}(\Omega))} =  
\| \nabla_x w_u \|_{L^2(Q)} = \| w_u \|_Y,
\end{eqnarray*}
and where $w_u \in Y$ is the unique solution of the variational formulation
\begin{equation}\label{eqn:Definition w H-1}
\int_Q \nabla_x w_u \cdot \nabla_x v \, dx \, dt = 
\langle \partial_t u , v \rangle_Q , \quad \forall v
\in Y;
\end{equation}
see \cite{OS15}. In fact, we have
\[
\| u \|_{X_0} = \Big[ \| w_u \|_Y^2 + \| u \|_Y^2 \Big]^{1/2} .
\]
The well-posedness of variational problems such as \eqref{eqn:stateequation} 
can be investigated by the Ne\u{c}as-Babu\v{s}ka theorem 
\cite{IB71,LSTY:Necas:1962a} that is sometimes also called the 
Banach-Ne\u{c}as-Babu\v{s}ka theorem  \cite{LSTY:ErnGuermond:2004a}
or the Babu\v{s}ka-Aziz theorem \cite{LSTY:BabuskaAziz:1972a}, see 
also \cite{LSTY:Braess:2007a}. This theorem states that the operator 
$B: X_0 \to Y^*$ generated by the bilinear form $b(\cdot,\cdot)$ 
is an isomorphism if and only if the following three conditions are 
fulfilled:
\begin{enumerate}
\item boundedness (continuity) of $b(\cdot,\cdot)$, i.e., there 
      exists a positive constant $\beta_1$:
      \begin{equation}\label{eqn:boundedness_b(.,.)}
      |b(u,v)| \leq \beta_1 \, \| u \|_{X_0} \| v \|_Y , \quad 
      \forall (u,v) \in X_0 \times Y;
      \end{equation}
\item inf-sup (stability) condition (surjectivity of $B^*$), 
      i.e., there exists a positive constant $\beta_2$
      such that
      \begin{equation}\label{eqn:infsup_b(.,.)}
      \inf\limits_{0 \neq u \in X_0}\sup\limits_{0 \neq v \in Y}
      \frac{b(u,v)}{\| u \|_{X_0} \| v \|_Y} \ge \beta_2; 
      \end{equation}
\item injectivity of $B^*$:
      \begin{equation}\label{eqn:injectivity_b(.,.)}
      \forall v \in Y \setminus \{0\} \; \exists 
      \widetilde{u} \in X_0: \; b(\widetilde{u},v) \neq 0.
      \end{equation}
\end{enumerate}
It is easy to show that the bilinear form \eqref{eqn:bilinearform_b(.,.)} 
is bounded with $\beta_1 = \sqrt{2}$. The inf-sup condition 
\eqref{eqn:infsup_b(.,.)} follows from \cite[Theorem 2.1]{OS15} with the 
stability constant $\beta_2 = 1/(2\sqrt{2})$. 
To prove \eqref{eqn:injectivity_b(.,.)}, for $v \in Y \setminus \{0\}$, 
we choose
\[
\widetilde{u}(x,t) = \int_0^t v(x,s) \, ds, \quad (x,t) \in Q .
\]
By definition, we have $\widetilde{u} \in X_0$, and
\[
b(\widetilde{u},v) = \| v \|^2_{L^2(Q)} +
\frac{1}{2} \, \| \nabla_x \widetilde{u}(T) \|^2_{L^2(\Omega)} > 0.
\]
Therefore, by the Ne\u{c}as-Babu\v{s}ka theorem, the variational problem 
\eqref{eqn:stateequation} is well-posed.

For the finite element discretization of the variational formulation 
\eqref{eqn:stateequation}, we introduce conforming space-time finite 
element spaces $X_{0,h} \subset X_0$ and $Y_h \subset Y$, 
where we assume $X_{0,h} \subseteq Y_h$. In particular, we may use 
$X_{0,h} = Y_h = S_h^1(Q_h) \cap X_0$ spanned by continuous and piecewise 
linear basis functions which are defined with respect to some 
admissible decomposition $\mathcal{T}_h(Q)$ of the space-time domain 
$Q$ into shape regular simplicial finite elements $\tau_\ell$, and which 
are zero at the initial time $t=0$ and at the lateral boundary $\Sigma$,
where $h$ denotes a suitable mesh-size parameter, see, e.g., 
\cite{LSTY:Braess:2007a,LSTY:ErnGuermond:2004a,OS15}.
Then the finite element approximation of 
\eqref{eqn:stateequation} is to find $u_h \in X_{0,h}$ such that
\begin{equation}\label{eqn:stateequation FEM}
b(u_h,v_h) = \langle z , v_h \rangle_Q, \quad \forall v_h \in Y_h .
\end{equation}
When replacing \eqref{eqn:Definition w H-1}
by its finite element approximation to find $w_{u,h} \in Y_h$ such that
\begin{equation}\label{eqn:Definition w H-1 FEM}
\int_Q \nabla_x w_{u,h} \cdot \nabla_x v_h \, dx \, dt =
\int_Q \partial_t u \, v_h \, dx \, dt, \quad \forall v_h \in Y_h,
\end{equation}
we can define a discrete norm
\[
\| u \|_{X_{0,h}} \, := \Big[ \| w_{u,h} \|_Y^2 + \| u \|_Y^2 \Big]^{1/2} \, .
\]
Due to the definition of $w_{u,h}$ as solution of the variational formulation
\eqref{eqn:Definition w H-1 FEM}, we conclude
\begin{equation}\label{Inequality wh w}
\| w_{u,h} \|_Y \leq \| w_u \|_Y \quad \mbox{for all} \; u \in X_0,
\end{equation}
while the opposite inequality is in general not true. 
As in the continuous case, see \eqref{eqn:infsup_b(.,.)},
we can prove a discrete inf-sup condition, see \cite[Theorem 3.1]{OS15},
\begin{equation}\label{eqn:bilinear form heat discrete stability}
\frac{1}{2\sqrt{2}} \, \| u_h \|_{X_{0,h}} \leq 
\sup\limits_{0 \neq v_h \in Y_h} \frac{b(u_h,v_h)}{\| v_h \|_Y} , \quad
\forall u_h \in X_{0,h} .
\end{equation}
Hence, from the discrete version of Ne\u{c}as-Babu\v{s}ka's theorem,
we conclude unique solvability of the Galerkin scheme 
\eqref{eqn:stateequation FEM}. Furthermore, we obtain the following
quasi-optimal error estimate, see \cite[Theorem 3.2]{OS15}:
\begin{equation}\label{eqn:stateequation FEM Cea}
\| u - u_h \|_{X_{0,h}} \leq 5 \inf\limits_{z_h \in X_{0,h}} \| u - z_h \|_{X_0} \, .
\end{equation}
In particular, when assuming $u \in H^2(Q)$, this finally results in the
energy error estimate, see \cite[Theorem 3.3]{OS15},
\begin{equation}
\| u - u_h \|_{L^2(0,T;H^1_0(\Omega))} \leq c \, h \, |u|_{H^2(Q)} \, .
\end{equation}
Once the basis is chosen, the finite element scheme 
\eqref{eqn:stateequation FEM} is nothing but a huge linear system 
of the form
\begin{equation}\label{eqn:FEsystem}
K_h \underline{u}_h = \underline{f}_h
\end{equation}
with a positive definite, but non-symmetric system matrix $K_h$ 
that can be generated together with the right-hand side $\underline{f}_h$
similar as in the elliptic case. 
The linear system \eqref{eqn:FEsystem} can efficiently be solved by means of 
the preconditioned GMRES method, see \cite{OSHY18,OSHY19}, 
where Algebraic Multigrid (AMG) preconditioning is used.
We also use AMG preconditioned GMRES 
as solver in all numerical experiments presented in this paper.
It is clear that the unstructured space-time approach to optimal 
control problems presented in this paper allows 
full space-time adaptivity and parallelization.

\section{Space-time tracking}
\label{sec:stl2}

\subsection{The model problem}
For a given target function $u_d \in L^2(Q)$ and a regularization 
parameter $\varrho > 0$, we consider the minimization of the cost functional
\begin{equation}
{\mathcal J}(u,z):=\frac{1}{2}\int_Q |u-u_d|^2 \, dx \, dt + 
\frac{1}{2}\varrho \, \|z\|_{L^2(Q)}^2 
\label{eqn:linmodspacetime}
\end{equation}
subject to the linear parabolic IBVP~\eqref{eqn:linearstateequation},
where the control $z$ is taken from $L^2(Q)$.

\subsection{The optimality system}
If a control $z$ is optimal with the associated state $u$, then the 
following first-order necessary optimality conditions must be 
satisfied: There is a unique solution $p$ of the 
{\em adjoint equation}
\begin{equation*}
-\partial_t p-\Delta_x p = u-u_d \textup{ in } Q, \quad
p = 0 \textup{ on } \Sigma, \quad
p = 0 \textup{ on } \Sigma_T := \Omega \times \{ T \},
\end{equation*}
such that the so-called {\em gradient equation}
\begin{equation*}
p +\varrho z  = 0 \quad \textup{ in } Q
\end{equation*}
is satisfied. When eliminating the control $z$, 
the following {\em optimality system} is necessary (and by convexity 
of the problem also sufficient) for the optimality of its solution $(u,p)$:
\begin{equation} \label{optimality_system}
\begin{array}{c}
\varrho \left[ \partial_t u-\Delta_x u \right] + p = 0 \textup{ in } Q, \quad
u = 0 \textup{ on } \Sigma, \quad
u = 0 \textup{ on } \Sigma_0, \\[1mm]
-\partial_t p-\Delta_x p = u-u_d \textup{ in } Q, \quad
p = 0 \textup{ on } \Sigma, \quad
p = 0 \textup{ on } \Sigma_T.
\end{array}
\end{equation}
The solution of this system exists and is unique, since the optimal 
control problem has a unique optimal solution; see \cite{LSTY:Lions:1968a}. 
This is due to the strict convexity of the functional $J$. If the 
solution $(u,p)$ is given, then $z = - p/\varrho$ is the optimal control.
The weak formulation of the optimality system \eqref{optimality_system} 
is to find $u \in X_0$ and $p \in X_T$ such that
\begin{equation}\label{sec:sl2:eq:weakforml2}
\begin{array}{rcl}
\displaystyle
\varrho \int_Q \Big[
\partial_t u \, v + \nabla_x u \cdot \nabla_x v \Big] dx \, dt 
+ \int_Q p \, v \, dx \, dt 
& = & 0, \\
\displaystyle
- \int_Q u \, q \, dx \, dt + \int_Q 
\Big[ - \partial_t p \, q + \nabla_x p \cdot \nabla_x q \Big] dx \, dt 
& = & \displaystyle
- \int_Q u_d \, q \, dx \, dt
\end{array}
\end{equation}
is satisfied for all $v,q\in Y$. Note that
\begin{eqnarray*}
X_T & := & L^2(0,T;H^1_0(\Omega)) \cap H^1_{,0}(0,T;H^{-1}(\Omega)) \\ & = &
\Big \{
p \in L^2(0,T;H^1_0(\Omega)) : \partial_t p \in L^2(0,T;H^{-1}(\Omega)), p = 0
\; \mbox{on} \; \Sigma_T \Big \} .
\end{eqnarray*}
An equivalent version of 
\eqref{sec:sl2:eq:weakforml2} is the saddle point problem to find 
$(u,p) \in X_0 \times X_T$ such that
\begin{equation}\label{sec:sl2:eqn:saddle}
{\mathcal B}(u,p;v,q) = - \langle u_d , q \rangle_{L^2(Q)}, \quad
\forall v,q \in Y ,
\end{equation}
where
\begin{eqnarray}\label{sec:sl2:eqn:def B}
{\mathcal B}(u,p;v,q) & = &
\varrho \int_Q \Big[ \partial_t u \, v +
\nabla_x u \cdot \nabla_x v \Big] dx \, dt +
\int_Q p \, v \, dx \, dt \\ & & \nonumber 
- \int_Q u \, q \, dx \, dt 
+ \int_Q \Big[ - \partial_t p \, q + 
\nabla_x p \cdot \nabla_x q \Big] dx \, dt 
\end{eqnarray}
is a bounded bilinear form for $(u,p) \in X_0 \times X_T$, and
$(v,q) \in Y \times Y$, i.e.,
\[
|{\mathcal B}(u,p;v,q)| \, \leq \,
c_{\mathcal{B}}(\varrho) \, \Big(
\| u \|_{X_0}^2 + \| p \|^2_{X_T} 
\Big)^{1/2} \Big( \| v \|_Y^2 + \| q \|^2_Y \Big)^{1/2}
\]
with some positive constant $c_{\mathcal{B}}(\varrho)$.

\begin{lemma}\label{Lemma inf sup L2}
The bilinear form as given in \eqref{sec:sl2:eqn:def B} satisfies the
stability condition
\[
\frac{1}{2\sqrt{2}}
\sqrt{\varrho \, \| u \|_{X_0}^2 + \| p \|_{X_T}^2}
\leq \sup\limits_{0 \neq (v,q) \in Y \times Y}
\frac{{\mathcal B}(u,p;v,q)}
{\sqrt{\varrho \, \| v \|_Y^2 + \| q \|_Y^2}} 
\]
for all $(u,p) \in X_0 \times X_T$.
\end{lemma}
\begin{proof}
For $u \in X_0$, we define $w_u \in Y$ as the unique solution 
of the elliptic variational problem
\[
\int_Q \nabla_x w_u \cdot \nabla_x v \, dx \, dt=
\int_Q \partial_t u \, v \, dx \, dt, \quad
\forall v \in Y \, .
\]
As in \cite{OS15}, we then have
\[
\| \partial_t u \|_{Y^*} = \| w_u \|_Y \, , \quad \mbox{i.e.}, \quad
\| u \|_{X_0}^2 = \| w_u \|_Y^2 + \| u \|_Y^2 \, .
\]
In the same way, we define $w_p \in Y$ as the unique solution of the
variational problem
\[
\int_Q \nabla_x w_p \cdot \nabla_x q \, dx \, dt =
- \int_Q \partial_t p \, q \, dx \, dt, \quad \forall q \in Y, 
\]
where we conclude
\[
\| \partial_t p \|_{Y^*} = \| w_p \|_Y \, , \quad \mbox{i.e.}, \quad
\| p \|^2_{X_T} = \| w_p \|^2_Y + \| p \|^2_Y \, .
\]
For $w_p \in Y$ and for almost all $x \in \Omega$, we define
\[
v_p(x,t) = \int_t^T w_p(x,s) \, ds,
\]
satisfying $v_p \in X_T$, i.e., $v_p=0$ on $\Sigma_T$. Using 
$w_p = - \partial_t v_p$, and integration by parts in time, 
and $u=0$ on $\Sigma_0$, we obtain
\begin{eqnarray*}
\langle u , w_p \rangle_{L^2(Q)} & = & 
- \int_0^T \int_\Omega u(x,t) \, \partial_t v_p(x,t) \, dx \, dt \\ & = &
- \left. \int_\Omega u(x,t) \, v_p(x,t) \, dx \right|_0^T +
\int_0^T \int_\Omega \partial_t u(x,t) \, v_p(x,t) \, dx \, dt  \\ & = &
\int_0^T \int_\Omega \nabla_x w_u(x,t) \cdot \nabla_x v_p(x,t) \, dx \, dt 
\\ & = &
\int_0^T \int_t^T \int_\Omega \nabla_x w_u(x,t) \cdot \nabla_x w_p(x,s) 
\, dx \, ds \, dt \, .
\end{eqnarray*}
Analogously, defining
\[
v_u(x,s) = \int_0^s w_u(x,t) \, dt, \quad x \in \Omega, \; s \in (0,T),
\]
we get
\begin{eqnarray*}
\langle p , w_u \rangle_{L^2(Q)} & = & 
\int_0^T \int_\Omega p(x,s) \, \partial_s v_u(x,s) \, dx \, ds \\ & = &
\left. \int_\Omega p(x,s) \, v_u(x,s) \, dx \right|_0^T -
\int_0^T \int_\Omega \partial_s p(x,s) \, v_u(x,s) \, dx \, ds \\ & = &
\int_0^T \int_\Omega \nabla_x w_p(x,s) \cdot \nabla_x v_u(x,s) \, dx \, ds 
\\ & = &
\int_0^T \int_0^s \int_\Omega \nabla_x w_p(x,s) \cdot \nabla_x w_u(x,t) 
\, dx \, dt \, ds \\ & = &
\int_0^T \int_t^T \int_\Omega \nabla_x w_p(x,s) \cdot \nabla_x w_u(x,t) 
\, dx \, ds \, dt .
\end{eqnarray*}
Hence, we have
\[
\langle u , w_p \rangle_{L^2(Q)} = \langle p , w_u \rangle_{L^2(Q)}.
\]
For $v = u + w_u \in Y$ and $q = p + w_p \in Y$, we now obtain
\begin{eqnarray*}
{\mathcal B}(u,p;v,q) & = & \\ & & \hspace*{-2.5cm} = \,
\varrho \int_Q \Big[ \partial_t u \, (u+w_u) +
\nabla_x u \cdot \nabla_x (u+w_u) \Big] dx \, dt +
\int_Q p \, (u+w_u) \, dx \, dt \\ & & \hspace*{-2.2cm} 
- \int_Q u \, (p+w_p) \, dx \, dt 
+ \int_Q \Big[ - \partial_t p \, (p+w_p) + 
\nabla_x p \cdot \nabla_x (p+w_p) \Big] dx \, dt 
\\ & & \hspace*{-2.5cm} = \,
\varrho \int_Q \Big[ \partial_t u \, (u+w_u) +
\nabla_x u \cdot \nabla_x (u+w_u) \Big] dx \, dt \\ & & 
\hspace*{-1.2cm} +
\int_Q \Big[ - \partial_t p \, (p+w_p) + 
\nabla_x p \cdot \nabla_x (p+w_p) \Big] dx \, dt \\ &&
\hspace*{-2.5cm} =
\varrho \left[
\frac{1}{2} \, \| u(T) \|^2_{L^2(\Omega)} +
\| w_u \|^2_Y + \| u \|_Y^2 + 
\int_Q \nabla_x u \cdot \nabla_x w_u \, dx \, dt
\right] \\ && \hspace*{-1.2cm} + \left[
\frac{1}{2} \| p(0) \|^2_{L^2(\Omega)} + \| w_p \|^2_Y + \| p \|_Y^2 +
\int_Q \nabla_x p \cdot \nabla_x w_p \, dx \, dt \right] \\ &&
\hspace*{-2.5cm} \geq
\varrho \, \Big[
\| w_u \|^2_Y + \| u \|_Y^2 - \| u \|_Y \| w_u \|_Y
\Big] + \Big[ \| w_p \|^2_Y + \| p \|_Y^2 - \| p \|_Y \| w_p \|_Y \Big] \\ &&
\hspace*{-2.5cm} \geq
\frac{\varrho}{2} \, \Big[ \| w_u \|^2_Y + \| u \|_Y^2 \Big] + 
\frac{1}{2} \, \Big[
\| w_p \|^2_Y + \| p \|_Y^2 \Big] \\ && \hspace*{-2.5cm}
= \frac{1}{2} \, \Big[ \varrho \, \| u \|^2_{X_0} + \| p \|^2_{X_T} \Big] \, .
\end{eqnarray*}
Moreover, using the triangle and H\"older's inequality, we get
\[
\| v \|_Y = \| u + w_u \|_Y \leq \| u \|_Y + \| w_u \|_Y \leq
\sqrt{2} \, \sqrt{\| u \|_Y^2 + \| w_u \|_Y^2} =
\sqrt{2} \, \| u \|_{X_0},
\]
as well as
\[
\| q \|_Y \leq \sqrt{2} \, \sqrt{\| p \|_Y^2 + \| w_p \|_Y^2} =
\sqrt{2} \, \| p \|_{X_T} .
\]
With this, we can now estimate
\begin{eqnarray*}
\sqrt{\varrho \, \|v\|_Y^2 + \| q \|_Y^2} \, 
\sqrt{\varrho \, \| u \|_{X_0}^2 + \| p \|_{X_T}^2} & \leq & \sqrt{2} \,
\Big(
\varrho \, \| u \|_{X_0}^2 + \| p \|_{X_T}^2 
\Big) \\ & \leq & 2 \sqrt{2} \, {\mathcal B}(u,p;v,q),
\end{eqnarray*}
which implies the stability condition as stated.
\end{proof}

\begin{lemma}
For all $v,q \in Y$, we have the injectivity condition
\[
\sup\limits_{(u,p) \in X_0 \times X_T} {\mathcal B}(u,p;v,q) > 0 \, .
\]
\end{lemma}
\begin{proof}
For $v \in Y$ and for almost all $x \in \Omega$, $s \in (0,T)$,
we define
\[
u_v(x,s) = \int_0^s v(x,t) \, dt, \quad \mbox{i.e.,} \; u_v \in X_0,
\]
while, for $q \in Y$ and for almost all $x \in \Omega$, $t \in (0,T)$,
we define
\[
p_q(x,t) = \int_t^T q(x,s) \, ds, \quad \mbox{i.e.,} \; p_q \in X_T .
\]
With this we have
\begin{eqnarray*}
{\mathcal B}(u_v,p_q;v,q) & = & 
\varrho \int_Q \Big[ \partial_s u_v \, v +
\nabla_x u_v \cdot \nabla_x v \Big] dx \, ds +
\int_Q p_q \, v \, dx \, dt \\ & & 
- \int_Q u_v \, q \, dx \, ds 
+ \int_Q \Big[ - \partial_t p_q \, q + 
\nabla_x p_q \cdot \nabla_x q \Big] dx \, dt 
\\ & & \hspace*{-2cm} = \,
\varrho \int_Q \Big[ v^2 +
\nabla_x u_v \cdot \nabla_x \partial_s u_v \Big] dx \, dt +
\int_0^T \int_t^T \int_\Omega q \, v \, dx \, ds \, dt \\ & & +
\int_Q \Big[ q^2 - 
\nabla_x p_q \cdot \nabla_x \partial_tp_q \Big] dx \, dt - 
\int_0^T \int_0^s \int_\Omega v \, q \, dx \, dt \, ds \\ & & \hspace*{-2cm} = \,
\varrho \, \Big[ 
\| v \|^2_{L^2(Q)} + \frac{1}{2} \, \| \nabla_x u_v(T) \|^2_{L^2(\Omega)} \Big]
+ \| q \|^2_{L^2(Q)} + \frac{1}{2} \, \| \nabla_x p_q(0) \|^2_{L^2(\Omega)}
\\ & & \hspace*{-2cm} > \, 0,
\end{eqnarray*}
which concludes the proof.
\end{proof}

\noindent
Now, as a consequence of the Ne\v{c}as-Babu\v{s}ka theorem, we are in the
position to state the main result of this subsection:

\begin{theorem}
For given $u_d \in L^2(Q)$,
the saddle point problem \eqref{sec:sl2:eqn:saddle} admits a unique
solution $(u,p) \in X_0 \times X_T$.
\end{theorem}

\noindent
Note that for $u_d \in Y^*$ the saddle point formulation
\eqref{sec:sl2:eqn:saddle} defines an isomorphism from
$X_0 \times X_T$ onto $Y^* \times Y^*$.

\subsection{Discretization of the optimality system}
For the numerical solution of the saddle point problem 
\eqref{sec:sl2:eqn:saddle}, we use the conforming space-time
finite element spaces $X_{0,h} = Y_{0,h} = S_h^1(Q_h) \cap X_0$, 
and we introduce the space $X_{T,h} = Y_{T,h} = S_h^1(Q_h) \cap X_T$ of
piecewise linear and continuous basis functions which are zero at the 
final time $T$.
Then the finite element approximation of \eqref{sec:sl2:eqn:saddle}
is to find $(u_h,p_h) \in X_{0,h} \times X_{T,h}$ such that
\begin{equation}\label{sec:sl2:eqn saddle FEM}
{\mathcal B}(u_h,p_h;v_h,q_h) =  - \langle u_d , q \rangle_{L^2(Q)} , \quad
\forall (v_h,q_h) \in Y_{0,h} \times Y_{T,h}.
\end{equation}
To ensure unique solvability of \eqref{sec:sl2:eqn saddle FEM}, we need to
establish a discrete inf-sup stability condition as follows, but which is
formulated in $Y \times Y$, instead of $X_0 \times X_T$ as used in the
continuous case.

\begin{lemma}
For all $(u_h,p_h) \in X_{0,h} \times X_{T,h}$, 
there holds the discrete inf--sup stability condition
\[
\frac{1}{2 \sqrt{2}} \,
\sqrt{\varrho \, \| u_h \|_Y^2 + \| p_h \|_Y^2}
\leq \sup\limits_{0 \neq (v_h,q_h) \in Y_{0,h} \times Y_{T,h}}
\frac{{\mathcal B}(u_h,p_h;v_h,q_h)}
{\sqrt{\varrho \, \| v_h \|_Y^2 + \| q_h \|_Y^2}} \, .
\]
\end{lemma}

\begin{proof}
For $(u_h,p_h) \in X_{0,h} \times X_{T,h}$, we define
$w_{u_h,h} \in Y_{0,h}$ as the unique solution of the variational problem
\begin{equation}\label{def wuh}
\int_Q \nabla_x w_{u_h,h} \cdot \nabla_x v_h \, dx \, dt =
\int_Q \Big[ \partial_t u_h + \frac{1}{\varrho} p_h \Big] 
\, v_h \, dx \, dt , \quad \forall v_h \in Y_{0,h} ,
\end{equation}
and $w_{p_h,h} \in Y_{T,h} $ satisfying
\begin{equation}\label{def wph}
\int_Q \nabla_x w_{p_h,h} \cdot \nabla_x q_h \, dx \, dt =
- \int_Q \Big[ \partial_t p_h + u_h \Big] \, q_h \, dx \, dt, 
\quad \forall q_h \in Y_{T,h}.
\end{equation}
For $v_h = u_h + w_{u_h,h} \in Y_{0,h}$ and $q_h = p_h + w_{p_h,h} \in Y_{T,h}$, 
we have, as in the proof
of Lemma \ref{Lemma inf sup L2}, see also \cite{OS15}, and
using \eqref{def wuh} and \eqref{def wph},
\begin{eqnarray*}
{\mathcal B}(u_h,p_h;v_h,q_h) & = & 
\varrho \int_Q \Big[ \partial_t u_h \, (u_h+w_{u_h,h}) +
\nabla_x u_h \cdot \nabla_x (u_h+w_{u_h,h}) \Big] dx \, dt  \\ & & \hspace*{-1cm}
+ \int_Q \Big[ - \partial_t p_h \, (p_h+w_{p_h,h}) + 
\nabla_x p_h \cdot \nabla_x (p_h+w_{p_h,h}) \Big] dx \, dt  \\ & & \hspace*{-1cm}
+ \int_Q p_h \, (u_h+w_{u_h,h}) \, dx \, dt 
- \int_Q u_h \, (p_h+w_{p_h,h}) \, dx \, dt 
\\ & & \hspace*{-2cm} \geq \, \varrho
\int_Q \Big[ \Big( \partial_t u_h + \frac{1}{\varrho} p_h \Big) w_{u_h,h} +
\nabla_x u_h \cdot \nabla_x (u_h+w_{u_h,h}) \Big] dx \, dt  \\ & & \hspace*{-1cm} 
+ \int_Q \Big[ - \Big( \partial_t p_h + u_h \Big) w_{p_h,h} + 
\nabla_x p_h \cdot \nabla_x (p_h+w_{p_h,h}) \Big] dx \, dt  
\\ & & \hspace*{-2cm} = \, 
\varrho \int_Q \Big[ 
\nabla_x w_{u_h,h} \cdot \nabla_x w_{u_h,h} +
\nabla_x u_h \cdot \nabla_x (u_h+w_{u_h,h}) \Big] dx \, dt  \\ & & 
\hspace*{-1cm} + \int_Q \Big[ 
\nabla_x w_{p_h,h} \cdot \nabla_x w_{p_h,h} + 
\nabla_x p_h \cdot \nabla_x (p_h+w_{p_h,h}) \Big] dx \, dt \\[2mm] & &
\hspace*{-2cm} = \, \varrho \, \Big(
\| \nabla_x w_{u_h,h} \|^2_{L^2(Q)} + \| \nabla_x u_h \|^2_{L^2(Q)} +
\langle \nabla_x u_h , \nabla_x w_{u_h,h} \rangle_{L^2(Q)} \Big) \\[1mm] & &
\hspace*{-1cm} 
+ \| \nabla_x w_{p_h,h} \|^2_{L^2(Q)} + \| \nabla_x p_h \|^2_{L^2(Q)} +
\langle \nabla_x p_h , \nabla_x w_{p_h,h} \rangle_{L^2(Q)} \\[2mm] & &
\hspace*{-2cm} \geq \, \frac{1}{2} \, \left[ \varrho \, \Big(
\| w_{u_h,h} \|^2_Y + \| u_h \|^2_Y 
\Big) + \| w_{p_h,h} \|^2_Y + \| p_h \|^2_Y \right] \\[2mm] & &
\hspace*{-2cm} \geq \, \frac{1}{2} \, \left[ \varrho \, \| u_h \|^2_Y 
+ \| p_h \|^2_Y \right]^{1/2}
\\ & & \hspace*{-1cm} \cdot \left[ \varrho \, \Big(
\| w_{u_h,h} \|^2_Y + \| u_h \|^2_Y 
\Big) + \| w_{p_h,h} \|^2_Y + \| p_h \|^2_Y \right]^{1/2} .
\end{eqnarray*}
With
\[
\| v_h \|_Y^2 \leq \Big( \| u_h \|_Y + \| w_{u_h,h} \|_Y \Big)^2 \leq 2 \,
\Big( \| u_h \|_Y^2 + \| w_{u_h,h} \|_Y^2 \Big)
\]
and
\[
\| q_h \|_Y^2 \leq 2 \,
\Big( \| p_h \|_Y^2 + \| w_{p_h,h} \|_Y^2 \Big) 
\]
we further have
\[
\varrho \, \| v_h \|_Y^2 + \| q_h \|_Y^2 \leq 
2 \, \left[ \varrho  \Big( \| u_h \|_Y^2 + \| w_{u_h,h} \|_Y^2 \Big)
+ \| p_h \|_Y^2 + \| w_{p_h,h} \|_Y^2 \right],
\]
which finally implies
\[
\geq \, \frac{1}{2\sqrt{2}} \, \left[ \varrho \, \| u_h \|^2_Y 
+ \| p_h \|^2_Y \right]^{1/2}
\left[ \varrho \, \| v_h \|^2_Y + \| q_h \|^2_Y \right]^{1/2} .
\]
\end{proof}

\noindent
Now, using standard arguments as for the heat equation, we can prove
a best approximation result, see Sect.~\ref{sec:prelim} and \cite{OS15}, 
\begin{eqnarray*}
&& \sqrt{\varrho \, \| u - u_h \|_Y^2 + \| p - p_h \|_Y^2}
\\ && \hspace*{1cm}
\leq (1 + 2 \sqrt{2} c_\mathcal{B}(\rho)) 
\inf\limits_{0 \neq (v_h,q_h) \in X_{0,h} \times X_{T,h}} 
\sqrt{\| u - v_h \|_{X_0}^2 + \| p - q_h \|_{X_T}^2},
\end{eqnarray*}
and therefore we can state the main result of this section. 

\begin{theorem}\label{thm:l2conv}
Assume that the solution $(u,p) \in X_0 \times X_T$ of the saddle point problem
\eqref{sec:sl2:eqn:saddle} satisfies $u,p \in H^2(Q)$.
Let $X_{0,h}=Y_{0,h} = S_h^1(Q_h) \cap X_0$ and 
$X_{T,h}=Y_{T,h} = S_h^1(Q_h) \cap X_T$ be conforming finite element
spaces. Then the discrete saddle point problem 
\eqref{sec:sl2:eqn saddle FEM} admits a unique solution
$(u_h,p_h) \in X_{0,h} \times X_{T,h}$ satisfying the error estimate
\[
\varrho \, \| u - u_h \|_Y^2 + \| p - p_h \|_Y^2 \leq c \, h^2 \,
\Big( |u|^2_{H^2(Q)} + |p|^2_{H^2(Q)} \Big) \, .
\]
\end{theorem}

\subsection{Numerical experiments}
In all the numerical examples considered in this work, we set 
$\Omega=(0, 1)^2$, $T=1$, and therefore $Q=(0,1)^3$. The initial 
(coarsest) space-time finite element mesh contains $125$ vertices 
($5$ vertices in each direction), $384$ tetrahedral elements, and
thus the initial mesh size is $h=1/4$. By uniform refinement 
(red-green refinement \cite{JB95}), the mesh size will be reduced 
successively, i.e., $h=1/8$, $1/16$ and so on. The numerical 
experiments are performed on a desktop with Intel@ Xeon@ Prozessor  
E5-1650 v4 ($15$ MB Cache, $3.60$ GHz), and $64$ GB memory. To solve the
discrete linear coupled first-order necessary optimality system, we use the
algebraic multigrid preconditioned GMRES method. The relative residual error
$\epsilon=10^{-7}$
is taken as a stopping criterion for the GMRES iteration. In
constructing the algebraic multigrid preconditioner for the coupled system, we
utilize a simple blockwise coarsening strategy and a blockwise ILU smoother on
each level; see the performance study for solving such coupled systems in
\cite{OSHY19}. 

\subsubsection{An example with explicitly known solution (Example 1)}
In the first example, we consider the following explicitly known solution 
of the first order optimality system
\begin{equation*}
\begin{aligned}
u(x,t)&=\sin(\pi x_1)\sin(\pi x_2)\left(at^2+bt\right),
\quad a=-\frac{2\pi^2+1}{2\pi^2+2}, \quad b=1, \\
p(x,t)&=-\varrho\sin(\pi x_1)\sin(\pi x_2)
\left( 2\pi^2 a t^2+(2\pi^2b + 2a)t+b\right),\\[1mm]
z(x,t)&=\sin(\pi x_1)\sin(\pi x_2)
\left(2 \pi^2 a t^2+(2\pi^2b + 2a)t+b \right),
\end{aligned}
\end{equation*}
and we set $\varrho=0.01$. The exact solution $u$ satisfies homogeneous 
initial and boundary conditions for the state equation, and $p$ obeys 
the homogeneous terminal and boundary conditions for the adjoint equation; 
see Fig.~\ref{fig:l2_p11} for an illustration. 
\begin{figure}[h]
\centering
\includegraphics[width=0.32\textwidth]{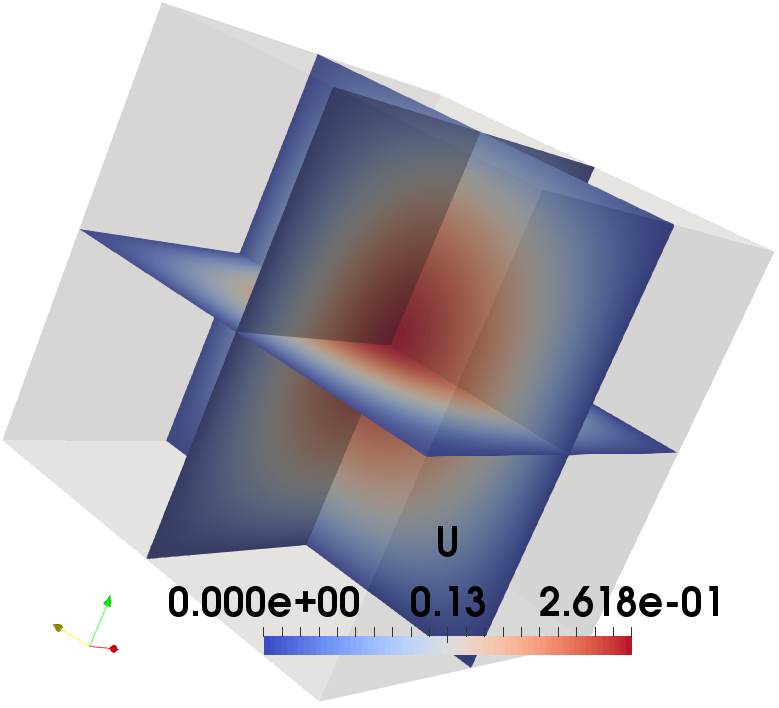}
\includegraphics[width=0.32\textwidth]{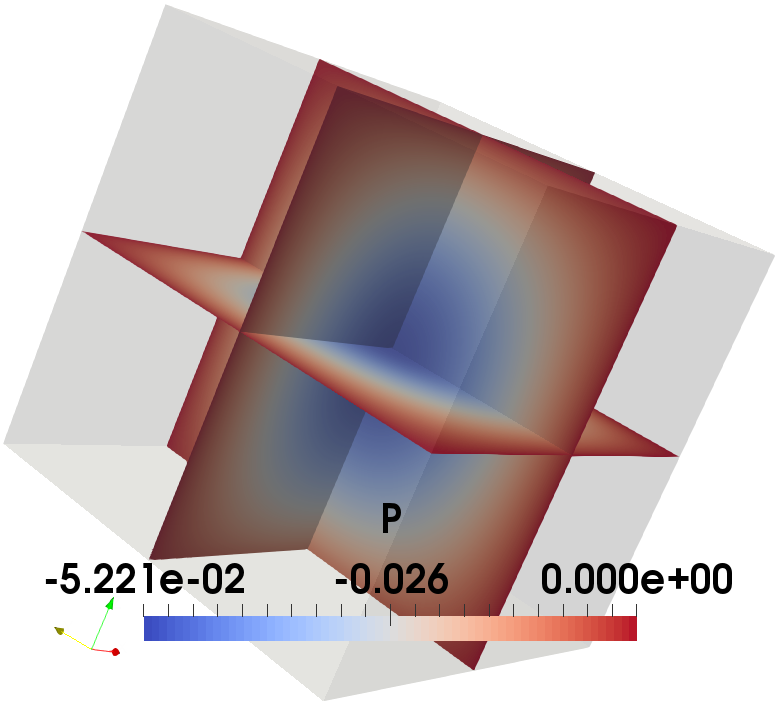}
\includegraphics[width=0.32\textwidth]{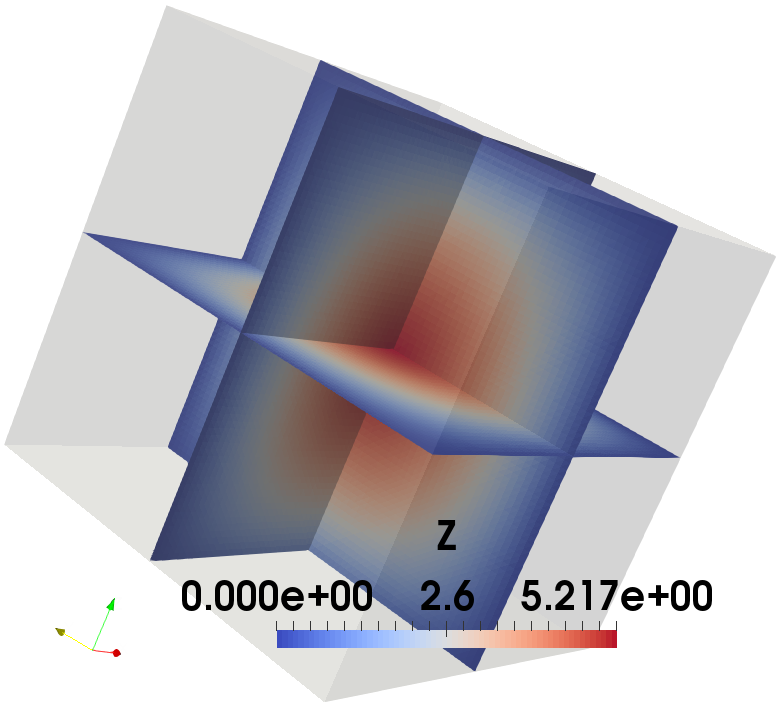}
\caption{Example 1, numerical solutions of $u$, $p$, and $z$ for the 
linear model problem (from left to right).}
\label{fig:l2_p11}
\end{figure} 
The estimated order of convergence (eoc) is provided in Tables
\ref{tab:l2_eocl2h1_p11}-\ref{tab:l2_eocobj_p11}. From these results, we
clearly see optimal convergence in $Y=L^2(0,T; H^1_0(\Omega))$ as predicted 
by Theorem \ref{thm:l2conv}. In addition, we observe a nearly optimal 
convergence rate in $L^2(Q)$. Finally, we see the second-order convergence 
rate of the objective functional.
\begin{table}[h]\caption{Example 1, estimated order of convergence (eoc) 
for $u_h$ and $p_h$ in $Y$ for the linear model problem.} 
\centering
\begin{tabular}{rrrrrr}
\toprule
\#Dofs & $h$ & $\|u-u_h\|_Y$ &  eoc & $\|p-p_h\|_Y$ &  eoc  \\
\midrule
$250$ & $2^{-2}$& $2.218e-1$    & $-$  & $4.201e-2$    & $-$       \\
$1,458$ & $2^{-3}$& $1.141e-1$    & $0.959$  &$2.235e-2$    & $0.910$  \\
$9,826$ & $2^{-4}$ & $5.677e-2$    & $1.007$ & $1.123e-2$    & $0.993$  \\
$71,874$ & $2^{-5}$ & $2.816e-2$    & $1.012$ & $5.588e-3$    & $1.007$   \\
$549,250$ & $2^{-6}$ &$1.400e-2$    & $1.008$ & $2.781e-3$    & $1.006$ \\
$2,146,689$ & $2^{-7}$ &$6.983e-3$    & $1.004$ & $1.387e-3$    & $1.003$  \\
\bottomrule
\end{tabular}\label{tab:l2_eocl2h1_p11}
\end{table}
\begin{table}[h]\caption{Example 1, estimated order of convergence (eoc) 
for $u_h$ and $p_h$ in $L^2(Q)$ for the linear model problem.} 
\centering
\begin{tabular}{rrrrrr}
\toprule
\#Dofs & $h$ & $\|u-u_h\|_{L^2(Q)}$ & eoc & $\|p-p_h\|_{L^2(Q)}$ & eoc \\
\midrule
$250$ & $2^{-2}$& $3.767e-2$    & $-$  & $4.146e-3$    & $-$   \\
$1,458$ & $2^{-3}$& $1.156e-2$    & $1.704$  &$1.160e-3$    & $1.837$  \\
$9,826$ & $2^{-4}$ & $3.009e-3$    & $1.942$ & $2.981e-4$   & $1.961$  \\
$71,874$ & $2^{-5}$ & $7.595e-4$    & $1.986$ & $7.515e-5$    & $1.988$ \\
$549,250$ &  $2^{-6}$ &$1.927e-4$    & $1.979$ & $1.950e-5$    & $1.947$ \\
$2,146,689$ & $2^{-7}$ &$4.948e-5$ & $1.961$ & $5.244e-6$ & $1.894$ \\
\bottomrule
\end{tabular}\label{tab:l2_eocl2l2_p11}
\end{table}
\begin{table}[h]\caption{Example 1, $J(u_h,z_h)$ and $|J(u_h,z_h)-J(u,z)|$
for the linear model problem, where 
$J(u,z)=9.53329e-2$.} 
\centering
\begin{tabular}{rrrrr}
\toprule
\#Dofs &  $h$ & $J(u_h, z_h)$ &  $|J(u_h, z_h)-J(u,z)|$ & eoc \\
\midrule
$250$ & $2^{-2}$& $1.04613e-1$      & $9.2801e-3$    & $-$   \\
$1,458$ & $2^{-3}$& $9.80559e-2$     &$2.7230e-3$    & $1.769$  \\
$9,826$ & $2^{-4}$ & $9.60214e-2$    & $6.8850e-4$    & $1.984$  \\
$71,874$ & $2^{-5}$ & $9.55024e-2$   & $1.6950e-4$    & $2.022$ \\
$549,250$ &  $2^{-6}$ &$9.53748e-2$   & $4.1900e-5$    & $2.016$ \\
$2,146,689$ &  $2^{-7}$ &$9.53433e-2$   & $1.0400e-5$    & $2.010$ \\
\bottomrule
\end{tabular}\label{tab:l2_eocobj_p11}
\end{table}

\subsubsection{An example with discontinuous target (Example 2)}
In the second example, the space-time domain $Q=(0,1)^3$ and the 
discontinuous target function 
\begin{equation*}
  u_d(x,t) =
  \begin{cases}
    &1\quad\text{ if }
    \sqrt{(x_1-\frac{1}{2})^2+(x_2-\frac{1}{2})^2+(t-\frac{1}{2})^2}
    \leq\frac{1}{4},\\
    &0\quad\textup{ else }
  \end{cases}
\end{equation*}
are considered. Further, we set homogeneous initial and boundary conditions for
the state equation, and homogeneous terminal and boundary conditions for the
adjoint equation. For the $L^2$-regularization parameter, we select
$\varrho=10^{-6}$.
Following the approach in \cite{OSHY18}, we have utilized a residual based
error indicator to drive our 
mesh refinements in order to resolve the discrete optimality system. The
space-time finite element solutions for the state and 
adjoint variables, as well as the time-dependent target are displayed in
Fig.~\ref{fig:ex2sol}. The control is reconstructed by a postprocessing step 
using piecewise constant ansatz functions, which is demonstrated in the last
column in Fig.~\ref{fig:ex2sol}. For a discussion and an error
  analysis of this postprocessing idea, we refer to \cite{CMAR04} in the case
  of elliptic PDE control. The adaptive mesh is illustrated in
  Fig.~\ref{fig:ex2mesh} at the $20$th refining step, which contains
  $2,080,493$ grid points, i.e., the total number of degrees of freedom for the coupled
state and adjoint equation is $4,160,986$.   

\begin{figure}[h]
  \centering
  \includegraphics[width=0.24\textwidth]{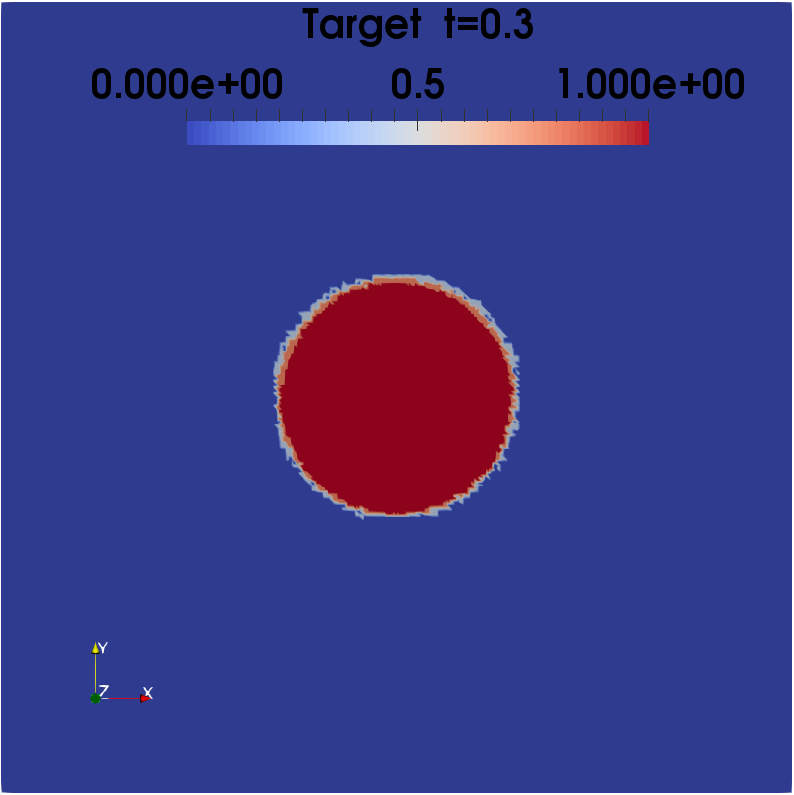}
  \includegraphics[width=0.24\textwidth]{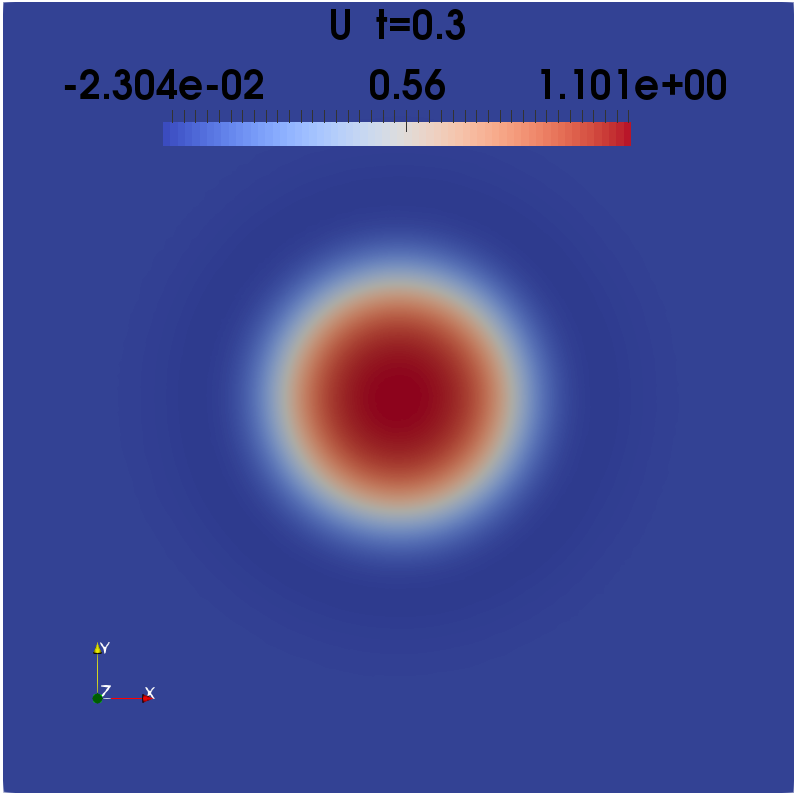}
  \includegraphics[width=0.24\textwidth]{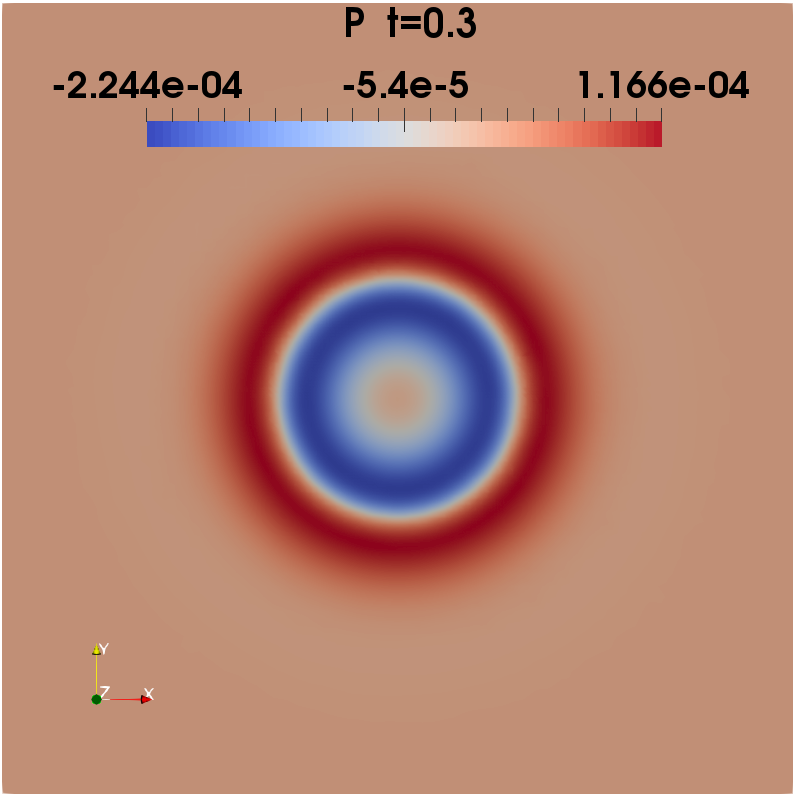}
  \includegraphics[width=0.24\textwidth]{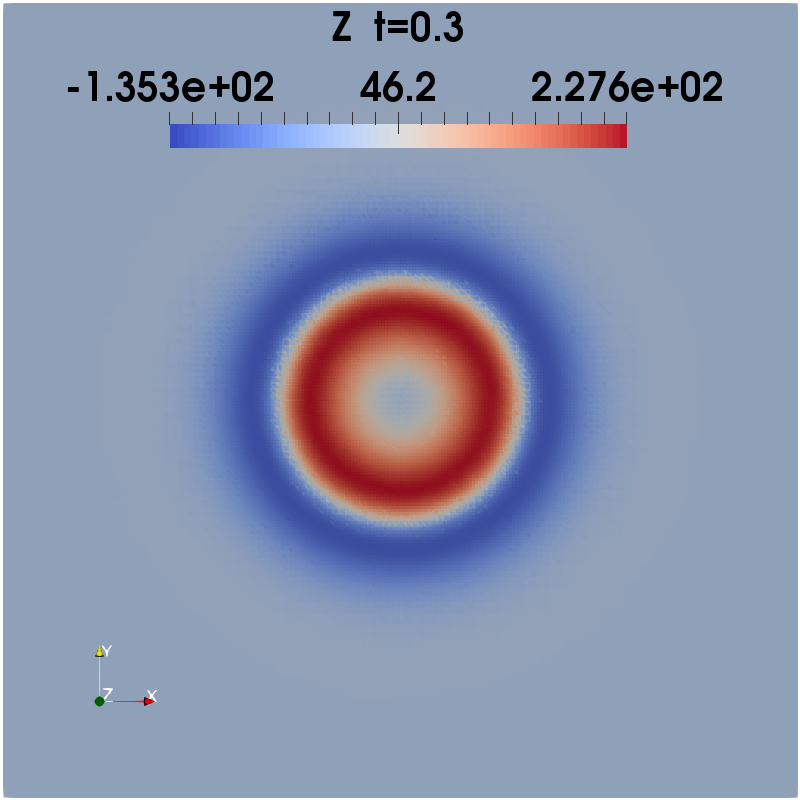}
  \includegraphics[width=0.24\textwidth]{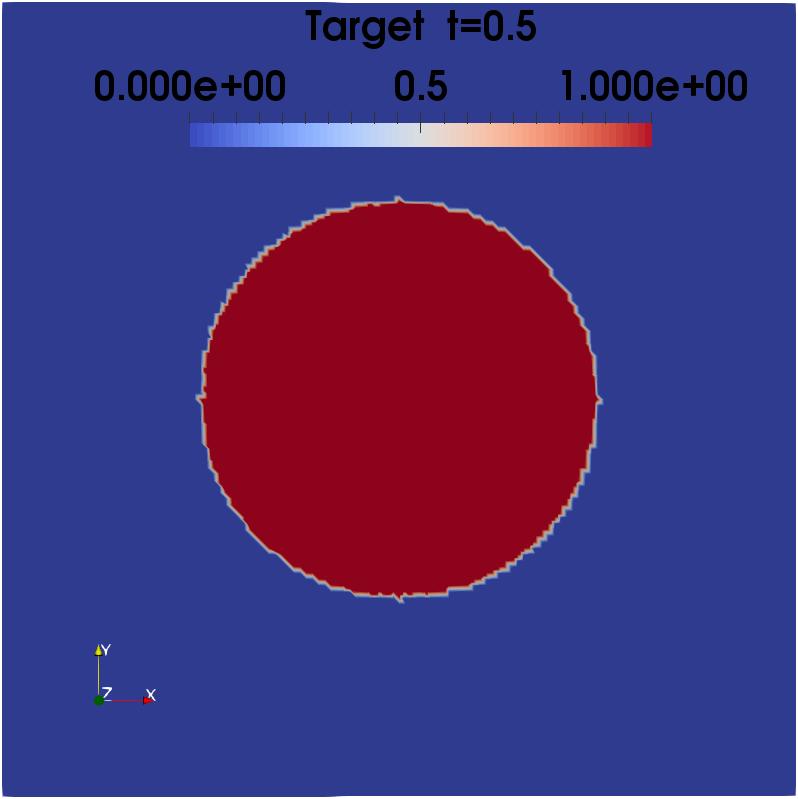}
  \includegraphics[width=0.24\textwidth]{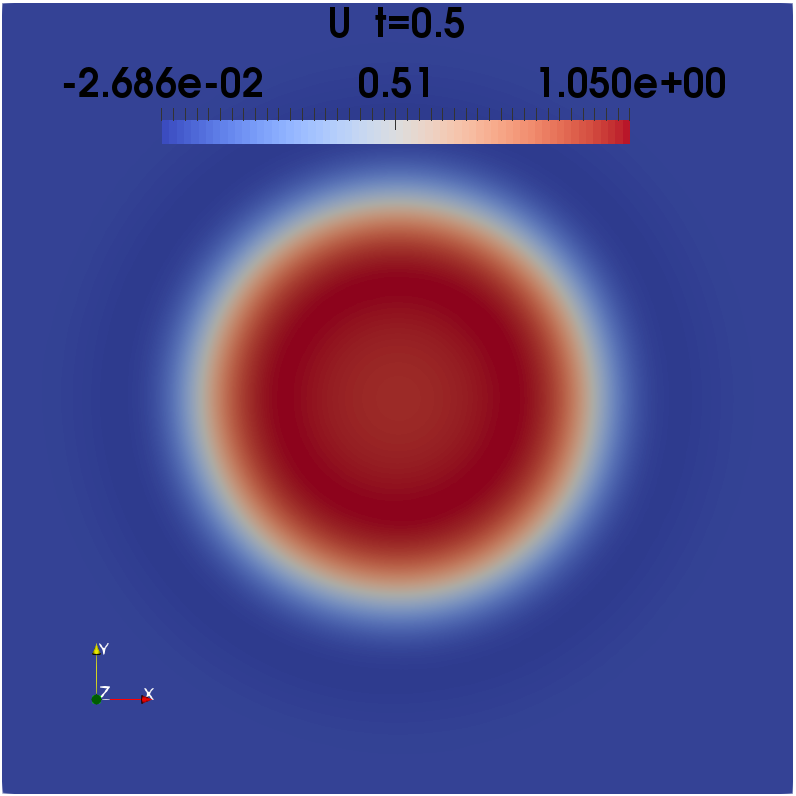}
  \includegraphics[width=0.24\textwidth]{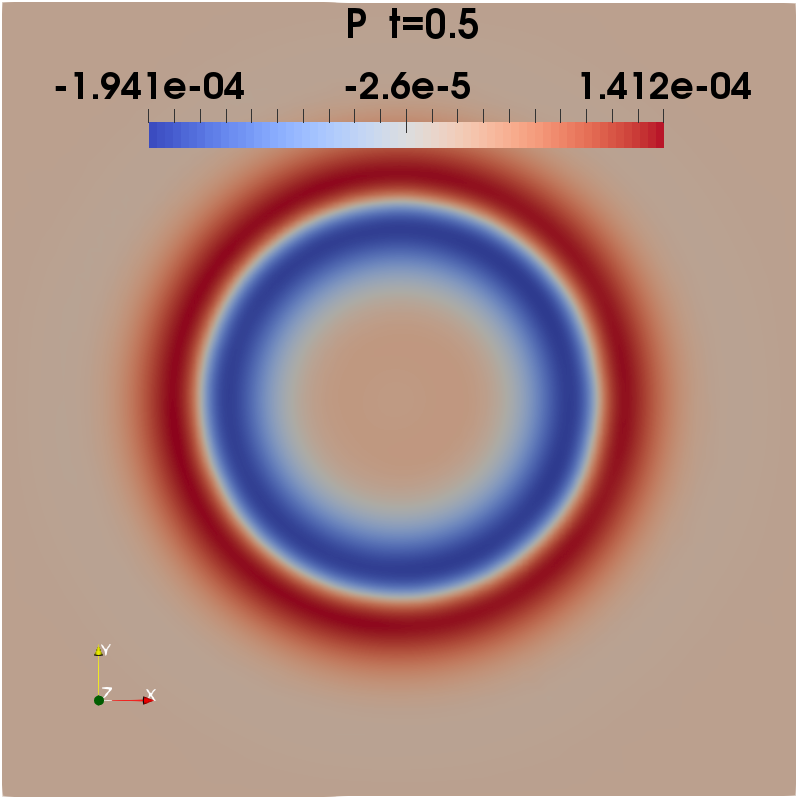}
  \includegraphics[width=0.24\textwidth]{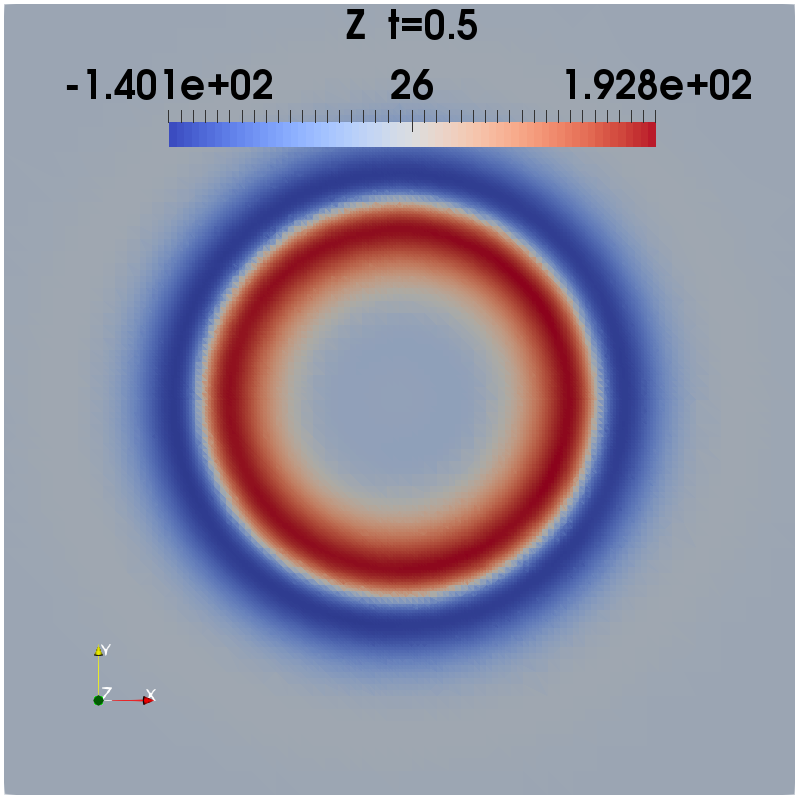}
  \includegraphics[width=0.24\textwidth]{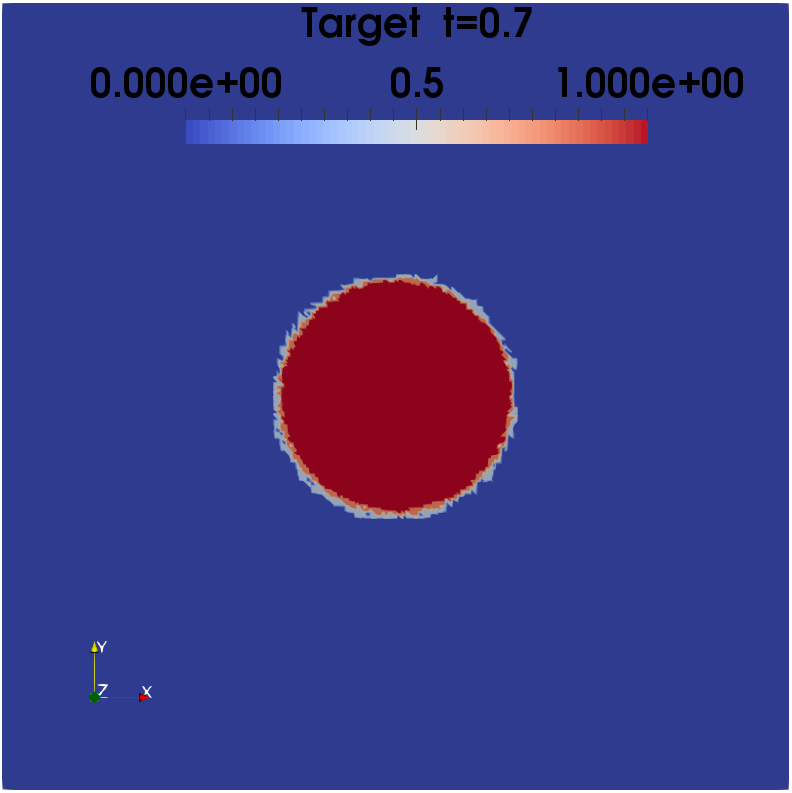}
  \includegraphics[width=0.24\textwidth]{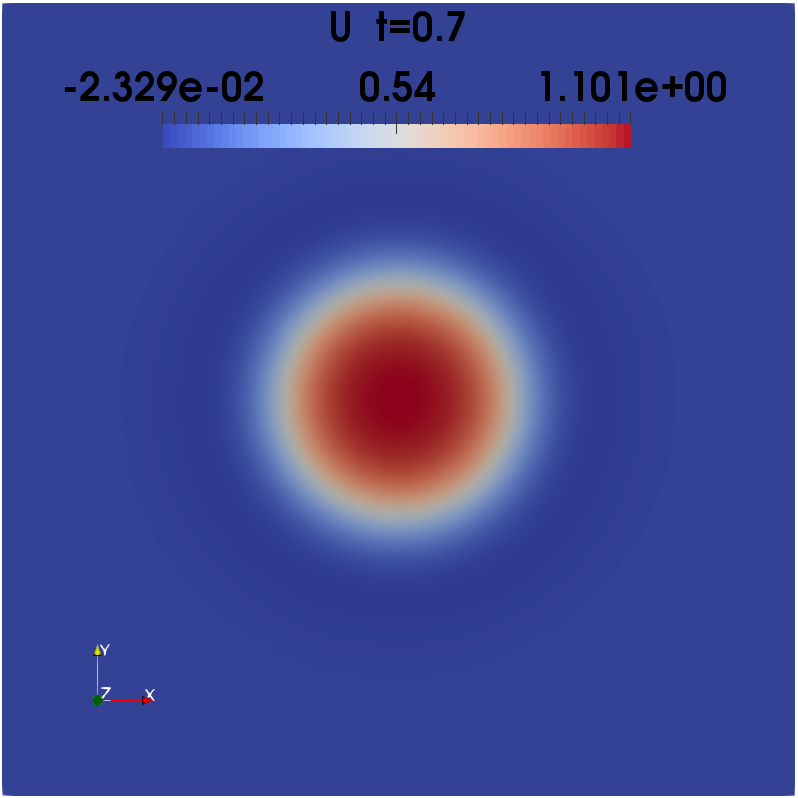}
  \includegraphics[width=0.24\textwidth]{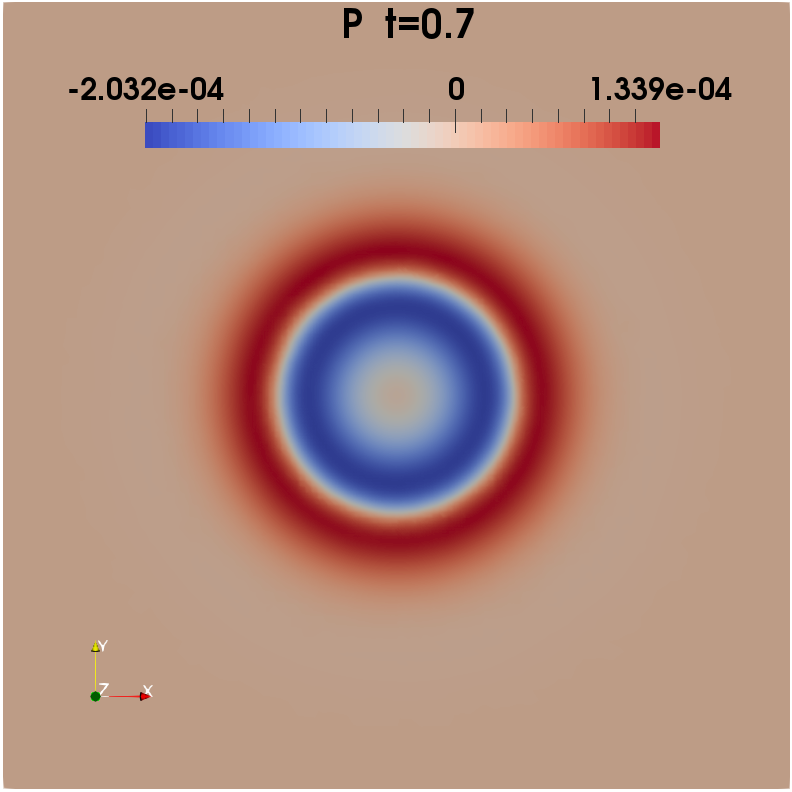}
  \includegraphics[width=0.24\textwidth]{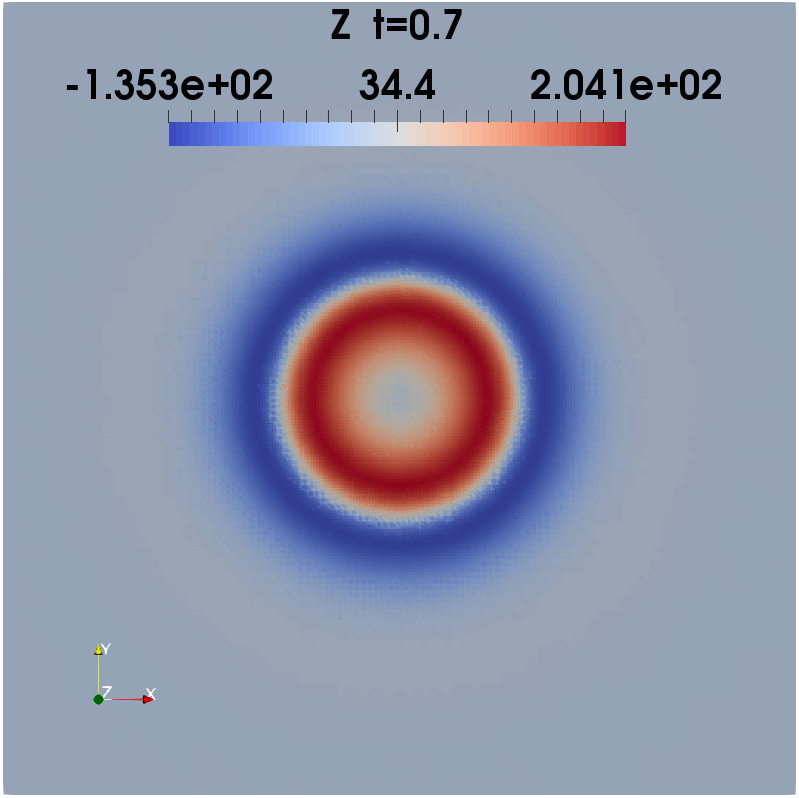}
  \caption{Example 2, numerical solutions of $u$, $p$, and $z$ for the linear model problem
    with a discontinuous target, at $t=0.3$,
    $0.5$, and $0.7$ (from top to bottom).}
  \label{fig:ex2sol}
\end{figure} 
\begin{figure}[h]
  \centering
  \includegraphics[width=0.24\textwidth]{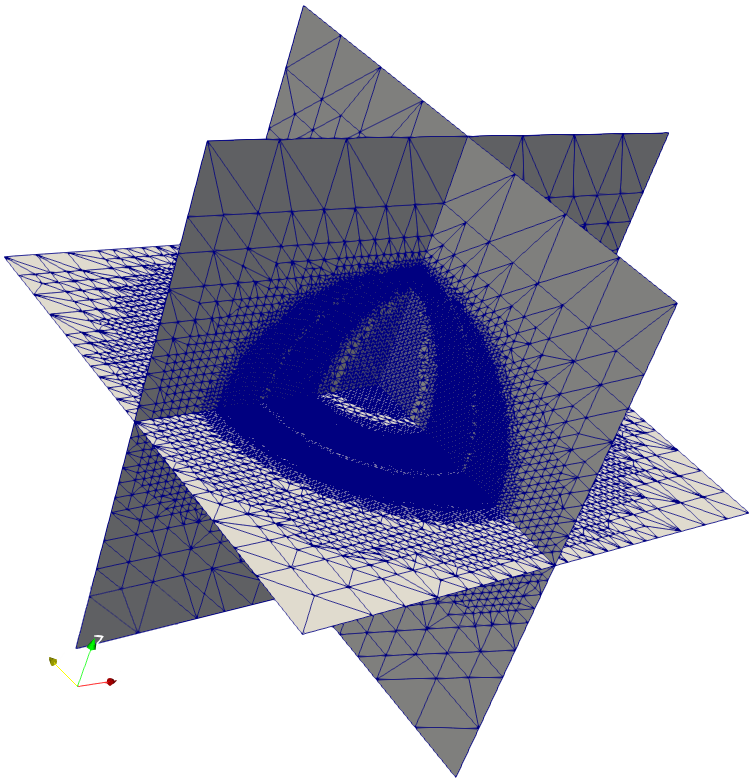}
  \includegraphics[width=0.24\textwidth]{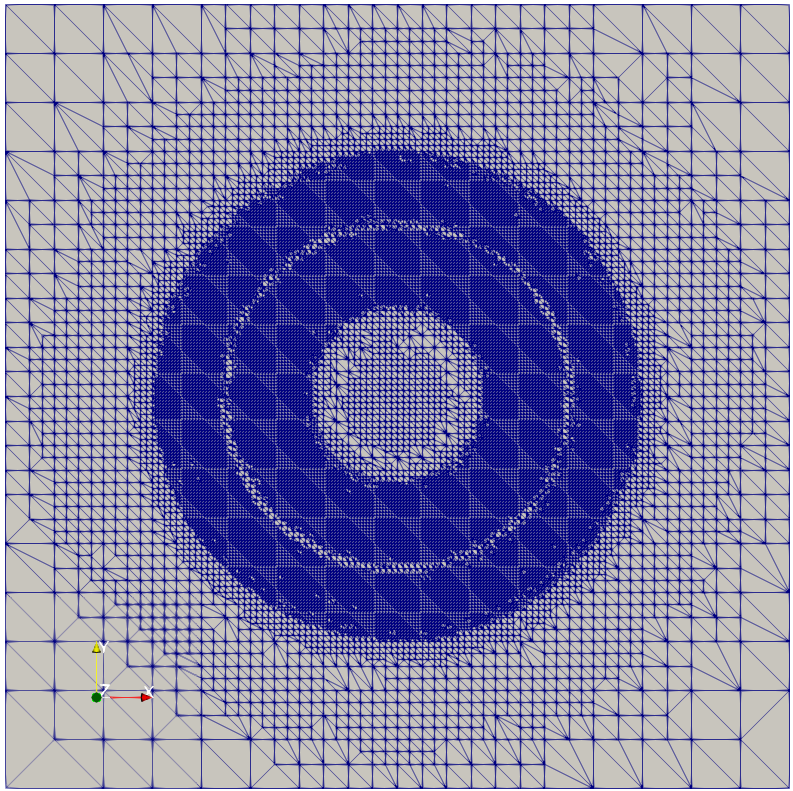}
  \includegraphics[width=0.24\textwidth]{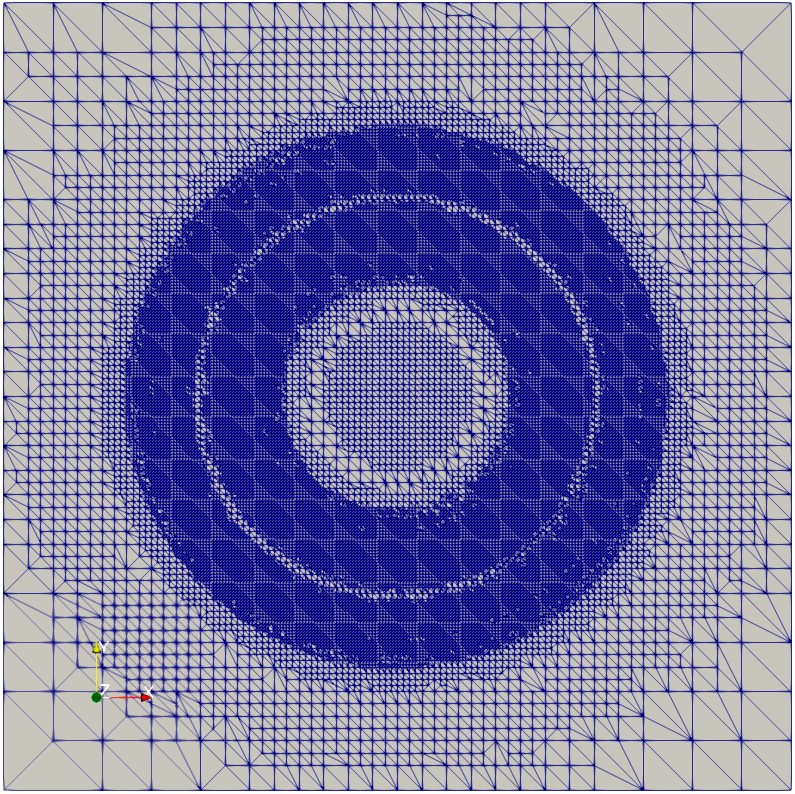}
  \includegraphics[width=0.24\textwidth]{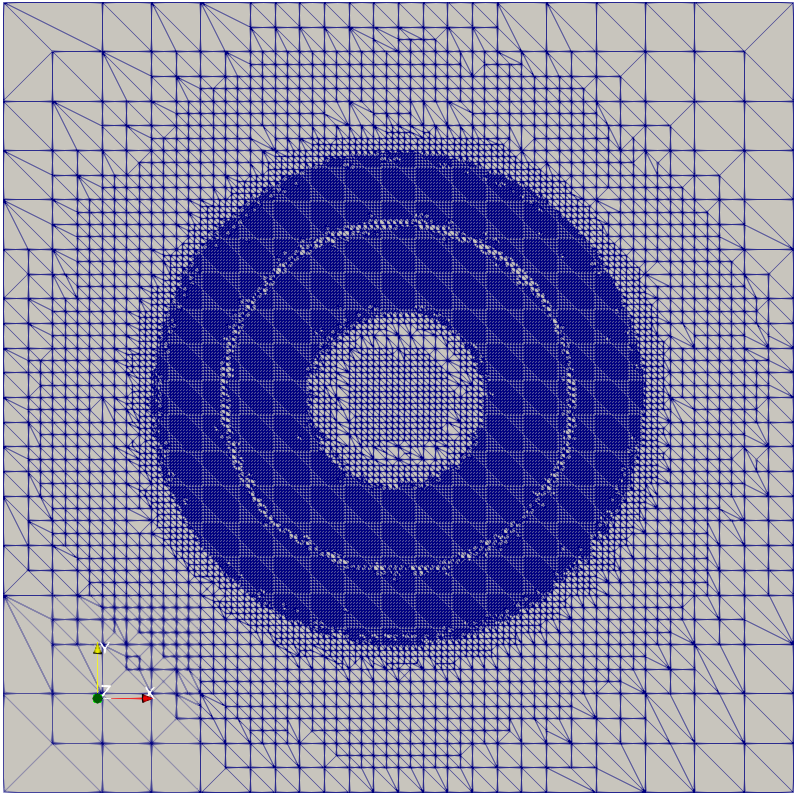}
  \caption{Example 2, adaptive mesh refinement at the $20$th step in space-time at
    $t=0.375$, $0.5$, and $0.625$ (from left to right), with a discontinuous target.}  
  \label{fig:ex2mesh}
\end{figure} 

\section{Control of a semilinear parabolic equation}
\label{sec:semilinearparabolicequations}

\subsection{The control problem}
Here, we consider the following optimal control problem for a semilinear 
heat equation to minimize
\begin{equation}\label{obj_nonlin}
J(u,z) := \frac{1}{2} \int_Q(u - u_d)^2 \, dx \, dt +  
\frac{\varrho}{2} \int_Q z^2 \, dx \, dt
\end{equation}
subject to 
\begin{equation} \label{eq_semilin}
\partial_tu - \Delta_x u + R(u) = z \mbox{ in } Q, \quad
\partial_n u = 0 \mbox{ in } \Sigma, \quad
u = 0 \mbox{ on } \Sigma_0 .
\end{equation}
The control $z$ is taken from the space $L^p(Q)$ with $p > 1+d/2$ to guarantee
existence and uniqueness of a bounded solution $u$ to
\eqref{eq_semilin}. We look for
the solution $u$ of \eqref{eq_semilin} in the space $Y \cap L^\infty(Q)$. 
Moreover, we here  assume $u_d \in L^p(Q)$ to ensure later that the adjoint 
state belongs to $L^\infty(Q)$.

The function $R: \mathbb{R} \to \mathbb{R}$ is a $C^2$-function with locally 
Lipschitz second-order derivative, i.e., for all $M > 0$, there is some 
$L(M) > 0$ such that
\[
|R''(v_1) - R''(v_2)| \le L(M) \, |v_1-v_2| \quad \forall v_i 
\mbox{ with  } |v_i|\le M, i = 1,2.
\]
Moreover, we require the existence of some $c_0 \in \mathbb{R}$ 
(possibly negative) such that
\[
R'(u) \ge c_0 \quad \forall u \in \mathbb{R}.
\]
An important particular case for the reaction term $R$ is 
\[
R(u) = (u - u_1)(u - u_2)(u - u_3),
\]
where real numbers $u_1\le u_2\le u_3$ are given. Obviously, this function 
obeys the assumptions above. It is used for the Schl\"ogl and 
FitzHugh-Nagumo equations. The following theorem on the solvability 
of \eqref{eq_semilin} is known:

\begin{theorem}[\cite{ECCRFT13}]
Let $\Omega \subset {\mathbb{R}}^d$, $d \le 3$,  be a bounded Lipschitz 
domain and let $R$ satisfy the conditions stated above. Then for all 
controls $z \in L^p(Q)$ with $p >1+d/2$, the equation \eqref{eq_semilin} 
has a unique solution $u \in Y \cap L^\infty(Q)$. The control-to-state 
mapping $G: z \mapsto u$ is of class $C^2$.
\end{theorem}

\noindent
The optimal control problem \eqref{obj_nonlin}-\eqref{eq_semilin} can be 
expressed in the reduced form
\[
\min_{z \in L^p(Q)} f(z) :=\frac{1}{2}\int_Q(G(z) - u_d)^2 \, dx \, dt +  
\frac{\varrho}{2} \int_Q z^2 \, dx \, dt.
\]
Compared with the quadratic optimal control problems of the former 
sections, several new difficulties occur. 

Though we have $f(u) \to \infty$ if $\|z\|_{L^2(Q)} \to \infty$, the existence 
of an optimal control $\bar z$ cannot be proved by standard weak compactness 
techniques. It was recently shown by fairly deep arguments that at least one 
(globally) optimal control exists for the unconstrained case, 
cf. \cite{CCMMAR2016,ECCRFT2018}. This justifies to consider the optimal 
control problem without control constraints. 

Moreover, even though the functional $J$ is convex, the reduced functional 
$f$ is not in general convex, because $G$ is nonlinear. Therefore, the 
optimality system is not sufficient for (local or global) optimality of 
its solution. We might also have different global or local solutions of the 
optimal control problem. They are even not guaranteed to be locally unique. 
To overcome these difficulties, we assume that a given reference solution 
of the optimality system satisfies a second-order sufficient optimality 
condition.

Finally, we should mention that the mapping $G$ is not in general 
differentiable or twice differentiable in the Hilbert space $L^2(Q)$. 

To allow for a Hilbert space setting as in the previous sections, 
we will proceed as follows: In 
the infinite-dimensional setting, we will apply a Newton method 
(sequential quadratic programming (SQP) method) that solves
the problem by a sequence of quadratic optimal control problems 
that are posed in Hilbert space. 
These are solved by the methods of the former sections. For convergence 
of this method, the second-order sufficient optimality condition is 
needed again. 

\subsection{The Lagrange-Newton-SQP method}
Let us define the Lagrangian $\mathcal{L}:(Y\cap L^\infty(Q))^2 \times
L^p(Q) \to \mathbb{R}$, $p > 1+d/2$, by
\[
\mathcal{L}(u,p,z) = J(u,z) - \int_Q \Big(
\partial_t u \, p + \nabla u \cdot \nabla p + R(u) \, p - z \, p 
\Big) \, dx \,dt.
\]
Then the SQP method proceeds as follows: An arbitrary triplet 
$(u_0,p_0,z_0) \in (Y\cap L^\infty(Q))^2 \times L^p(Q)$ is taken as 
initial iterate. For a given iterate $(u_n,p_n,z_n)$
the following quadratic optimal control problem ($QP_n$) is considered:
\[
\tag{$QP_n$}
\left. \hspace*{-1mm} 
\begin{array}{l}
\min J'(u_n,z_n)(u - u_n,z - z_n) + \displaystyle \frac{1}{2} \mathcal{L}_{u,z}''(u_n,p_n,z_n)(u-u_n,z-z_n)^2\\[2ex]
\qquad \mbox{subject to the linearized equation}\\[2ex]
\begin{array}{rcll}
\partial_tu - \Delta_x u + R(u_n) + R'(u_n)(u - u_n) & = & z \quad
& \mbox{in} \; Q, \\[1mm]
\partial_n u & = & 0 & \mbox{on} \; \Sigma, \\[1mm]
u & = & 0 & \mbox{on} \; \Sigma_0 .
\end{array} \end{array}
\right \}
\]
The next iterate $z_{n+1}$ is the optimal control of ($QP_n$), 
provided that it exists, $u_{n+1}$ is the associated
optimal state and $p_{n+1}$ is the associated adjoint state.

The numerical treatment of ($QP_n$) requires the solution of the 
optimality system
\[
\tag{$OS_n$}
\left. \hspace*{-4mm}
\begin{array}{lcl}
\partial_tu - \Delta_x u + R(u_n) + R'(u_n)(u - u_n) + 
\displaystyle \frac{1}{\varrho} \, p&=&0, \; u(0) = u_0, \\[2ex]
-\partial_tp- \Delta_x p + R'(u_n)p + p_n R''(u_n)(u - u_n) &=& u_n - u_d, \;
p(T) = 0
\end{array}
\right \}
\]
subject to homogeneous Neumann conditions. Then we have $u_{n+1} = u$ and 
$p_{n+1} = p$; the new control iterate is $z_{n+1} = - p_{n+1}/\varrho.$ This 
iteration method is called Lagrange-Newton method, because it comes from 
linearizing the whole optimality system. In contrast to this, the SQP 
method would not linearize the state equation. We refer for a general 
exposition to \cite{WA1990,WA1992}, for the convergence analysis for 
semilinear parabolic equations in an $L^\infty$-setting to 
\cite{FT1999,FT10} and the discussion in a Hilbert space setting 
to \cite{MHRPMUSU2009}.

To make all iterates well defined, we will invoke a second-order sufficient 
optimality condition.

\subsection{Second-order sufficient optimality condition and 
convergence of the Lagrange-Newton method}
Let $(\bar u,\bar p,\bar z)$ be a fixed triplet that satisfies the optimality 
system for the optimal control problem \eqref{obj_nonlin}-\eqref{eq_semilin}. 
We say that the triplet fulfils the {\em second-order sufficient optimality 
condition}, if it enjoys the following property of positive definiteness: 
A number $\sigma > 0$ exists, such that
\begin{equation}\label{SSC}
\mathcal{L}''_{u,z}(\bar u,\bar p,\bar z)(u,z)^2 \ge \sigma \, \|z\|_{L^2(Q)}^2
\end{equation}
holds for all pairs $(u,z)$ that obey the linearized equation
\[
\partial_t u - \Delta_x u + R'(\bar u) u = z \; \mbox{in} \; Q, \quad
\partial_n u = 0 \; \mbox{on} \; \Sigma, \quad u=0 \; \mbox{on} \; \Sigma_0 .
\]
It is known that this condition is sufficient for the local optimality of 
$\bar z$ in the sense of $L^2(Q)$, \cite{ECCRFT2015}. Moreover, $\bar z$ 
is unique in a certain $L^2(\Omega)$-neighborhood. It is not in general 
possible to verify this condition by numerical methods. As usual, it is just
a theoretical basis for the analysis.

In the case of the optimal control problem 
\eqref{obj_nonlin}-\eqref{eq_semilin}, the derivative $\mathcal{L}''$
has the form
\[
\mathcal{L}''_{u,z}(\bar u,\bar p,\bar z)(u,z)^2 = 
\|u\|_{L^2(Q)}^2 + \varrho \, \|z\|_{L^2(Q)}^2 - 
\int_Q \bar p \, R''(\bar u) \, u^2 \, dx \, dt.
\]
Therefore, the second-order sufficient optimality condition is satisfied 
in particular, if 
\[
1-\bar p(x,t) \, R''(\bar u(x,t)) \ge 
\sigma \quad \mbox{ for  a.a. } (x,t) \in Q.
\]
For instance, this holds, if $\bar p$ is small, i.e., $\bar u$ is close to
$u_d$. The convergence theorem below is based on the second-order sufficient 
condition. Since it needs box constraints on the control to ensure that 
all iterates belong to a bounded set of $L^\infty(Q)$, we invoke the 
following result:

\begin{lemma} \label{L:bounded}
There is at least one optimal control of  problem 
\eqref{obj_nonlin}-\eqref{eq_semilin} that belongs to $L^\infty(Q)$.
\end{lemma}

\begin{proof}
We rely on Theorem 2.4 of \cite{ECCRFT2018} that guarantees the existence 
of at least one optimal control $\bar u$ of 
\eqref{obj_nonlin}-\eqref{eq_semilin} that is bounded in 
$L^\infty(0,T;L^2(\Omega))$. Therefore, the search for a control can be 
restricted to the set $\{u \in L^\infty(0,T;L^2(\Omega)): 
\|u\|_{L^\infty(0,T;L^2(\Omega))} \le R\}$, where 
$R =  \|\bar u\|_{L^\infty(0,T;L^2(\Omega))}$. Thanks to the existence and 
regularity Theorem 2.1 of \cite{ECCRFT2018},
all associated states are bounded in $L^\infty(Q)$. The right-hand side 
of the associated adjoint equation is $\bar u - u_d \in L^p(Q)$, since 
we assumed $u_d \in L^p(Q)$. Therefore, the adjoint state $\bar p$
is also a function of $L^\infty(Q)$. This property transfers to $\bar u$ 
by the gradient equation $\bar u = - \bar p / \varrho$, hence the 
existence of an optimal control in $L^\infty(Q)$ is proved.
\end{proof}

\noindent
The following result is known for the convergence of the Lagrange-Newton-SQP 
method for problems of the type \eqref{obj_nonlin}-\eqref{eq_semilin} with 
additional box constraints
\begin{equation} \label{boxcon}
a \le u(x,t) \le b \quad \mbox{for a.a. } (x,t) \in Q,
\end{equation}
where $-\infty < a < b < \infty$:

\begin{theorem}[Convergence of the Lagrange-Newton method]
Let $(\bar u,\bar p,\bar z)$ be a triplet that satisfies the optimality 
system for the optimal control problem
\eqref{obj_nonlin}-\eqref{eq_semilin} with additional box constraints
\eqref{boxcon}. Assume that this triplet satisfies the second-order
sufficient condition \eqref{SSC}.  

Then the Lagrange-Newton method converges locally and quadratic to 
$(\bar u,\bar p,\bar z)$. This means the following: There exist 
$r > 0, \, C > 0$ such that, if the initial iterate 
$(u_0,p_0,z_0) \in(Y\cap L^\infty(Q))^2\times L^\infty(Q)$  satisfies
\[
\|(u_0,p_0,z_0)-(\bar u,\bar p,\bar z)\|_{L^\infty(Q)^3}\le r,
\]
then the system ($OS_n$) is uniquely solvable for all $n \ge 0$.  
The iterates fulfill 
\[
\|(u_n,p_n,z_n)-(\bar u,\bar p,\bar z)\|_{L^\infty(Q)^3}\le r,\quad \forall n \ge 1,
\]
and 
\[
\|(u_{n+1},p_{n+1},z_{n+1})-(\bar u,\bar p,\bar z)\|_{L^\infty(Q)^3} \le 
C \, \|(u_n,p_n,z_n)-(\bar u,\bar p,\bar z)\|_{L^\infty(Q)^3}^2
\]
for all $n \ge 0$.
\end{theorem}

\noindent
The proof is a bit delicate, because the $L^2(Q)$-norm appears in the 
second-order sufficient condition, while the differentiability of 
$\mathcal{L}$ is considered in $L^\infty(Q)$. This is the so-called 
two-norm discrepancy. For a proof with additional control constraints 
$a \le z \le b$, we refer to \cite{FT1999}. For problems, where the 
two-norm-discrepancy does not appear, a convergence analysis for the 
Newton method in the unconstrained method is given in \cite{DP2011} 
and, in the context of optimal control, in \cite{MHRPMUSU2009}.

The Lagrange-Newton-SQP method differs from the Lagrange-Newton me\-thod 
by adding the box constraints \eqref{boxcon} to the subproblems (QP$_n$) 
and the associated projection formula to the optimality system
(QP$_n$). In our implementation, we formally added box constraints with 
$a = -10^6$, $b = 10^6$, justified by Lemma \ref{L:bounded}. These bounds 
became never active. Hence, the Lagrange-Newton-SQP method was equivalent 
to the Lagrange-Newton method described in the last subsection. 

\subsection{Numerical experiments}
For the nonlinear first order necessary optimality system, when
considering a nonlinear reaction term in the state equation or in 
the presence of box constraints on the control, we apply a (semismooth) 
Newton method in the outer iteration. Usually, we need about $2-8$ 
iterations to reach a precision of 
$10^{-8}$ 
of the relative residual error. 
Inside each Newton iteration, we apply the same algebraic multigrid 
preconditioned GMRES solver as used for the linear system. However, the
performance study and the development of robust and efficient solvers are
beyond the scope of this work; we will investigate them somewhere else.  

\subsubsection{Unconstrained control with explicitly known optimal 
solution (Example 3)}
Following the approach provided in \cite{FT10}, we construct an exact 
solution for the following modified optimal control problem of a 
semilinear parabolic equation:  
\[
\min J(u,z) := \frac{1}{2} \int_Q(u - u_d)^2 \, dx \, dt + 
\frac{\varrho}{2} \int_Q z^2 \, dx \, dt
\]
subject to 
\[
\partial_t u - \Delta_x u + R(u) = z + e_u \; \mbox{in } Q, \quad 
u = 0 \; \mbox{on } \Sigma, \quad
u = 0 \; \mbox{on } \Sigma_0,
\]
where the function $e_u$ is defined such that a desired pair $u$, $z$ 
is optimal. The first order necessary optimality system for this 
semilinear model problem reads as follows (see \cite{FT10}):
It is composed of the state equation
\begin{equation}\label{optimality_system_p31}
\partial_t u-\Delta_x u  + R( u ) + \frac{1}{\varrho} p
= e_u \textup{ in } Q,\quad 
u = 0 \textup{ on } \Sigma, \quad u  =  0 \textup{ on } \Sigma_0,
\end{equation}
and the adjoint equation
\begin{equation}\label{optimality_system_p312}
-\partial_t p-\Delta_x p + R'(u)p = u-u_d   \textup{ in } Q,\quad 
p = 0   \textup{ on } \Sigma,\quad
p = 0   \textup{ on } \Sigma_T.
\end{equation}
The desired solutions of the optimality system are given by 
\begin{equation*}
\begin{aligned}
&u(x,t)=\sin(\pi x_1)\sin(\pi x_2)\left(at^2+bt\right),\\
&p(x,t)=-\varrho\sin(\pi x_1)\sin(\pi x_2)
\left( 2\pi^2 a t^2+(2\pi^2b + 2a)t+b\right),\\
&z(x,t)=\sin(\pi x_1)\sin(\pi x_2)\left(2 \pi^2 a t^2+(2\pi^2b + 2a)t+b
\right),\\
&R(u)=u(u-0.25)(u+1),\\
\end{aligned}
\end{equation*}
where $a=-\frac{2\pi^2+1}{2\pi^2+2}$, $b=1$, and 
$\varrho=10^{-4}$. 
It is easy to
see that the state obeys homogeneous initial and boundary conditions, and
the adjoint state satisfies homogeneous terminal and boundary conditions;
see the illustration in Fig. \ref{fig:sol_p31}. The functions $e_u$
and $u_d$ are computed by inserting the above solutions to the system  
\eqref{optimality_system_p31} and \eqref{optimality_system_p312}.  

\begin{figure}[h]
  \centering
    \includegraphics[width=0.4\textwidth]{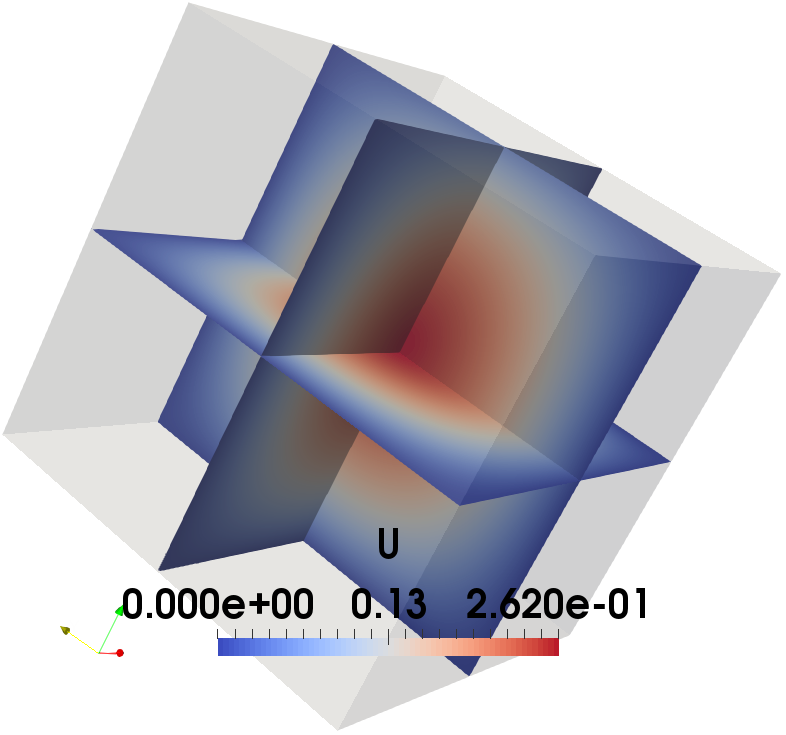}
    \includegraphics[width=0.4\textwidth]{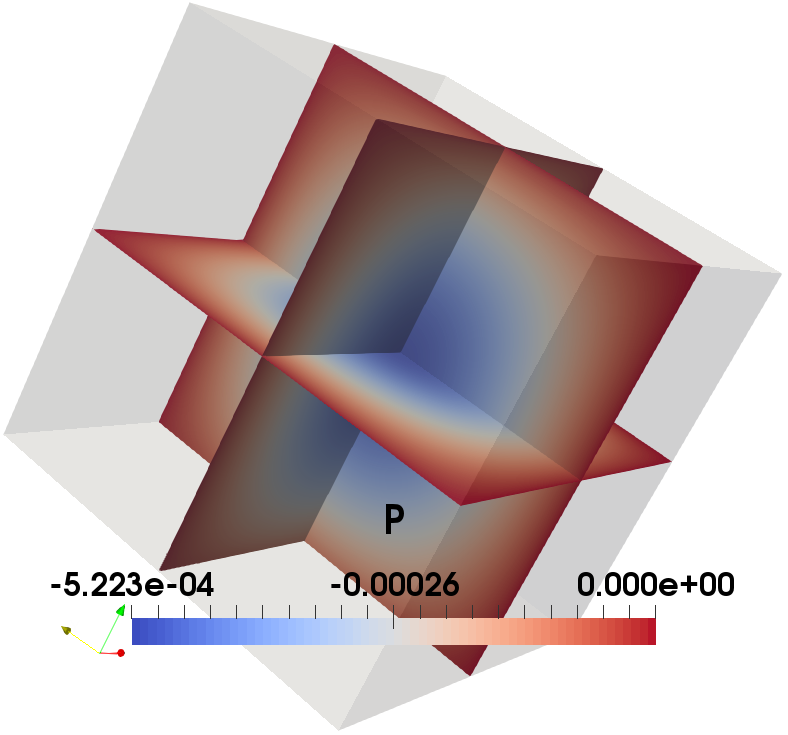}
    \includegraphics[width=0.4\textwidth]{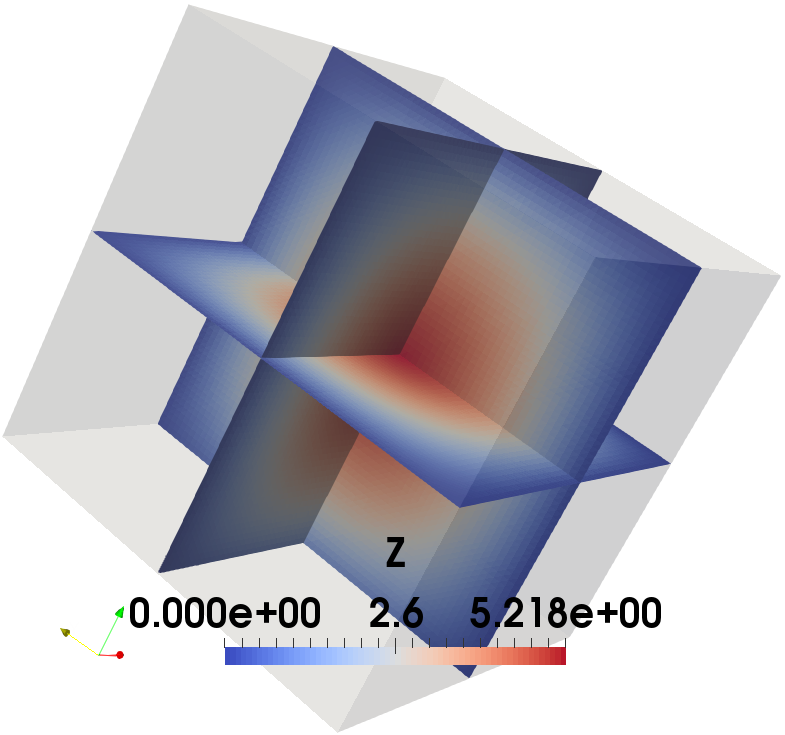}  
    \includegraphics[width=0.4\textwidth]{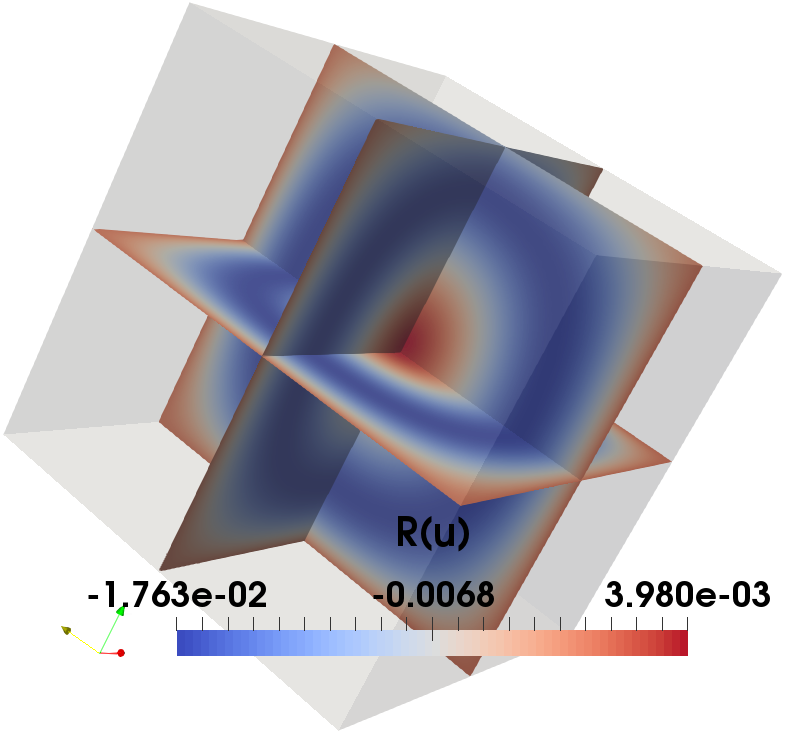}  
  \caption{Example 3, numerical solutions of $u$, $p$, and $z$, and 
    $R(u)$ for the semilinear model problem
    without control constraint.}
    \label{fig:sol_p31}
\end{figure} 

The estimated order of convergence is given in Tables
\ref{tab:l2_eocl2h1_p31}-\ref{tab:l2_eocobj_p31}. As we expect, we
observe optimal convergence rates in $Y=L^2(0,T; H^1_0(\Omega))$. However, the
convergence rate is not that good as expected in $L^2(Q)$ for this particular
example. This requires further investigation. Moreover, we see almost
second-order convergence for the objective functional. 
 
\begin{table}[h]
\caption{Example 3, estimated order of convergence (eoc) of $u_h, p_h$ in
    $Y$ for the semilinear model problem without
    control constraint.} 
\centering
\begin{tabular}{rrrrrr}
\toprule
\#Dofs & $h$ & $\|u-u_h\|_Y$ & eoc & $\|p-p_h\|_Y$ & eoc  \\
\midrule
$250$ & $2^{-2}$& $2.344e-1$    & $-$  & $8.136e-4$    & $-$       \\
$1,458$ & $2^{-3}$& $1.159e-1$    & $1.017$  &$2.795e-4$    & $1.541$  \\
$9,826$ & $2^{-4}$ & $5.690e-2$    & $1.026$ & $1.193e-4$    & $1.228$  \\
$71,874$ & $2^{-5}$ & $2.815e-2$    & $1.015$ & $5.691e-5$    & $1.068$   \\
$549,250$ & $2^{-6}$ &$1.400e-2$    & $1.008$ & $2.801e-5$    & $1.023$ \\
$2,146,689$ & $2^{-7}$ &$6.982e-3$    & $1.003$ & $1.394e-5$    & $1.007$ \\
\bottomrule
\end{tabular}\label{tab:l2_eocl2h1_p31}
\end{table}

\begin{table}[h]
\caption{Example 3, estimated order of convergence of $u_h, p_h$ in
    $L^2(Q)$ for the semilinear model problem without control constraint.}  
\centering
\begin{tabular}{rrrrrr}
\toprule
\#Dofs &  $h$ & $\|u-u_h\|_{L^2(Q)}$ & eoc & $\|p-p_h\|_{L^2(Q)}$ & eoc \\
\midrule
$250$ & $2^{-2}$& $1.315e-2$    & $-$  & $9.435e-5$    & $-$   \\
$1,458$ & $2^{-3}$& $3.692e-3$    & $1.833$  &$2.104e-5$    & $2.165$  \\
$9,826$ & $2^{-4}$ & $1.008e-3$    & $1.873$ & $4.770e-6$   & $2.141$  \\
$71,874$ & $2^{-5}$ & $2.621e-4$    & $1.943$ & $1.213e-6$    & $1.976$ \\
$549,250$ &  $2^{-6}$ &$7.218e-5$    & $1.861$ & $3.542e-7$    & $1.775$ \\
$2,146,689$ &  $2^{-7}$ &$3.180e-5$    & $1.183$ & $1.417e-7$    & $1.321$ \\
\bottomrule
\end{tabular}\label{tab:l2_eocl2l2_p31}
\end{table}

\begin{table}[h]
\caption{Example 3, $J(u_h,z_h)$ and $|J(u_h,z_h)-J(u,z)|$
for the semilinear model problem without control constraint, 
where $J(u,z)=1.98767e-4$.} 
\centering
\begin{tabular}{rrrrr}
\toprule
\#Dofs &  $h$ & $J(u_h, z_h)$ &  $|J(u_h, z_h)-J(u,z)|$ & eoc   \\
\midrule
$250$ & $2^{-2}$& $4.60861e-4$      & $2.6209e-4$    & $-$   \\
$1,458$ & $2^{-3}$& $2.37900e-4$     &$3.9133e-5$    & $2.744$  \\
$9,826$ & $2^{-4}$ & $2.06470e-4$    & $7.7030e-6$    & $2.345$  \\
$71,874$ & $2^{-5}$ & $2.00532e-4$   & $1.7650e-6$    & $2.126$ \\
$549,250$ &  $2^{-6}$ &$1.99206e-4$   & $4.3900e-7$    & $2.007$ \\
$2,146,689$ &  $2^{-7}$ &$1.98887e-4$   & $1.2000e-7$    & $1.871$ \\
\bottomrule
\end{tabular}\label{tab:l2_eocobj_p31}
\end{table}

\subsubsection{Box constrained control with explicitly known 
optimal solution (Example 4)}
In this example, we minimize
\[
J(u,z):=\frac{1}{2}\int_Q (u-u_d)^2 \, dx \, dt + 
\frac{\varrho}{2} \int_Q z^2 \, dx \, dt + \int_Q e_z \, z \, dx \, dt
\]
subject to 
\[
\partial_t u - \Delta_x u + R(u) = z + e_u \textup{ in } Q, \quad
u = 0 \textup{ on } \Sigma,\quad
u = 0 \textup{ on } \Sigma_0,
\]
and 
\[
a \leq z(x,t) \leq b \quad \textup{ for a.a. } (x,t) \in Q. 
\]
For this optimal control problem, we compute the functions $e_u$ and $e_z$
such that the desired solutions $u$, $p$ and $z$ satisfy the first order 
necessary optimality conditions. This system consists of the state equation
\[
\partial_t u  - \Delta_x u   + R(u)  = z + e_u \text{ in } Q,\quad
    u    =     0 \text{ on }  \Sigma, \quad
    u   =    0  \text{ on }  \Sigma_0,
\]
the adjoint equation
\[
    -\partial_t p - \Delta_x p  + R'(u)p  =  u - u_d \text{ in } Q,\quad
    p  =   0 \text{ on }\Sigma,\quad
    p  =   0 \text{ on } \Sigma_T,
\]
and the gradient equation
\[
z = {\mathbf P}_{[a,b]} \left(-\frac{1}{\varrho}\left(p  + e_z\right) 
\right) \text{ in } Q.
\]
The projection formula is equivalent to the variational inequality
\begin{equation}\label{eq:varinq}
\int_Q (p+\varrho z + e_z)(\hat{z}-z) \, dx \, dt \geq 0 
\textup{ for all } \hat{z}\in [ a , b ],
\end{equation}
for more details, see \cite{FT10}.

We now prescribe the solutions of the optimality system as follows:
\begin{equation*}
\begin{aligned}
  &u(x,t)=\sin(\pi x_1)\sin(\pi x_2)\left(ct^2+dt\right),\\
  &p(x,t)=-\varrho\sin(\pi x_1)\sin(\pi x_2)
\left( 2\pi^2 c t^2+(2\pi^2d + 2c)t+d\right),\\
  &z(x_1, x_2 , t)=
  \begin{cases}
    -1 &\textup { if }   0 \leq x_2 \leq -x_1 + 1/2  \textup{ and } 0\leq x_1\leq
    1/2, \\
    1  &\textup { if }  - x_1 +3/2 \leq x_2 \leq 1 \textup{ and  }  1/2 \leq
    x_1\leq 1, \\
    2x_1 + 2 x_2 - 2 &\textup { else },\\
  \end{cases}\\
\end{aligned}
\end{equation*}
where $c=-\frac{2\pi^2+1}{2\pi^2+2}$ and $d=1$. As nonlinearity, we fix
$R(u)=u(u-0.25)(u+1)$. For the constraints, we use the bounds 
$a=-1$ and $b=1$, and we set $\varrho=0.001$ as regularization parameter. The
constructed solutions $u$ and $p$ fulfill the 
initial/terminal and boundary conditions for the state and adjoint; see the
numerical solutions in Fig.~\ref{fig:numsol_p33} for an illustration. 

\begin{figure}[h]
  \centering
  \includegraphics[width=0.32\textwidth]{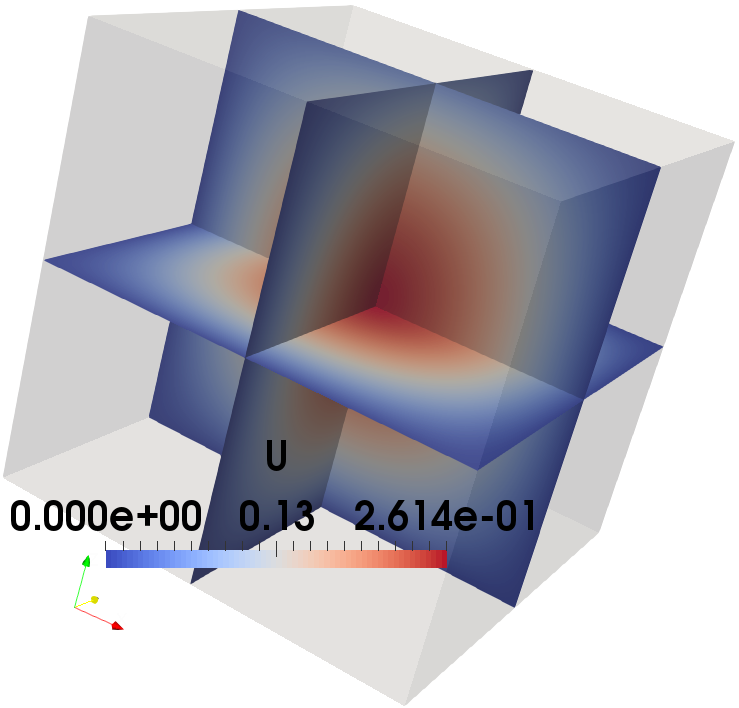}
  \includegraphics[width=0.32\textwidth]{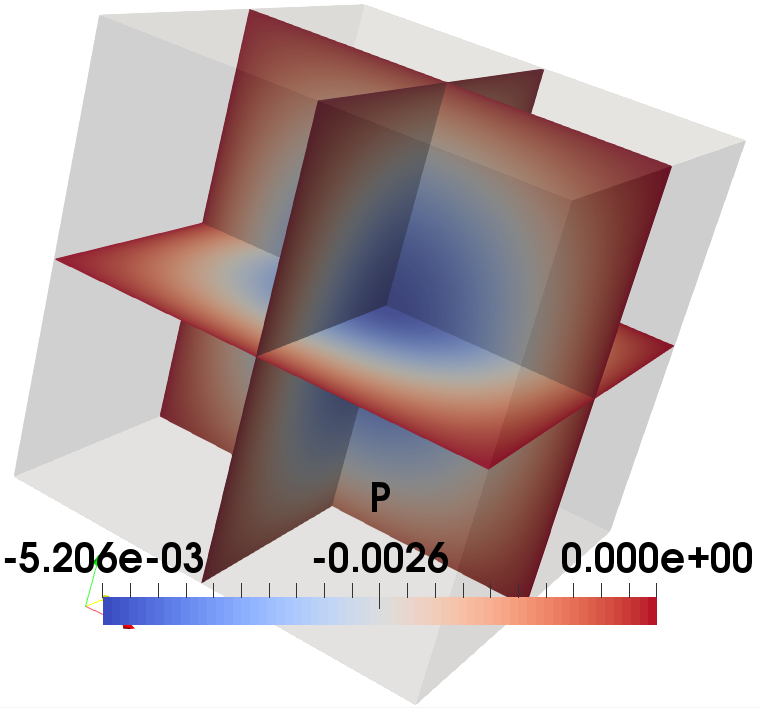}
  \includegraphics[width=0.32\textwidth]{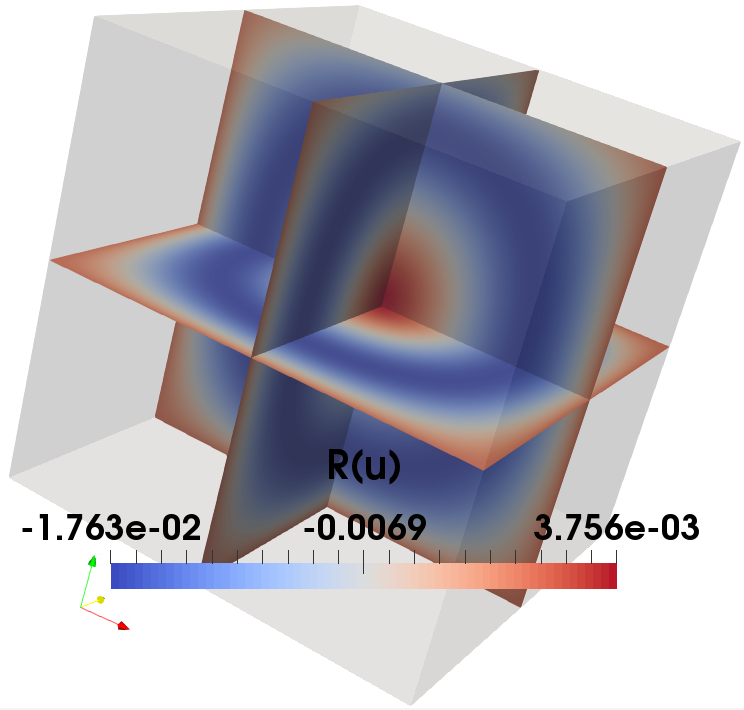}
  \includegraphics[width=0.32\textwidth]{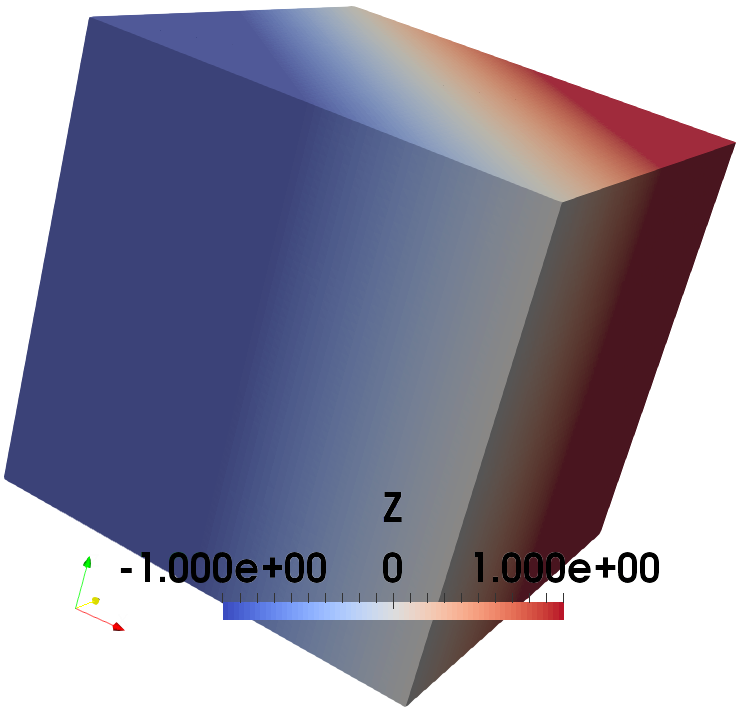}
  \includegraphics[width=0.32\textwidth]{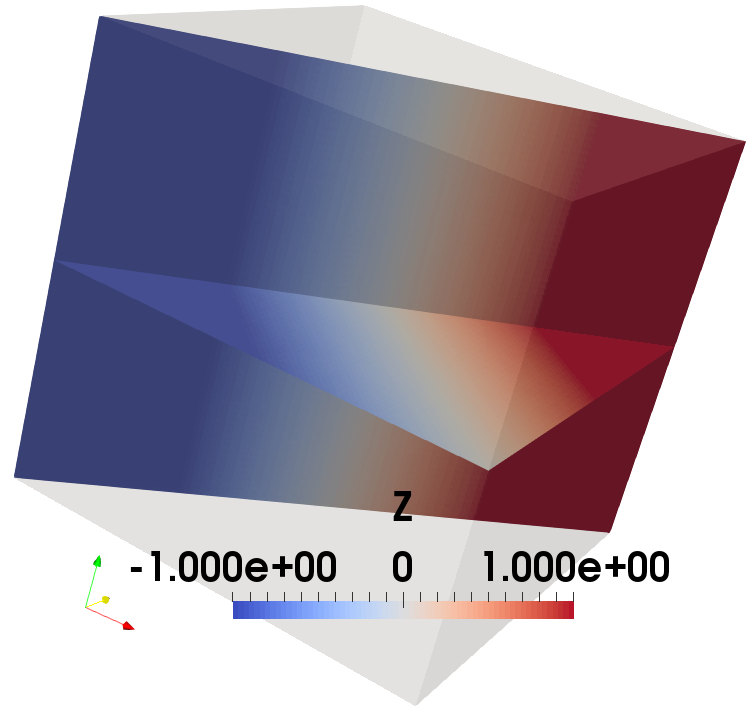}
  \caption{Example 4, numerical solutions of $u$, $p$, and $z$, and $R(u)$ for
    the semilinear model problem with box constraint on control.} 
  \label{fig:numsol_p33}
\end{figure} 

Inserting the exact solutions prescribed above in the primal and adjoint
equations, respectively, the unknown functions $e_u$ and $u_d$ are obtained. 
Along with the defined active and inactive sets
\begin{equation}\label{eq:actinactset}
  \begin{aligned}
    &{\mathcal A}_{a}:=\{(x,t)\in Q: -\varrho^{-1}(p+e_z) < a \}, \\
    &{\mathcal A}_{b}:=\{(x,t)\in Q: -\varrho^{-1}(p+e_z) > b \}, \\
    &{\mathcal I}:=Q\backslash\{{\mathcal A}_{a}\cup {\mathcal A}_{b} \},
  \end{aligned}
\end{equation} 
using the variational inequality \eqref{eq:varinq}, we can construct the
remaining unknown function $e_z$ as follows:
\begin{equation}\label{eq:ezfunc}
  e_z=
  \begin{cases}
    {(p+\varrho a)_{-}} \textup{ on } {\mathcal A}_{a},\\
    -(p+\varrho z)  \textup{ on } {\mathcal I},\\
    {-(p+\varrho b)_{+}}\textup{ on } {\mathcal A}_{b}.
  \end{cases}
\end{equation}
Now, starting from the variational inequality \eqref{eq:varinq}, using 
a piecewise constant ansatz for the control, we arrive at the 
elementwise projection formula
\begin{equation}\label{eq:projform}
z_\ell={\mathbf P}_{[a,b]}\left( -\frac{1}{\varrho |\tau_\ell|} 
\int_{\tau_\ell} (p_h + e_z) \, dx \, dt \right), \quad
\forall \tau_\ell \in {\mathcal{T}}_h(Q).
\end{equation}
Inserting this formula in the discrete state equation, we obtain a
coupled nonlinear first order necessary optimality system for the state and
adjoint variables. This discrete nonlinear system is solved by the semismooth
Newton method. After having computed $p_h$, in a final postprocessing step, we
use the projection formula \eqref{eq:projform} to 
compute the optimal control as a piecewise linear and continuous function.  

The estimated order of convergence in $Y=L^2(0,T; H^1_0(\Omega))$ is 
displayed in Table~\ref{tab:l2_eocl2h1_p33}. We clearly observe an 
optimal convergence rate. For this example, we do not have optimal convergence
in $L^2(Q)$; see Table \ref{tab:l2_eocl2l2_p33}. However, we see a
second-order convergence rate for the objective functional;
cf.~Table~\ref{tab:l2_eocobj_p33}.  

\begin{table}[h]\caption{Example 4, estimated order
    of convergence (eoc) of $u_h, p_h$ in 
    $Y$ for the semilinear model
    problem with box constraints.} 
\centering
\begin{tabular}{rrrrrr}
\toprule
\#Dofs & $h$ & $\|u-u_h\|_Y$ & eoc & $\|p-p_h\|_Y$ & eoc  \\
\midrule
$250$ & $2^{-2}$& $2.121e-1$    & $-$  & $5.272e-3$    & $-$       \\
$1,458$ & $2^{-3}$& $1.126e-1$    & $0.913$  &$2.396e-3$    & $1.138$  \\
$9,826$ & $2^{-4}$ & $5.653e-2$    & $0.995$ & $1.142e-3$    & $1.069$  \\
$71,874$ & $2^{-5}$ & $2.817e-2$    & $1.008$ & $5.608e-4$    & $1.026$   \\
$549,250$ & $2^{-6}$ &$1.401e-2$    & $1.005$ & $2.790e-4$    & $1.007$ \\
$2,146,689$ & $2^{-7}$ &$7.017e-3$    & $0.997$ & $1.407e-4$    & $0.988$  \\
\bottomrule
\end{tabular}\label{tab:l2_eocl2h1_p33}
\end{table}

\begin{table}[h]\caption{Example 4, estimated order of convergence (eoc)
of $u_h, p_h, z_h$ in $L^2(Q)$ for the semilinear model
    problem with box constraints ($\|\cdot\|=\|\cdot\|_{L^2(Q)}$).} 
\centering
\begin{tabular}{rrrrrrrr}
\toprule
\#Dofs &  $h$ & $\|u-u_h\|$ & eoc & $\|p-p_h\|$ & eoc &
$\|z-z_h\|$ & eoc \\
\midrule
$250$ & $2^{-2}$& $2.133e-2$ & $-$ & $3.811e-4$ & $-$ & $3.072e-1$ & $-$ \\
$1,458$ & $2^{-3}$& $5.873e-3$ & $1.861$ &$1.018e-4$ & $1.905$ & $9.984e-2$ 
& $1.622$  \\
$9,826$ & $2^{-4}$ & $1.566e-3$ & $1.907$ & $2.489e-5$ & $2.032$ & $3.251e-2$ 
& $1.619$ \\
$71,874$ & $2^{-5}$ & $4.733e-4$ & $1.727$ & $6.355e-6$ & $1.970$ &$1.379e-2$ 
& $1.237$ \\
$549,250$ & $2^{-6}$ & $2.224e-4$ & $1.089$ & $3.731e-6$ & $0.768$ &$7.435e-3$ 
& $0.891$ \\
$2,146,689$ & $2^{-7}$ &$1.741e-4$ & $0.353$ & $3.793e-6$ & $-0.024$ 
& $5.016e-3$ & $0.568$  \\
\bottomrule
\end{tabular}\label{tab:l2_eocl2l2_p33}
\end{table}

\begin{table}[h]
\caption{Example 4, $J(u_h,z_h)$ and $|J(u_h,z_h)-J(u,z)|$ for the semilinear model
problem with box constraint on control, where $J(u,z)=5.1743e-4$.} 
\centering
\begin{tabular}{rrrrr}
       \toprule
       \#Dofs &  $h$ & $J(u_h, z_h)$ &  $|J(u_h, z_h)-J(u,z)|$ &  eoc   \\
       \midrule
       $250$ & $2^{-2}$& $1.8162e-3$      & $1.299e-3$    & $-$   \\
       $1,458$ & $2^{-3}$& $8.1878e-4$     &$3.014e-4$    & $2.108$  \\
       $9,826$ & $2^{-4}$ & $5.8804e-4$    & $7.061e-5$    & $2.094$  \\
       $71,874$ & $2^{-5}$ & $5.3448e-4$   & $1.705e-5$    & $2.050$ \\
       $549,250$ &  $2^{-6}$ &$5.2161e-4$   & $4.180e-6$    & $2.028$ \\
       $2,146,689$ &  $2^{-7}$ &$5.1846e-4$   & $1.030e-6$    & $2.021$ \\
       \bottomrule
     \end{tabular}\label{tab:l2_eocobj_p33}
\end{table}

\subsubsection{Example with a turning wave (Example 5)}
As final example,  we consider the following optimal control problem:
\[
\min{\mathcal J} (u,z) := \frac{1}{2}\int_Q ( u  - u_d )^2 \, dx \, dt
  + \frac{\varrho}{2} \int_Q z^2 \, dx \, dt
\]
subject to the state equation
\[
\partial_t u  - \Delta_x u + R(u) = z \text{ in } Q,\quad
\partial_n u = 0 \text{ on }  \Sigma,\quad
u = u_0 \text{ on } \Sigma_0,
\]
and
\[
a \leq z(x,t) \leq b \quad \text{ for a.a. } (x,t)\in Q.
\]
This simplified model problem is an adapted version of that one provided
in \cite{ECCRFT13}, here without $L^1$-regularization. The first-order 
necessary optimality system for this model problem reads as
follows: In addition to the above state equation,
the adjoint equation is
\[
- \partial_t p - \Delta_x p  + R'(u) p  =  u - u_d \text{ in } Q,\quad
\partial_n p  = 0 \text{ on } \Sigma,\quad 
p = 0 \text{ on } \Sigma_T,       
\]
and the gradient equation reads
\[
z={\mathbf P}_{[a,b]}\left(-\frac{1}{\varrho} p \right)\text{ in } Q.
\]
As for the turning wave example constructed in \cite{ECCRFT13}, we define 
the nonlinear reaction term $R(u)=u(u-0.25)(u+1)$, the initial condition 
\[
u_0 = \left(1+\exp\left(\frac{\frac{70}{3}-70x_1}{\sqrt{2}} 
\right)\right)^{-1}+\left(1+\exp\left(\frac{70
x_1-\frac{140}{3}}{\sqrt{2}}\right)\right)^{-1}-1\textup{ on }\Sigma_0 
\] 
for the state, and the target
\begin{equation*}
    \begin{aligned}
      &{u}_d=\left( 1.0+ \exp \left( \frac{\cos ( g(t)
        )\left(\frac{70}{3}- 70x_1\right)+ \sin ( g(t) ) \left(\frac{70}{3} -
        70x_2\right) }{\sqrt{2}} \right) \right)^{-1}\\
      &\quad + \left( 1.0+ \exp \left( \frac{ \cos ( g(t) ) \left( 70x_1 -
        \frac{140}{3}\right) + \sin ( g(t) ) \left( 70x_2 -
        \frac{140}{3}\right)}{\sqrt {2}} \right) \right)^{-1} - 1
    \end{aligned}
  \end{equation*}
in $Q$, where $g(t)=\frac{2\pi}{3}\min\left\{\frac{3}{4}, t\right\}$. We 
should mention that the target defined in \cite{ECCRFT13} contained a typo; 
this is corrected here. The wave front turns $90$ degrees from time 
$t=0$ to $t=0.75$ and remains fixed after
$t=0.75$; see the target at $t=0$, $0.25$, $0.5$, and $0.75$ illustrated in
Fig.~\ref{fig:targetturningwave}.     

\begin{figure}[htb]
    \centering
    \includegraphics[scale=0.138]{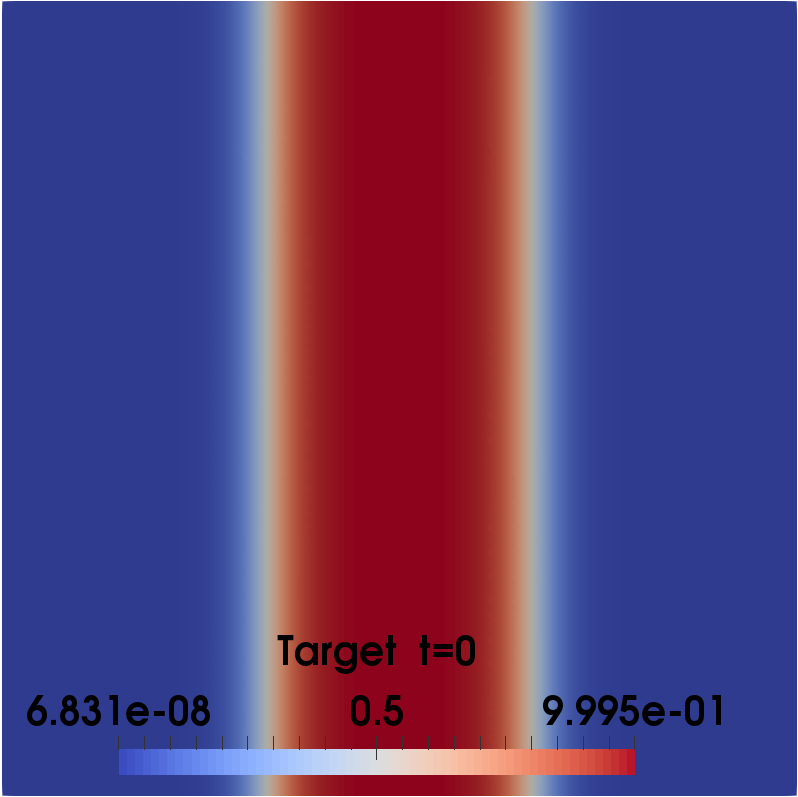}
    \includegraphics[scale=0.138]{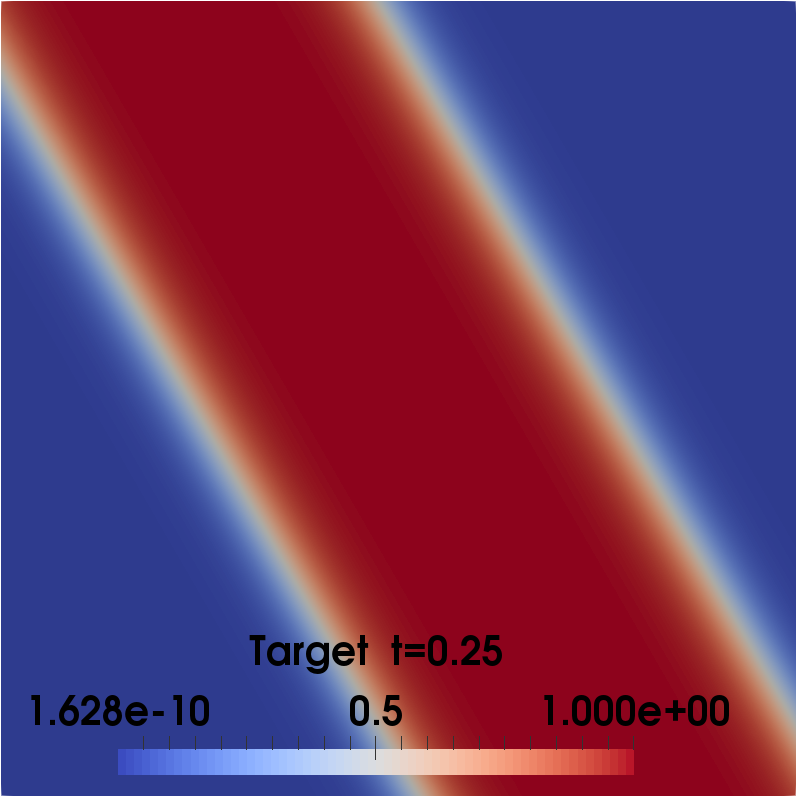}
    \includegraphics[scale=0.138]{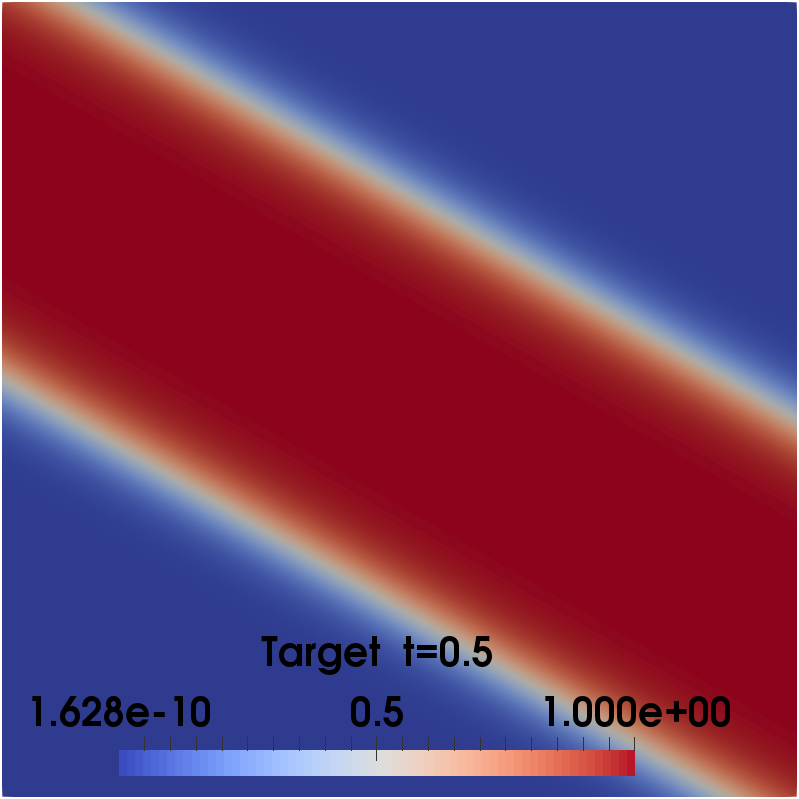}
    \includegraphics[scale=0.138]{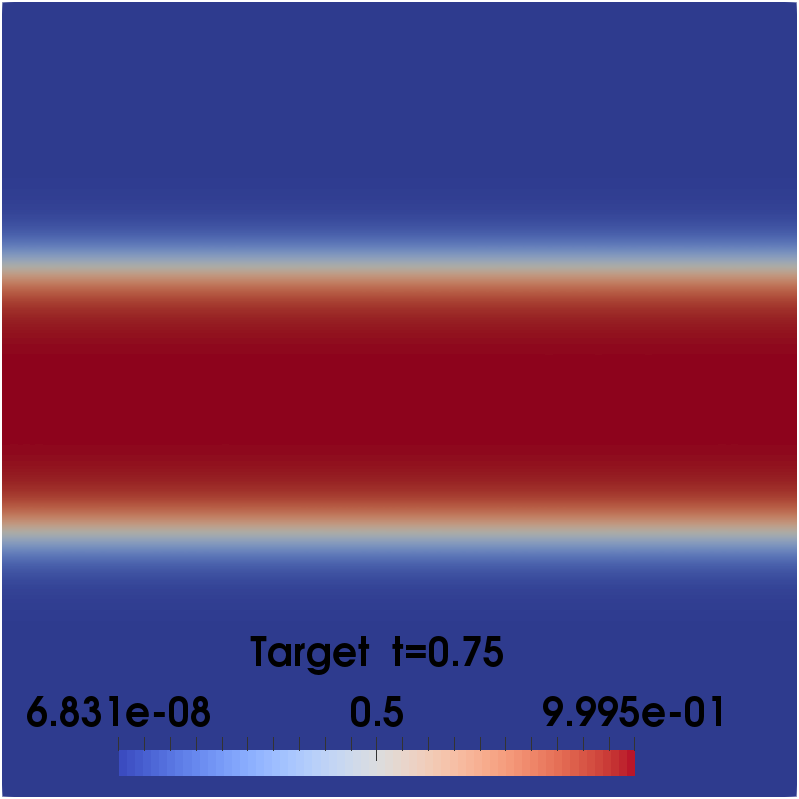}
    \caption{Example 5, plots of the target at time $t=0,\;0.25,\;\;0.5,\;0.75$
      for the turning wave example.} \label{fig:targetturningwave}
 \end{figure} 
    
As parameters, we use $\varrho=10^{-6}$, $a=-10^{+6}$, and $b=10^{+6}$ in the 
unconstrained case, while $a=-10^{+2}$ and $b=10^{+2}$ are set in the
constrained case. We then follow the approach of the previous 
section to solve the coupled nonlinear first order optimality system by
the semismooth Newton method. The numerical solutions for state, adjoint
state, and control in the space-time domain are visualized in
Fig.~\ref{fig:turningwaveupz}.  

\begin{figure}[htb]
    \centering
    \includegraphics[scale=0.2]{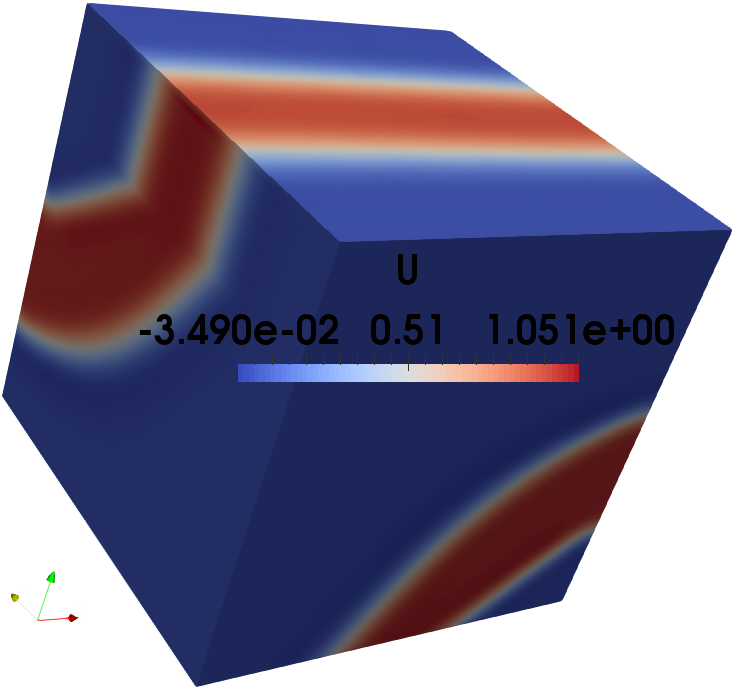}
    \includegraphics[scale=0.2]{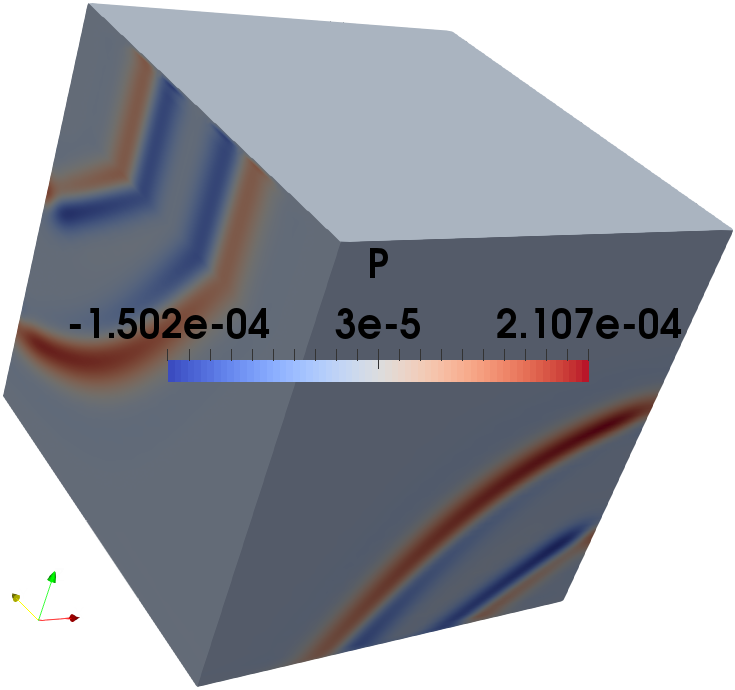}
    \includegraphics[scale=0.2]{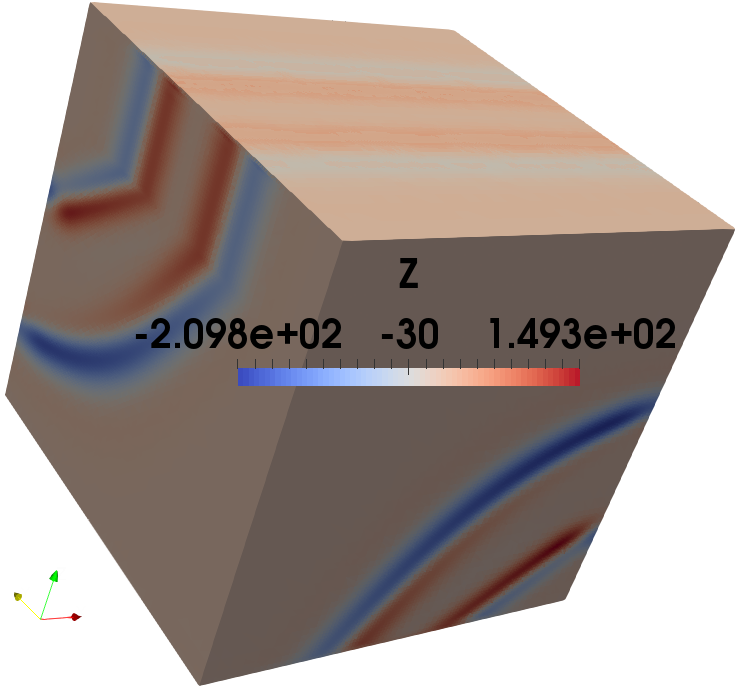}
    \includegraphics[scale=0.2]{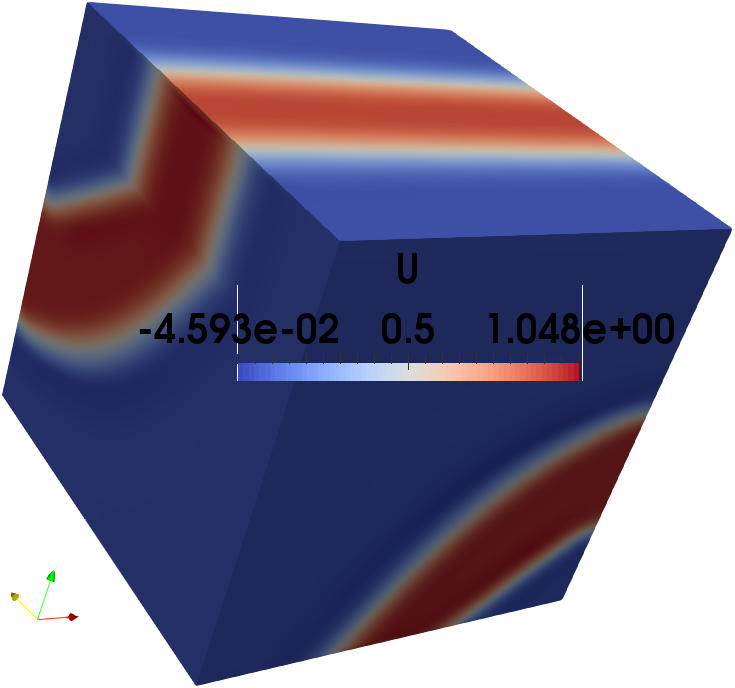}
    \includegraphics[scale=0.2]{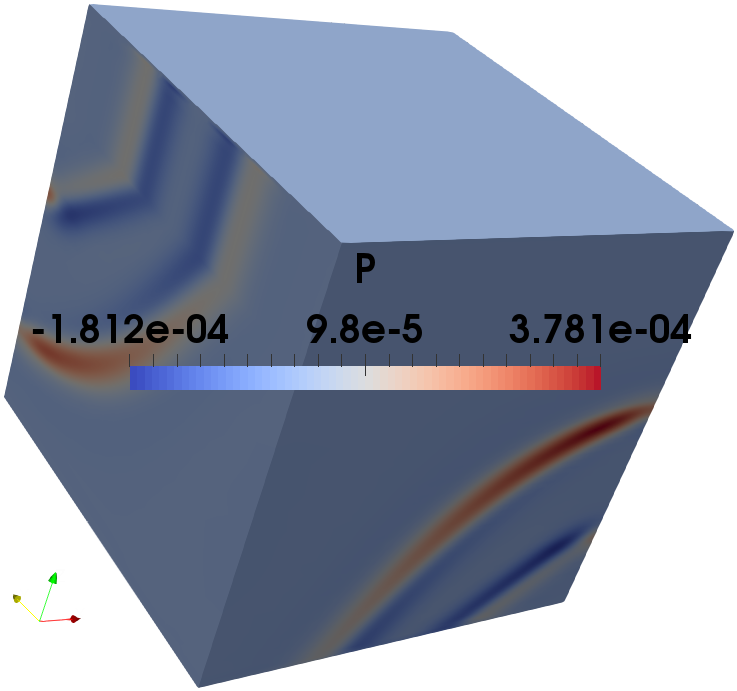}
    \includegraphics[scale=0.2]{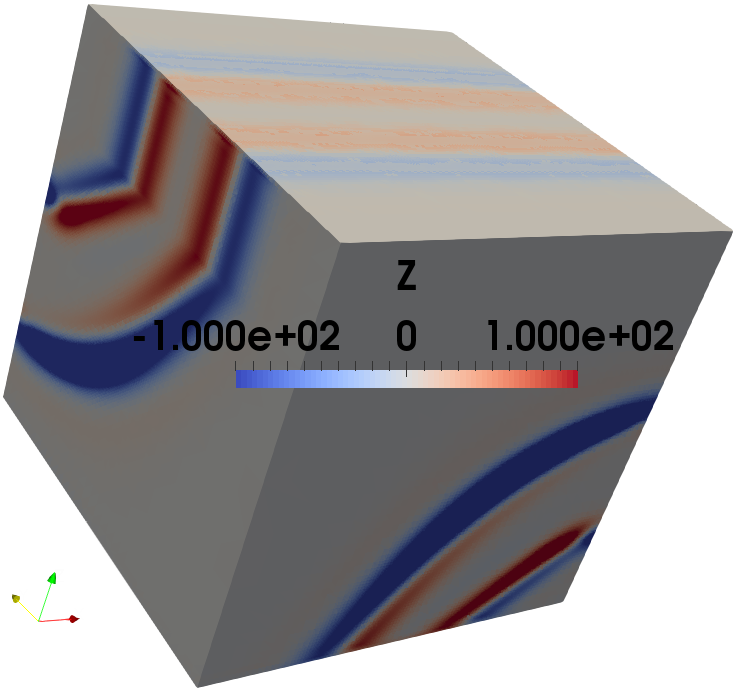}
    \caption{Example 5, visualization of the numerical solutions for state,
      adjoint state and control in the space-time domain: 
      Without box constraints ($a=-1e+6$, $b=1e+6$, top), with box 
      constraints ($a=-1e+2$, $b=1e+2$, bottom).} \label{fig:turningwaveupz}
\end{figure} 
  
In Fig.~\ref{fig:turningwavesliceu} and Fig.~\ref{fig:turningwaveslicez}, 
we visualize the numerical solutions for the state and the control at 
different times $t=0,\;0.25,\;0.5,\;0.75$, respectively. In this particular 
turning wave example, we see almost no difference between the cases with 
or without control constraints.   

\begin{figure}[htb]
    \centering
    \includegraphics[scale=0.138]{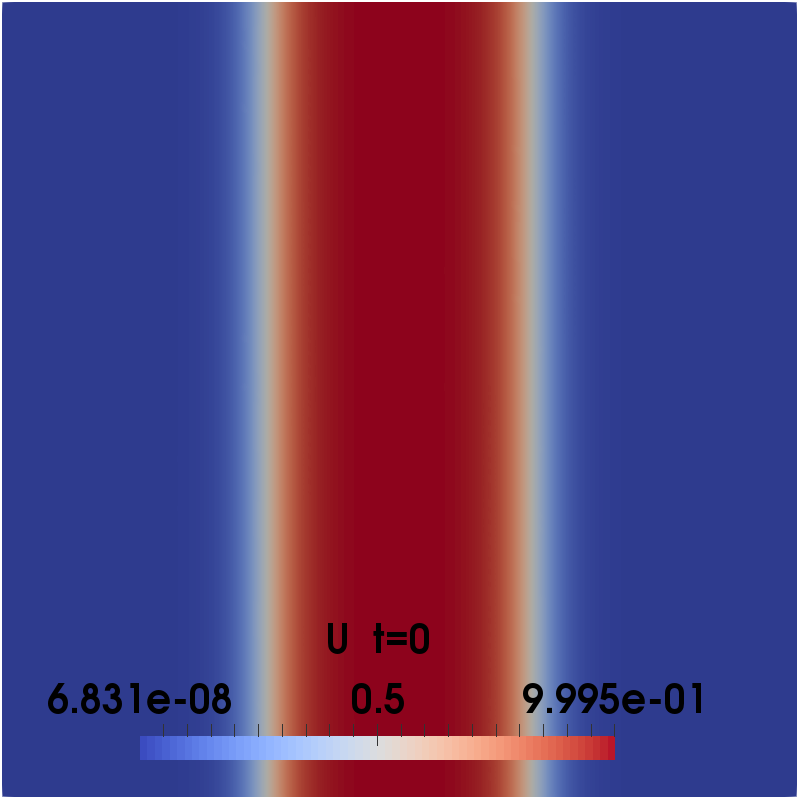}
    \includegraphics[scale=0.138]{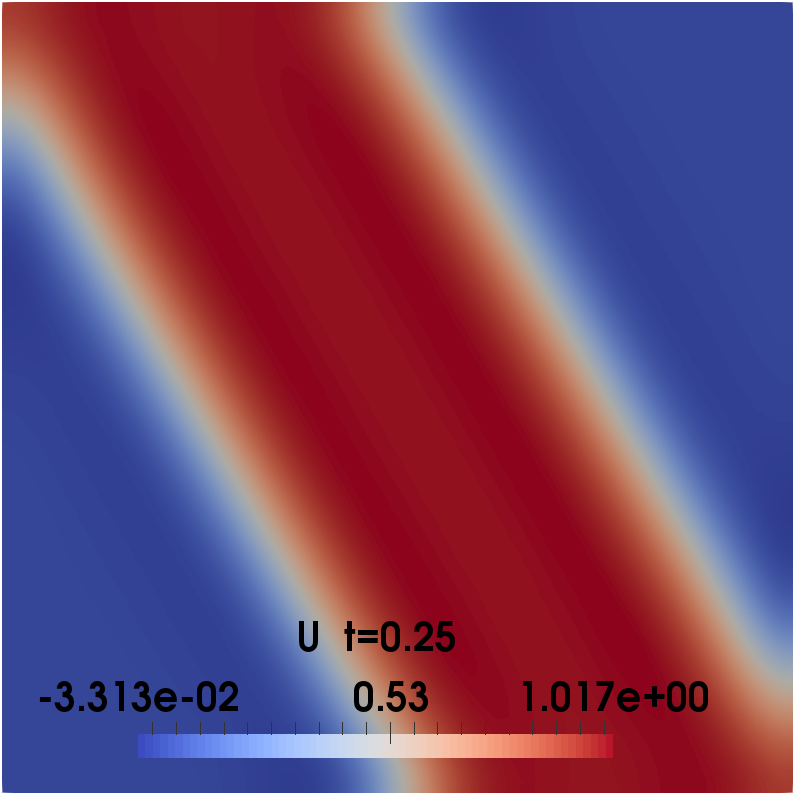}
    \includegraphics[scale=0.138]{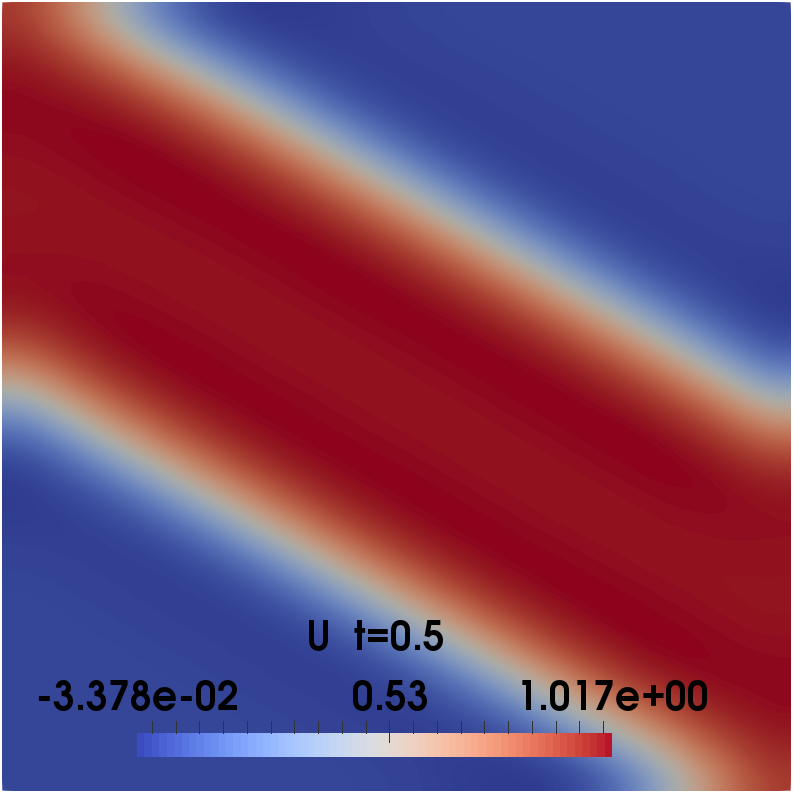}
    \includegraphics[scale=0.138]{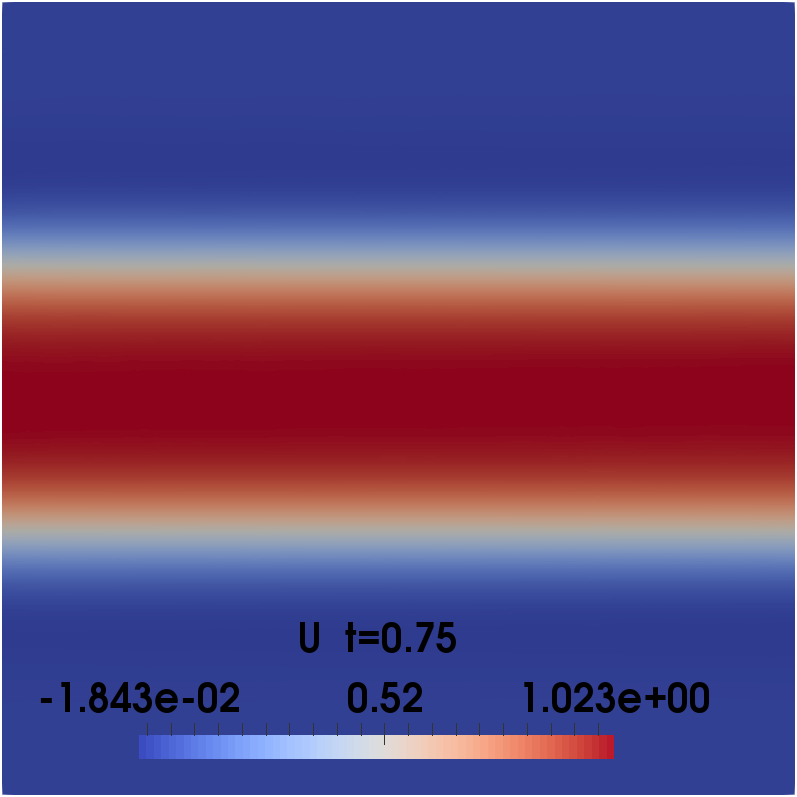}
    \includegraphics[scale=0.138]{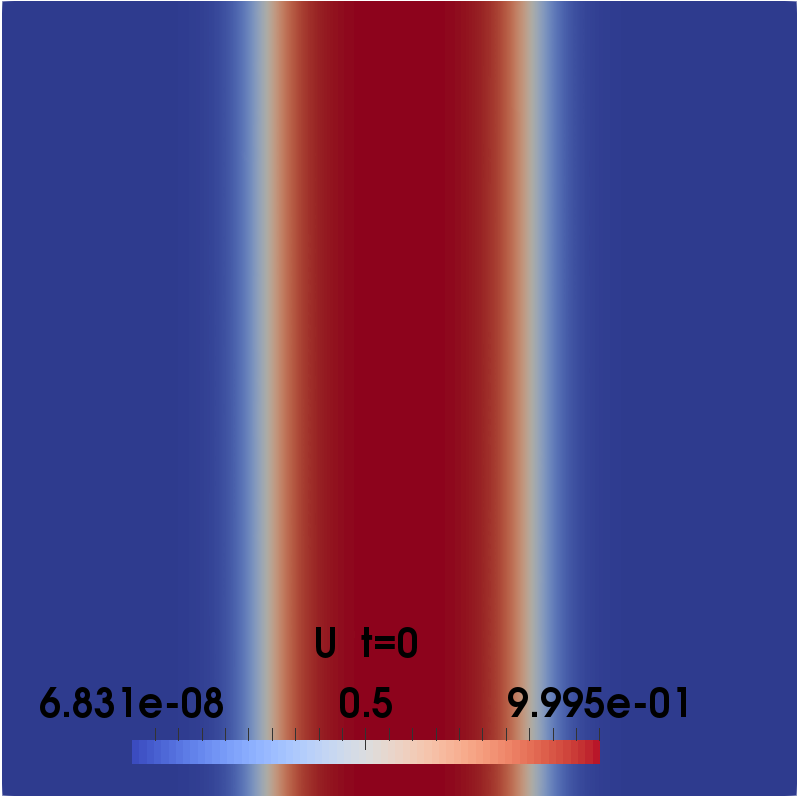}
    \includegraphics[scale=0.138]{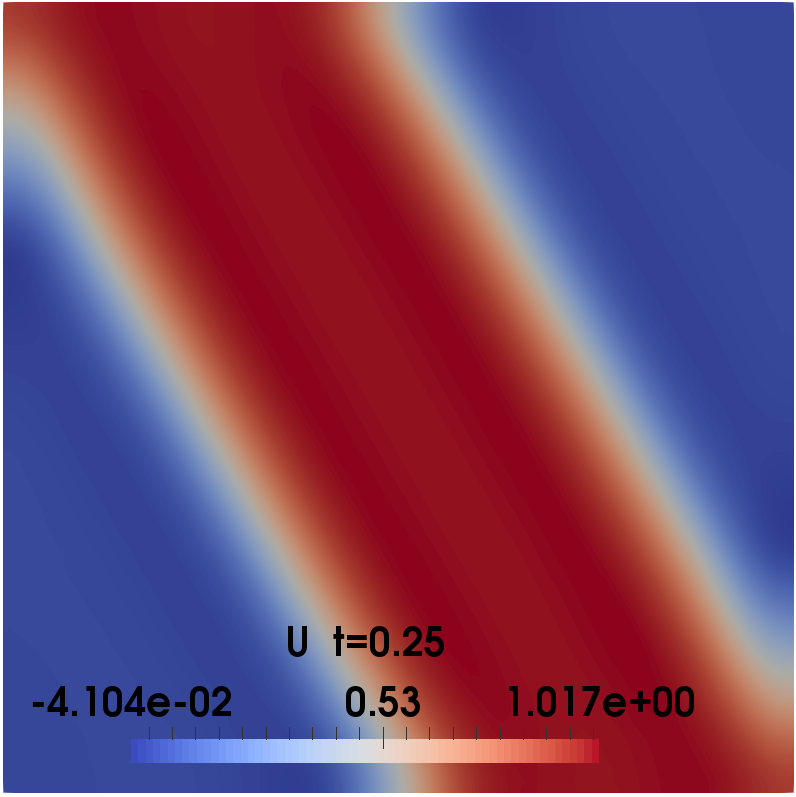}
    \includegraphics[scale=0.138]{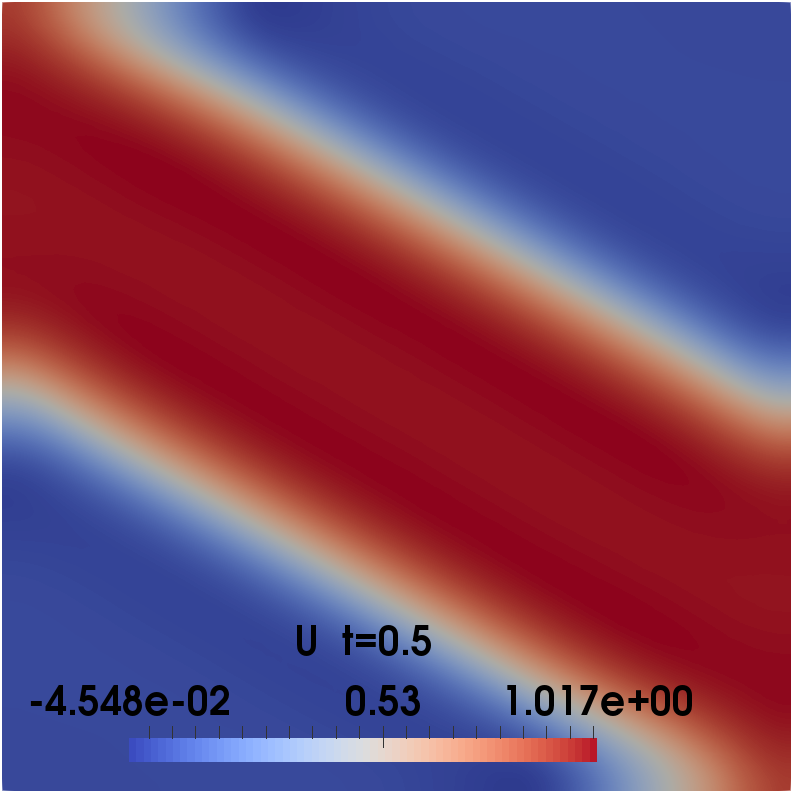}
    \includegraphics[scale=0.138]{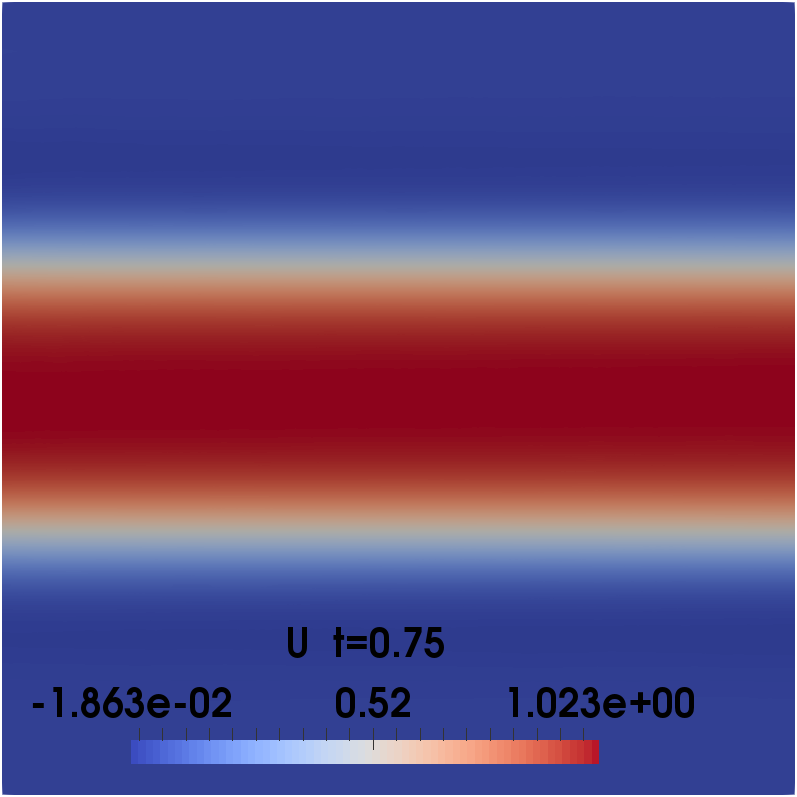}
    \caption{Example 5, visualization of the numerical solutions for the state
      at $t=0,\;0.25,\; 0.5,\;0.75$, without box constraints ($a=-1e+6$,
      $b=1e+6$, top) and with box constraints ($a=-1e+2$, $b=1e+2$,
      bottom).} \label{fig:turningwavesliceu} 
\end{figure} 

\begin{figure}[htb]
    \centering
    \includegraphics[scale=0.138]{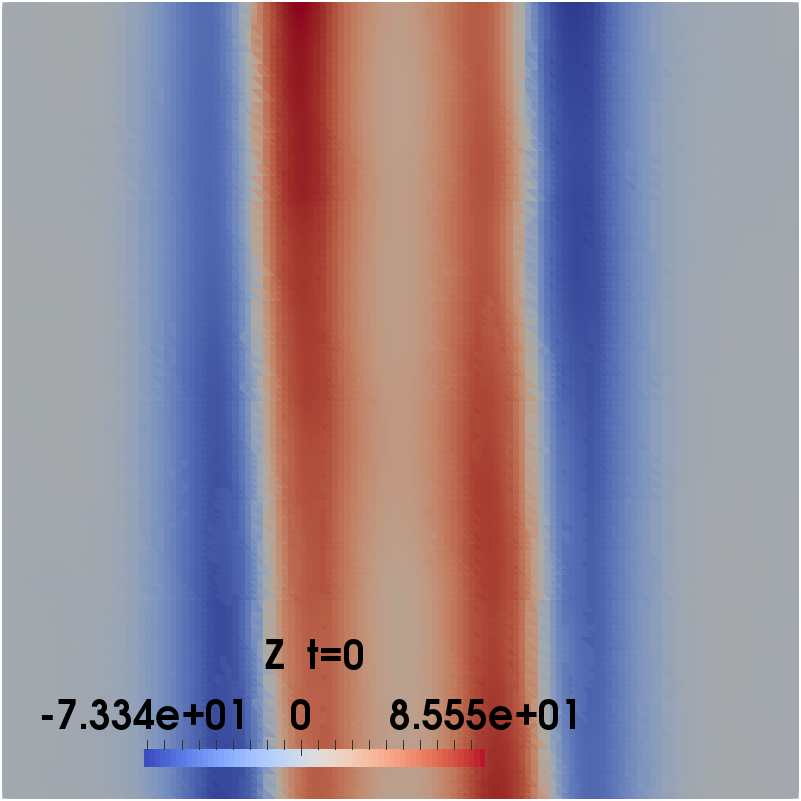}
    \includegraphics[scale=0.138]{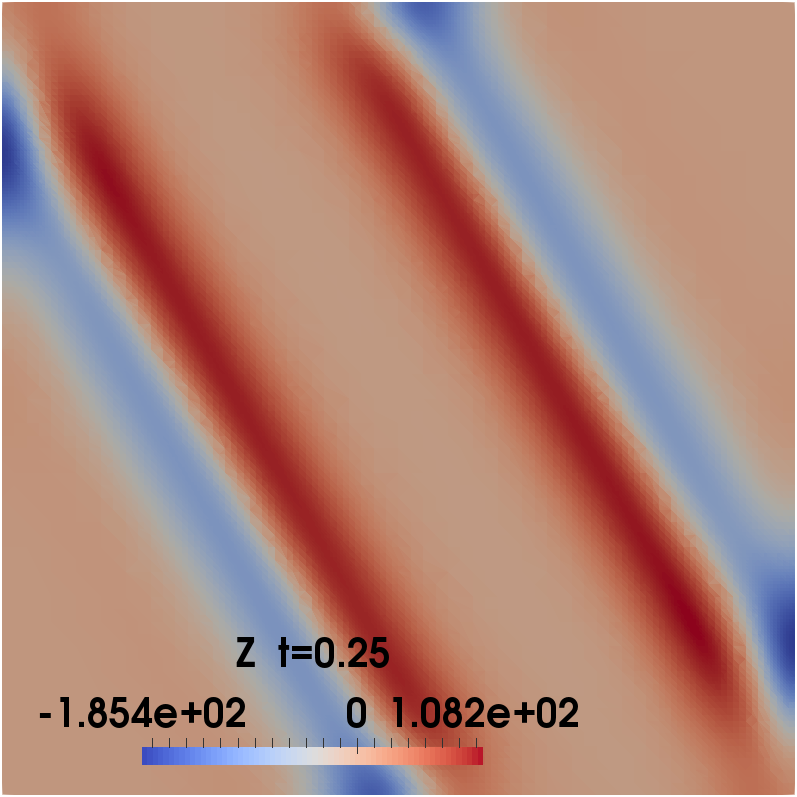}
    \includegraphics[scale=0.138]{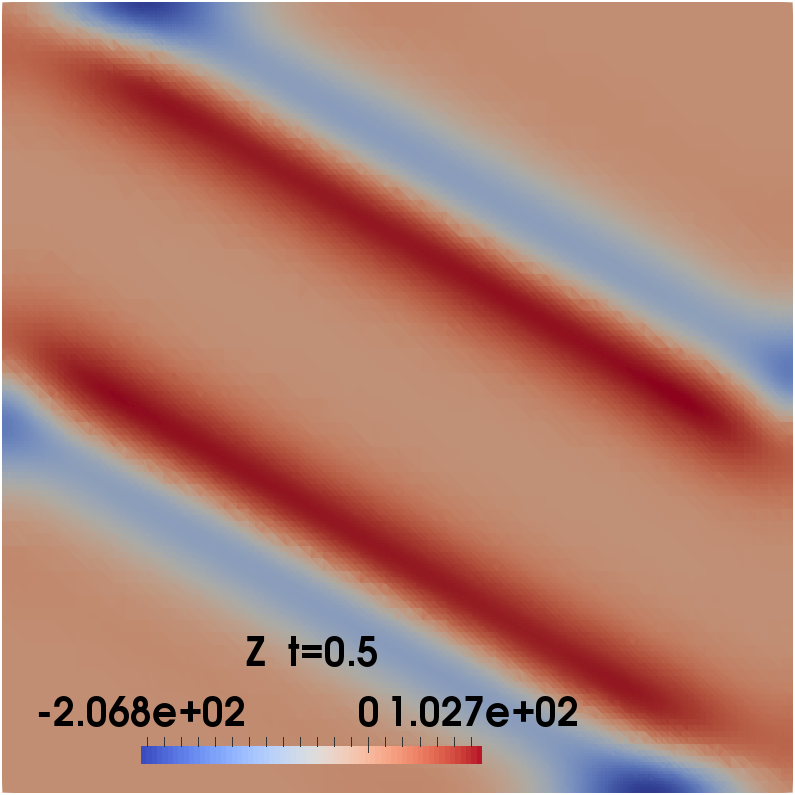}
    \includegraphics[scale=0.138]{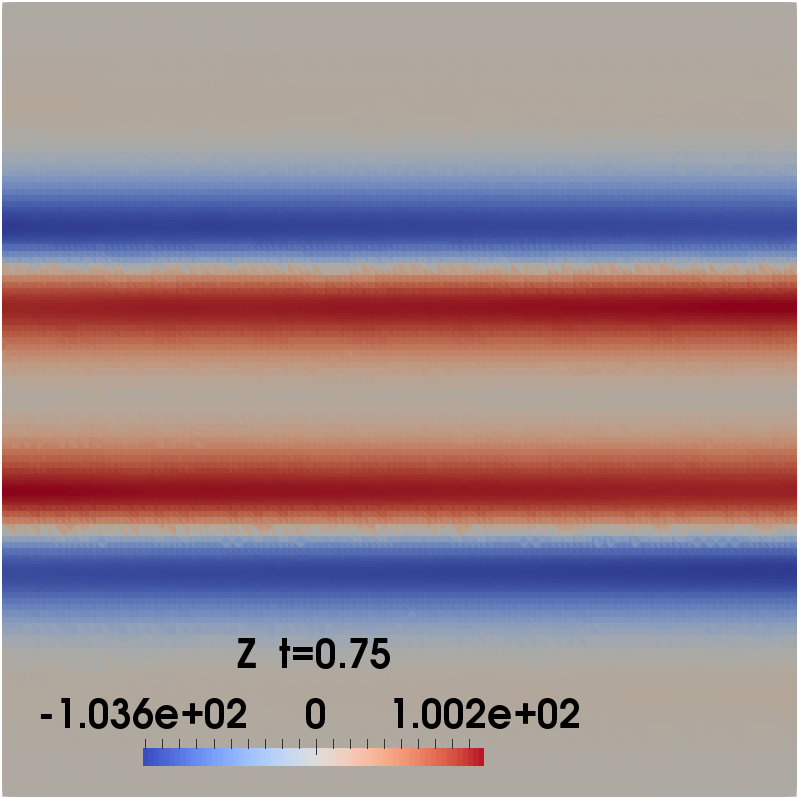}
    \includegraphics[scale=0.138]{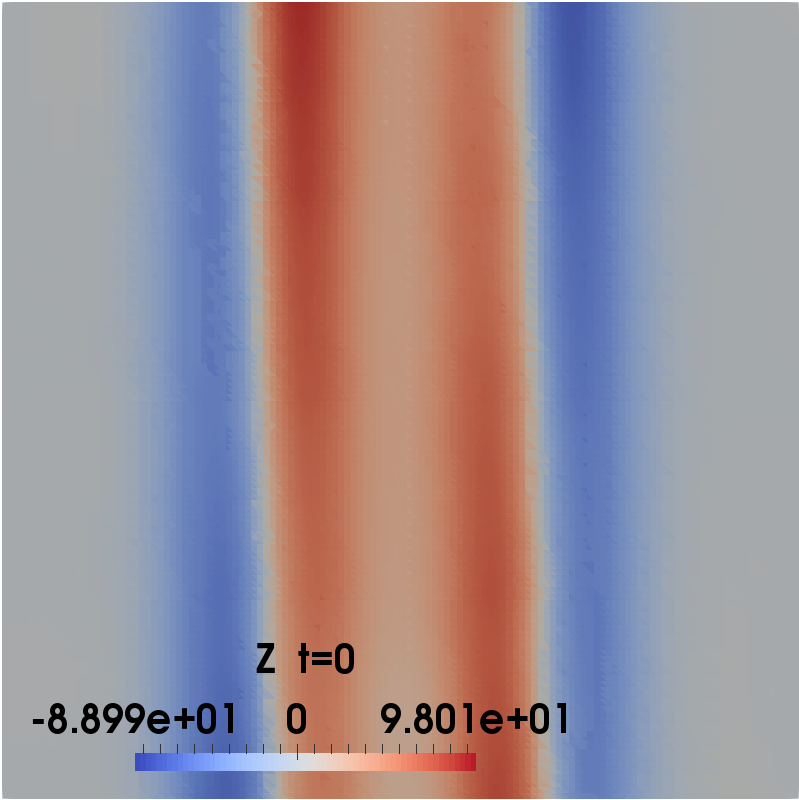}
    \includegraphics[scale=0.138]{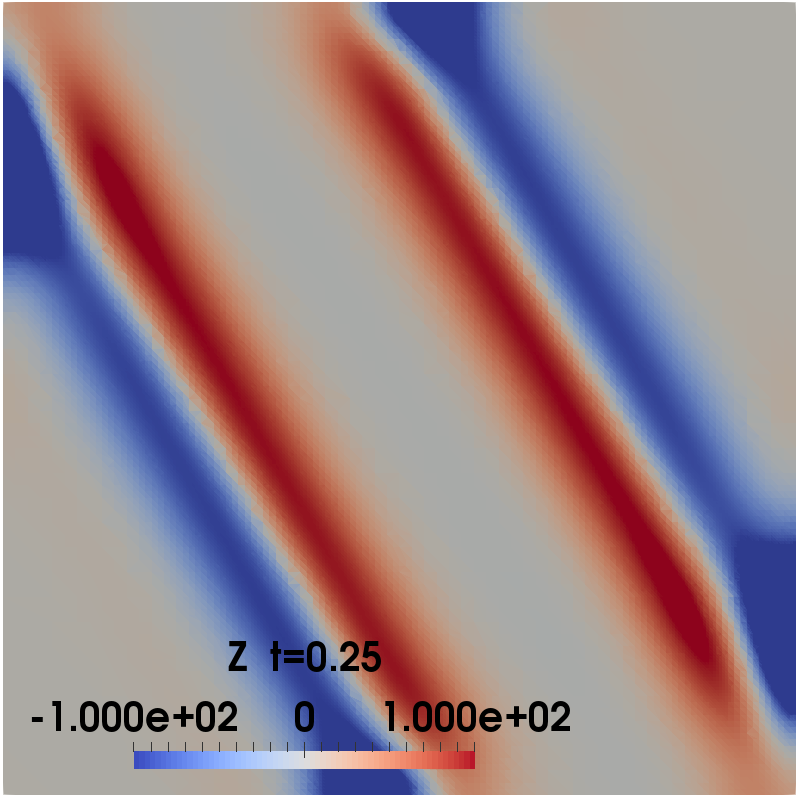}
    \includegraphics[scale=0.138]{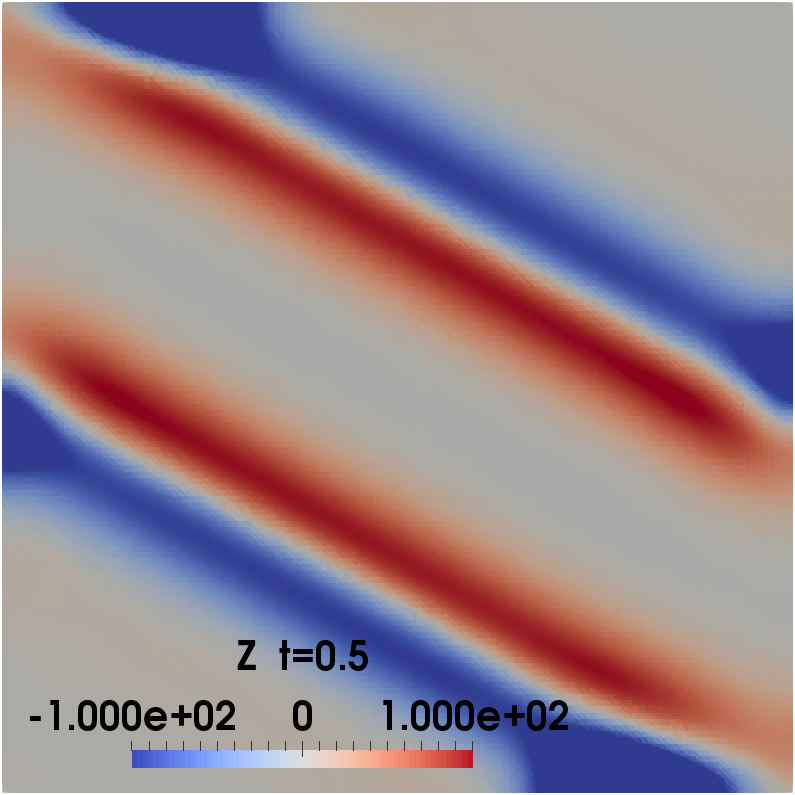}
    \includegraphics[scale=0.138]{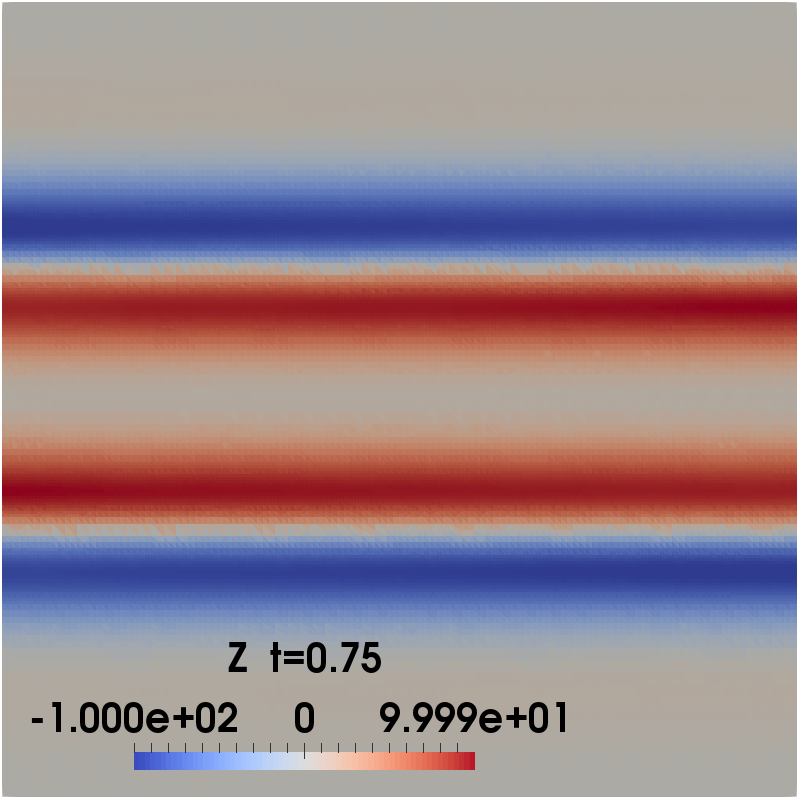}
    \caption{Example 5, visualization of the numerical solutions for the
      control at $t=0,\;0.25,\; 0.5,\;0.75$, without box constraints
      ($a=-1e+6$, $b=1e+6$, top) and with box constraints ($a=-1e+2$, $b=1e+2$,
      bottom).} \label{fig:turningwaveslicez} 
\end{figure} 
  
For the coupled state and adjoint state system, we applied again an
adaptive method as it was used in \cite{OSHY18} for the heat (state)
equation. The adaptive meshes in the
space-time domain and at different time levels $t=0,\;0.25,\;0.5,\;0.75$ are
illustrated in Fig.~\ref{fig:adaptiveturningwaves}. We clearly observe
that our adaptive mesh refinements follow the rotation of the turning 
wave fronts in both the unconstrained and constrained cases. In the 
unconstrained problem, the mesh is visualized for the
$25$th refinement step, containing $3,774,637$ grid points, i.e., 
$7,549,274$ degrees of freedom in total for the coupled first order 
necessary optimality system.  In the constrained setting, the mesh is 
displayed for the $28$th refinement step, containing
$5,100,060$ grid points, i.e., $10,200,120$ degrees of freedom in total 
for the coupled optimality system.

\begin{figure}[htb]
    \centering
    \includegraphics[scale=0.112]{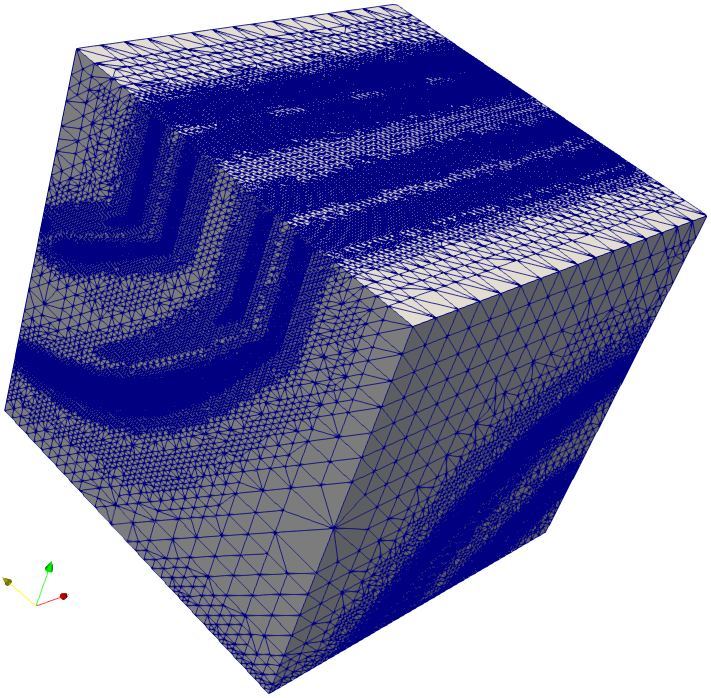}
    \includegraphics[scale=0.112]{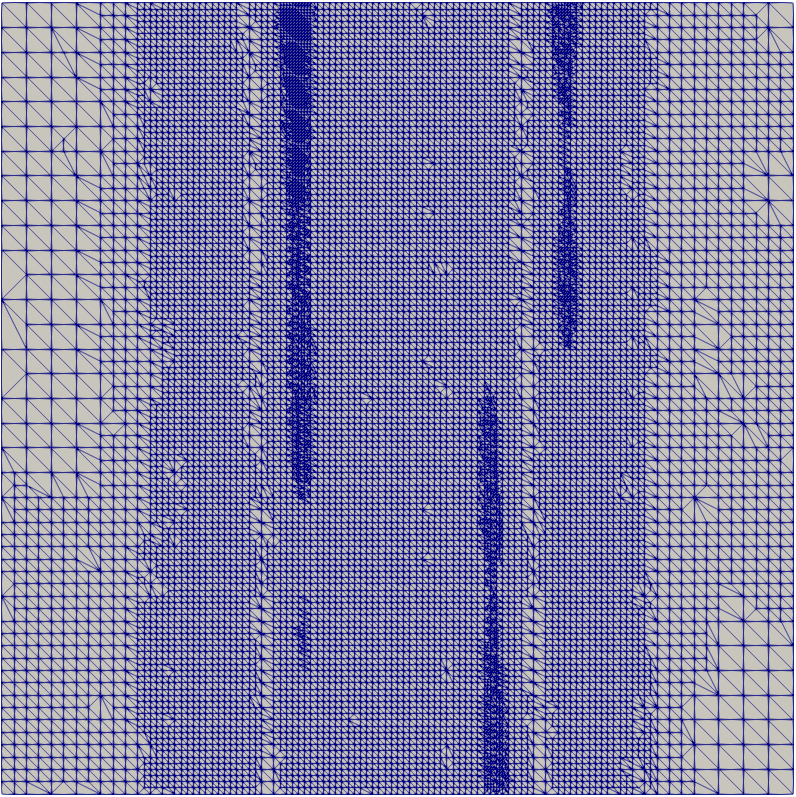}
    \includegraphics[scale=0.112]{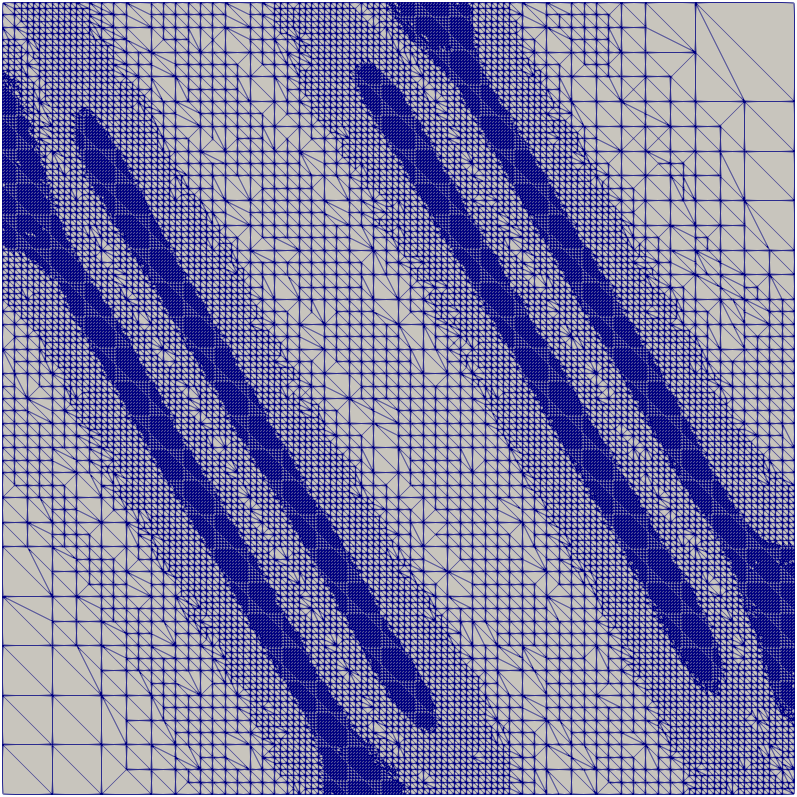}
    \includegraphics[scale=0.112]{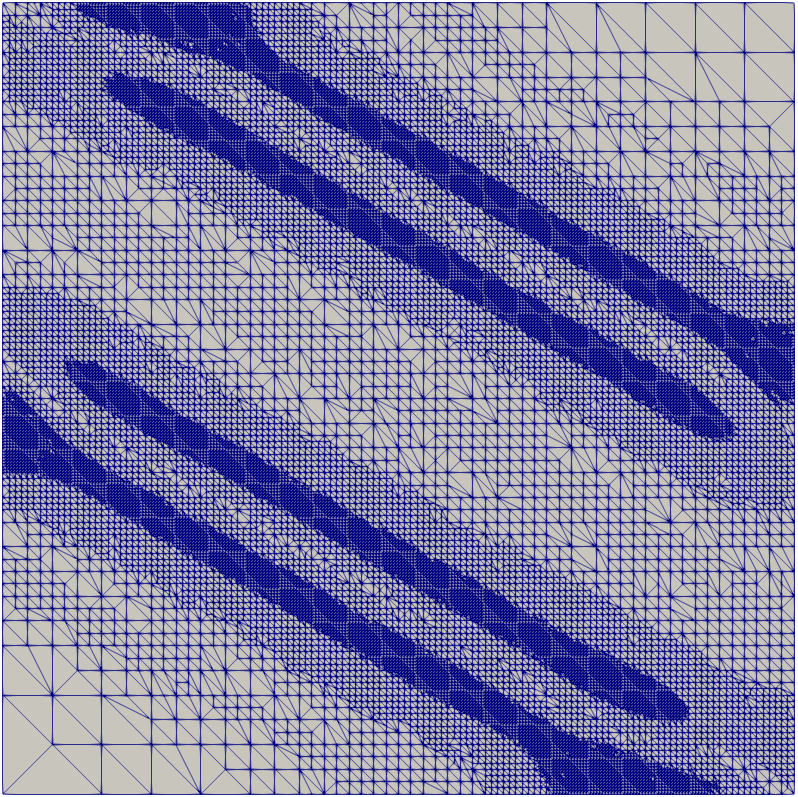}
    \includegraphics[scale=0.112]{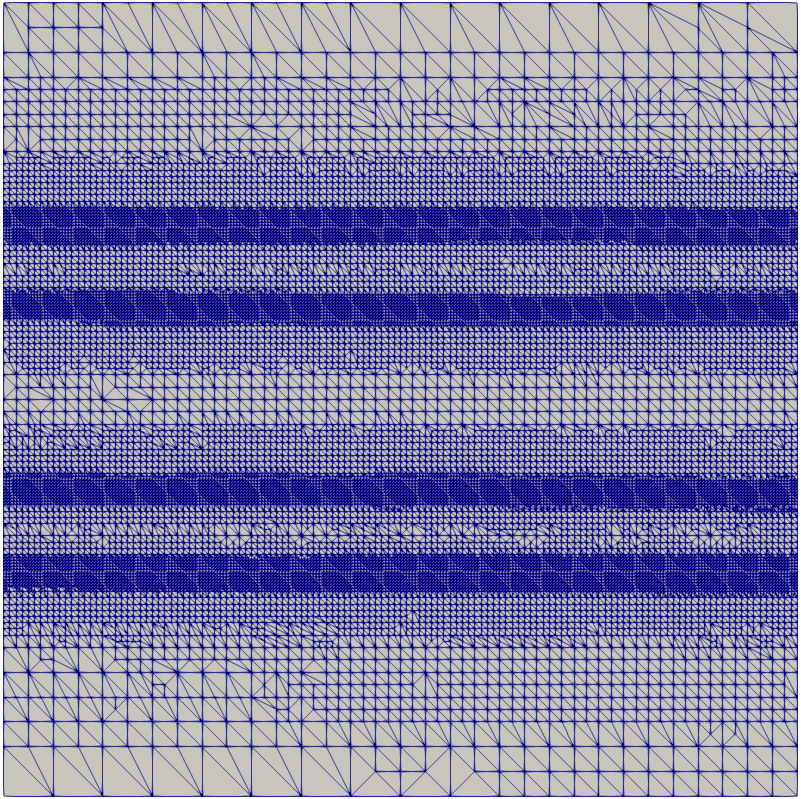}
    \includegraphics[scale=0.112]{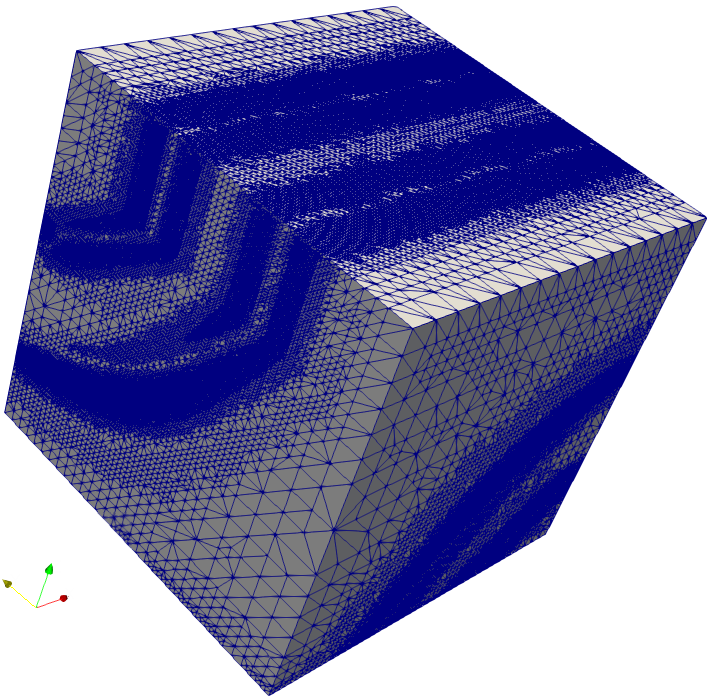}
    \includegraphics[scale=0.112]{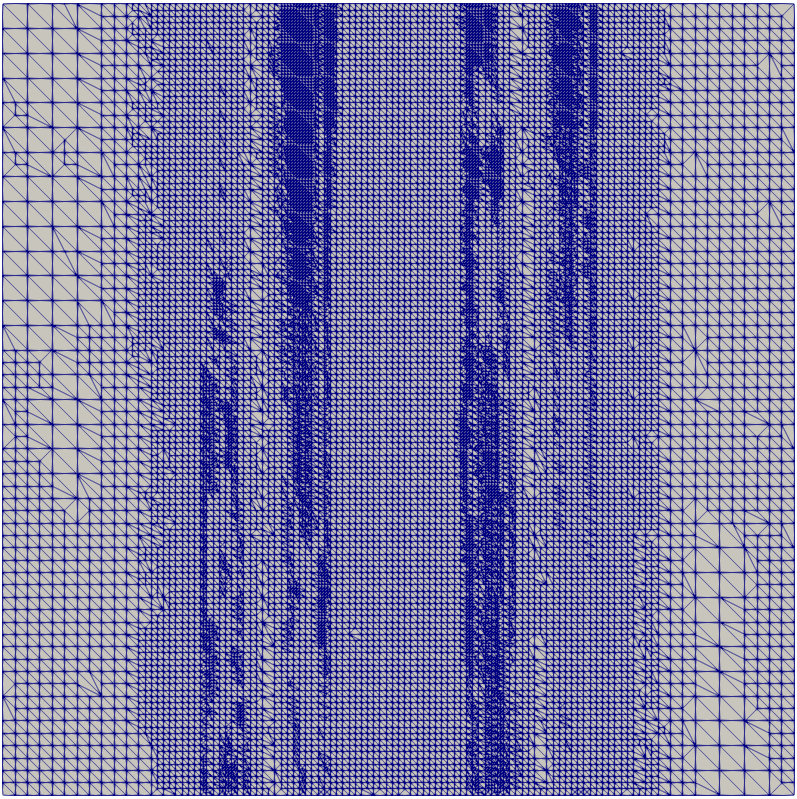}
    \includegraphics[scale=0.112]{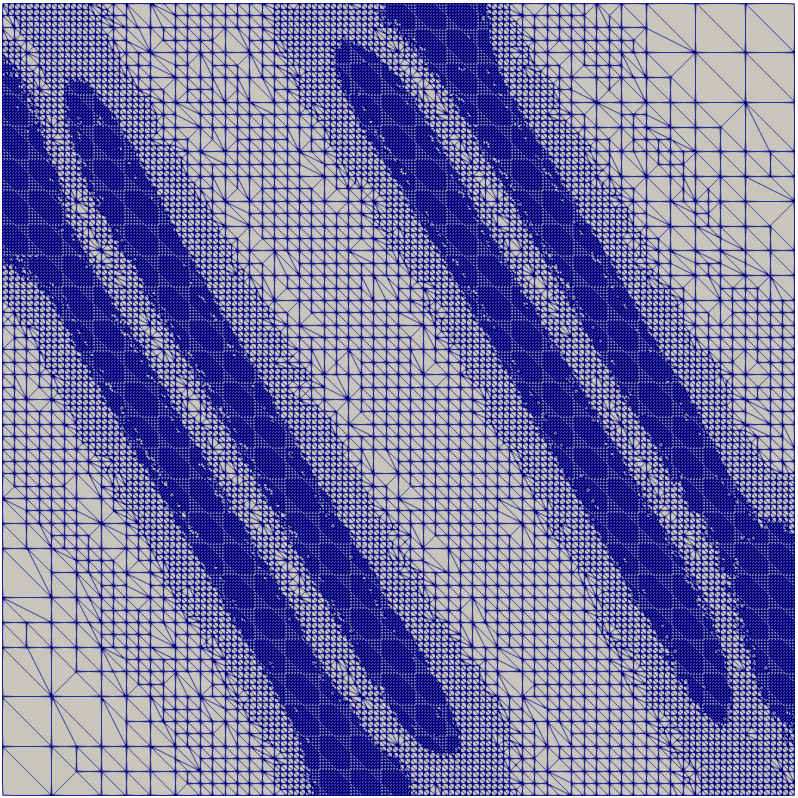}
    \includegraphics[scale=0.112]{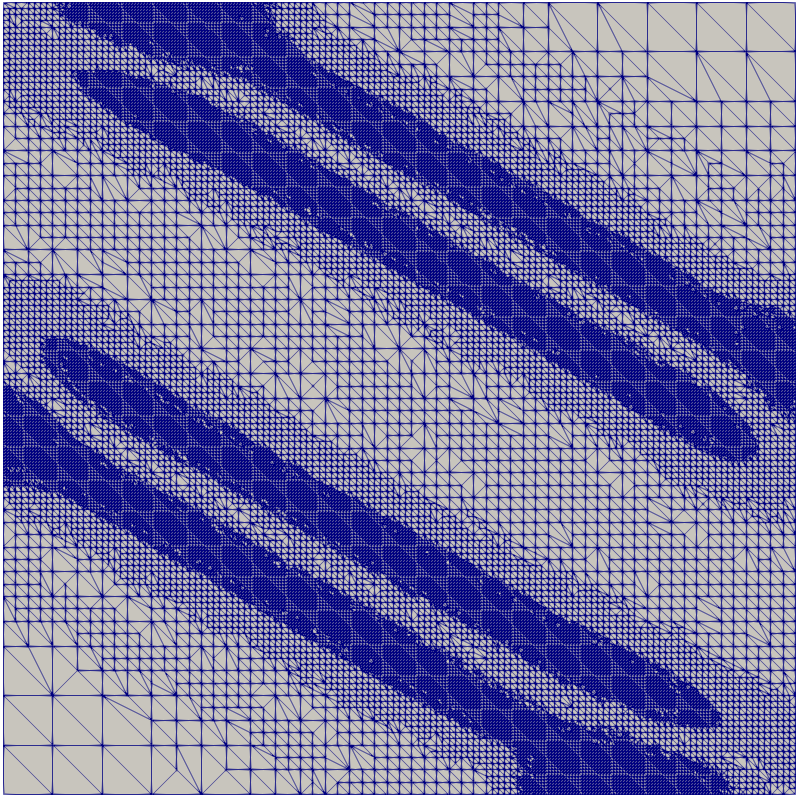}
    \includegraphics[scale=0.112]{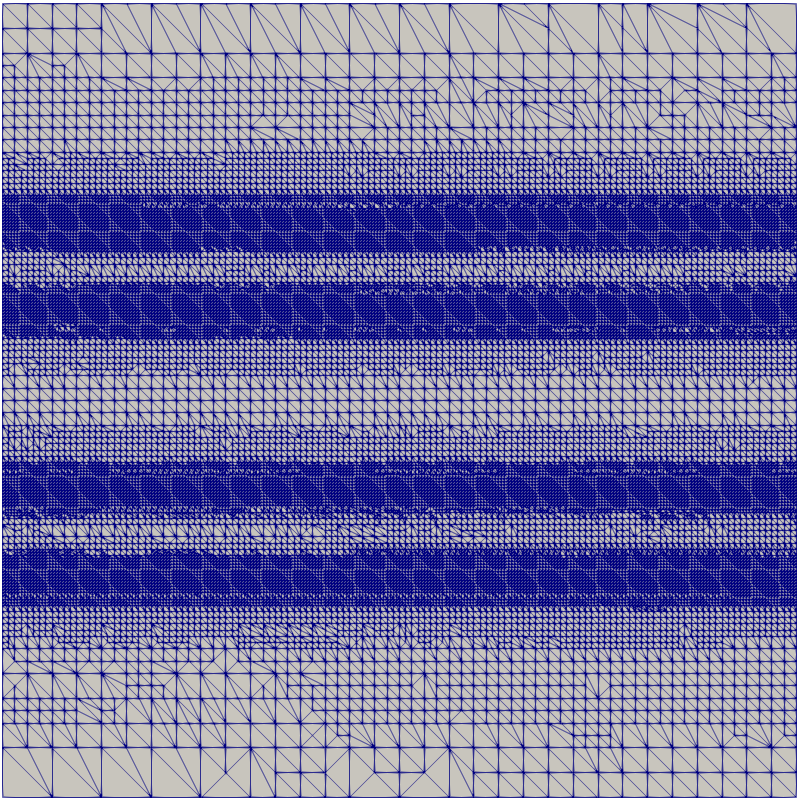}
    \caption{Example 5, visualization of the adaptive mesh refinements 
      at the $25$th adaptive refining step without constraint (top) 
      and the $28$th step with constraints (bottom), in the space-time 
      domain at $t=0,\;0.25,\;0.5,\;0.75$ (from left to
      right)} \label{fig:adaptiveturningwaves} 
\end{figure} 

\section{Conclusions}
\label{sec:con}
In this work, we have considered unstructured space-time finite element 
methods for the optimal control of linear and semilinear parabolic equations, 
without or with box constraints imposed on the control. 
We
  have shown stability of the continuous and 
  the discrete optimality system (with linear state equations and without
  control constraints), and derived error estimates. Our numerical results 
  confirm the theorems and show optimal convergence rates of the space-time
  finite element approximations. Further, our methods are applicable to more
  complicated optimal control problems with semilinear state 
  equations and box constraints. This is also confirmed by our numerical
  experiments, using the Lagrange-Newton method.
We use adaptivity based on a residual error indicator to reduce the complexity.
The rigorous analysis of adaptive space-time procedures is certainly 
a challenging task of future research work.
The linear system respectively the linearized systems of finite element equations 
are solved by an algebraic multigrid preconditioned GMRES method.
This GMRES works fine in practice, but a rigorous convergence analysis is still missing.
The development of parallel solvers will 
certainly make this space-time 
approach an efficient alternative to time-stepping methods 
which are sequential in time.


\bibliography{LSTY}
\bibliographystyle{abbrv}

\end{document}